\documentclass[a4paper, noindent,12pt]{article}
\usepackage{enumerate}
\usepackage{mdframed}
\usepackage{ulem}
\usepackage{setspace}
\usepackage{ucs}
\usepackage[utf8]{inputenc}
\usepackage{subcaption}
\usepackage{fullpage}
\usepackage{color}
\usepackage{graphicx}
\usepackage{wrapfig}
\usepackage{todonotes}
\usepackage[centertags]{amsmath}
\usepackage{esint}
\usepackage{amsfonts,amssymb,amsthm} 
\newtheorem{theorem}{Theorem}[section]
\newtheorem{definition}[theorem]{Definition}
\newtheorem{proposition}[theorem]{Proposition}
\newtheorem{remark}[theorem]{Remark}
\newtheorem{lemma}[theorem]{Lemma}
\newtheorem{corollary}[theorem]{Corollary}

\newtheorem{example}[theorem]{Example}
\newtheorem{problem}{Problem}
\numberwithin{equation}{section}

\newcommand{\Z}{\mathbb Z}
\newcommand{\R}{\mathbb R}
\newcommand{\N}{\mathbb N}

\newcommand{\dist}{\operatorname{dist}}
\newcommand{\diam}{\operatorname{diam}}
\newcommand{\supp}{\operatorname{supp}}
\newcommand*\dd{d}
\newcommand{\e}{\varepsilon}
\newcommand{\ext}[1]{\overline{#1}}

\def\aa{a}
\newcommand{\prob}{{\mathbb P}}
\newcommand{\expec}[1]{\big\langle #1 \big\rangle}
\def\loc{\operatorname{loc}}
\def\wto{\rightharpoonup}

\def\sym{\operatorname{sym}}

\newcommand{\step}[1]{\medskip

{\it Step #1.\ }}

\newcommand{\substep}[1]{\smallskip

{\it Supstep #1.\ }}

\begin{document}
\begin{center}
  {\large Lecture Notes}\\[1cm]
  {\large\bf An introduction to the qualitative and quantitative theory of homogenization}\\[1cm]
  \today\\[.5cm]

  {Stefan Neukamm\footnote{\texttt{stefan.neukamm@tu-dresden.de}}}\\
  {Department of Mathematics, Technische Universit\"at Dresden}

\medskip
\end{center}

\bigskip

\begin{abstract}
We present an introduction to periodic and stochastic homogenization of elliptic partial differential equations. The first part is concerned with the qualitative theory, which we present for equations with periodic and random coefficients in a unified  approach based on Tartar's method of oscillating test functions. In particular, we present a self-contained and elementary argument for the construction of the sublinear corrector of stochastic homogenization. (The argument also applies to elliptic systems and in particular to linear elasticity). In the second part we briefly discuss the representation of the homogenization error by means of a two-scale expansion. In the last part we discuss some results of quantitative stochastic homogenization in a discrete setting. In particular, we discuss the quantification of ergodicity via concentration inequalities, and we illustrate that the latter in combination with elliptic regularity theory leads to a quantification of the growth of the sublinear corrector and the homogenization error.
\medskip

\end{abstract}

\paragraph{Preface and Acknowledgments}
The present notes originate from a one week mini-course given by the author during the GSIS International Winter School 2017 on ``Stochastic Homogenization and its Applications'' at the  Tohoku University, Sendai, Japan. The author would like to thank the organizers of that workshop, especially Reika Fukuizumi, Jun Masamune and Shigeru Sakaguchi for their very kind hospitality. The present notes are devoted to graduate students and young researchers with a basic knowledge in PDE theory and functional analysis. The first three chapters are rather self-contained and offer an introduction to the basic theory of periodic homogenization and its extension to homogenization of elliptic operators with random coefficients. The last chapter, which is in parts based on an extended preprint to the paper \cite{GNO1} by Antoine Gloria, Felix Otto and the author, is a bit more advanced, since it invokes some input from elliptic regularity theory (in a discrete setting) that we do not develop in this manuscript. The author would like to thank Mathias Sch\"affner and Helmer Hoppe for proofreading the original manuscript, and Andreas Kunze for providing the illustrations and numerical results, which were obtained in his master thesis \cite{Kunze}. The author was supported by the DFG in the context of TU Dresden's Institutional
Strategy ``The Synergetic University''.
\newpage

\tableofcontents
\newpage

\section{Introduction -- a one-dimensional example}
Consider a heat conducting body that occupies some domain $O\subset\R^d$,  where $d=1,2,\ldots$ denotes the dimension. Suppose that the body is exposed to a heat source/sink that does not vary in time, and suppose that the body is cooled at its boundary, such that its temperatur is zero at the boundary. If time evolves the temperatur of the body will converge to a steady state, which can be described by the elliptic boundary value problem
\begin{align*}
  -\nabla\cdot(a\nabla u)&=f\qquad\text{in }O,\\
  u&=0\qquad\text{on }\partial O.
\end{align*}
In this equation
\begin{itemize}
\item $u:O\to\R$ denotes the (sought for) temperatur field,
\item $f:O\to\R$ is given and describes the heat source.
\end{itemize}
The ability of the material to conduct heat is described by a material parameter $a\in(0,\infty)$, called the \textit{conductivity}. The material is \textit{homogeneous}, if $a$ does not depend on $x$. The material is called \textit{heterogeneous}, if $a(x)$ varies in $x\in O$. In this lecture we are interested in heterogeneous materials with \textit{microstructure}, which means that the heterogeneity varies on a length scale, called the \textit{microscale}, that is much smaller than a \textit{macroscopic} length scale of the problem, e.g. the diameter of the domain $O$ or the length scale of the right-hand side $f$.
\medskip

To fix ideas, suppose that $a(x)=a_0(\frac{x}{\ell})$ with $a_0$ 
periodic, i.e.\ the conductivity is periodic with the period $\ell$. If the ratio 
\begin{equation*}
  \e:=\tfrac{\text{microscale}}{\text{macroscale}}=\tfrac{\ell}{L}
\end{equation*}
is a small number, e.g. $\e\lesssim 10^{-3}$, then we are in the regime of a microstructured material. The goal of homogenization is to derive a simplified PDE by studying the limit $\e\downarrow 0$, i.e.\ when the micro- and macroscale separate. 
\medskip

In the rest of the introduction we treat the following one-dimensional example: Let $O=(0,L)\subset\R$, $\e>0$ and let $u_\e:O\to \R$ be a solution to the equation
\begin{align}
  -\partial_x\left(a\left(\tfrac{x}{\e}\right)\partial_x u_\e(x)\right)&=f\qquad\text{in }O, \label{Int:Eq1} \\
  u_\e&=0\qquad\text{on }\partial O. \label{Int:Eq2}
\end{align}
We suppose that $a:\R\to\R$ is $1$-periodic and uniformly elliptic, i.e.\ there exists $\lambda>0$ such that $a(x)\in(\lambda, 1)$ for all $x\in\R$. For simplicity, we assume that $f$ and $a$ are smooth.
\smallskip

\begin{figure}[h]
  \begin{tabular}[t]{p{.33\linewidth} p{.33\linewidth} p{.33\linewidth}}
    \begin{subfigure}[t]{.9\linewidth}
      \centering
      \caption{}
      \includegraphics[width=1\linewidth]{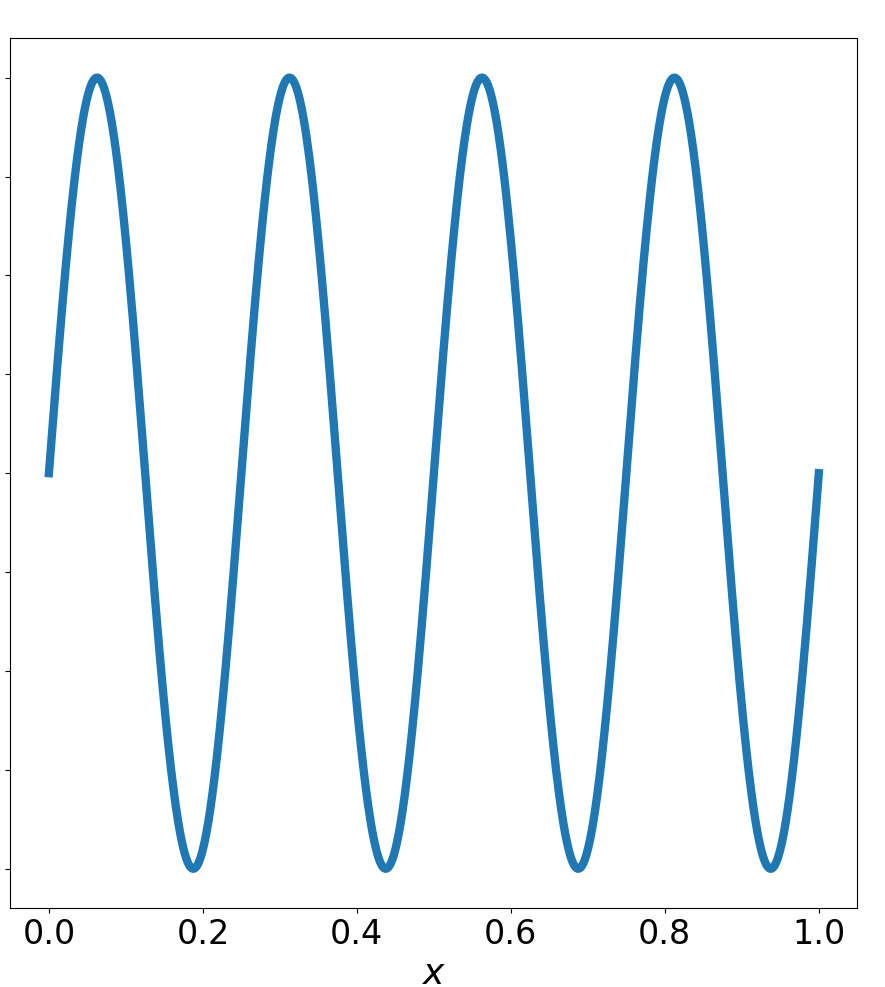}
      \label{fig.1d1}
    \end{subfigure}
    & 
    \begin{subfigure}[t]{.9\linewidth}
      \centering
      \caption{}
      \includegraphics[width=1\linewidth]{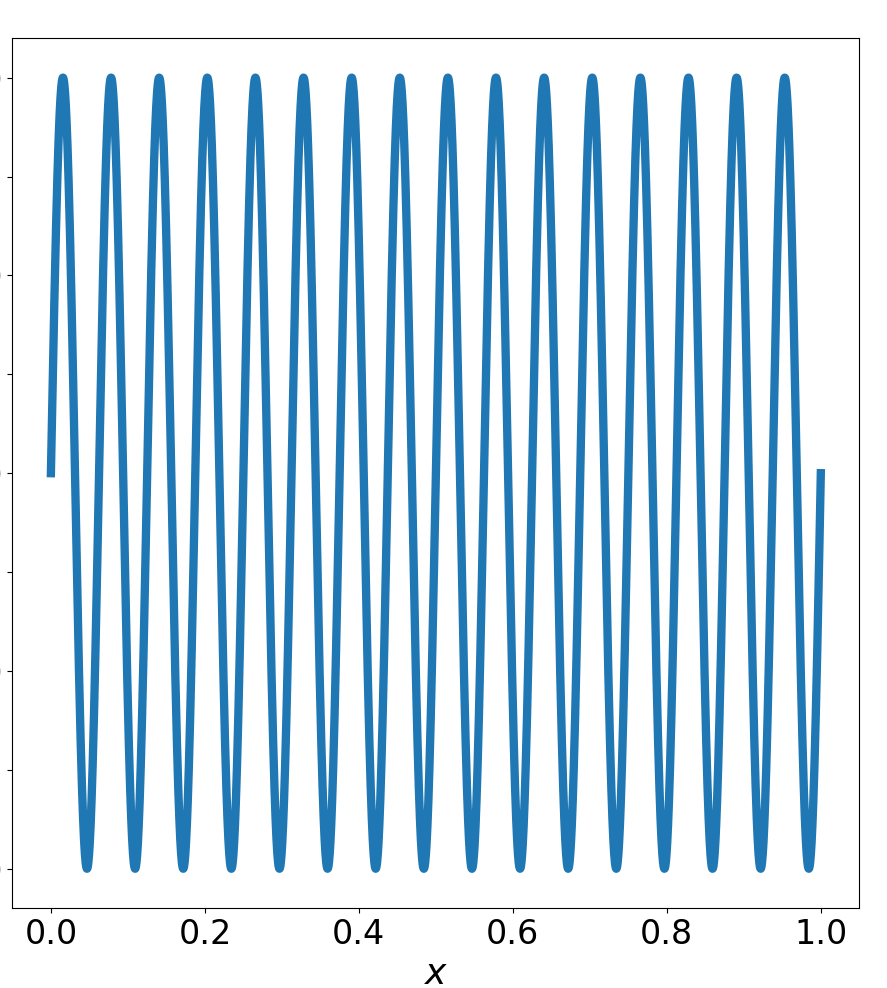}
      \label{fig.1d2}
    \end{subfigure}
    & 
    \begin{subfigure}[t]{.9\linewidth}
      \centering
      \caption{}
      \includegraphics[width=1\linewidth]{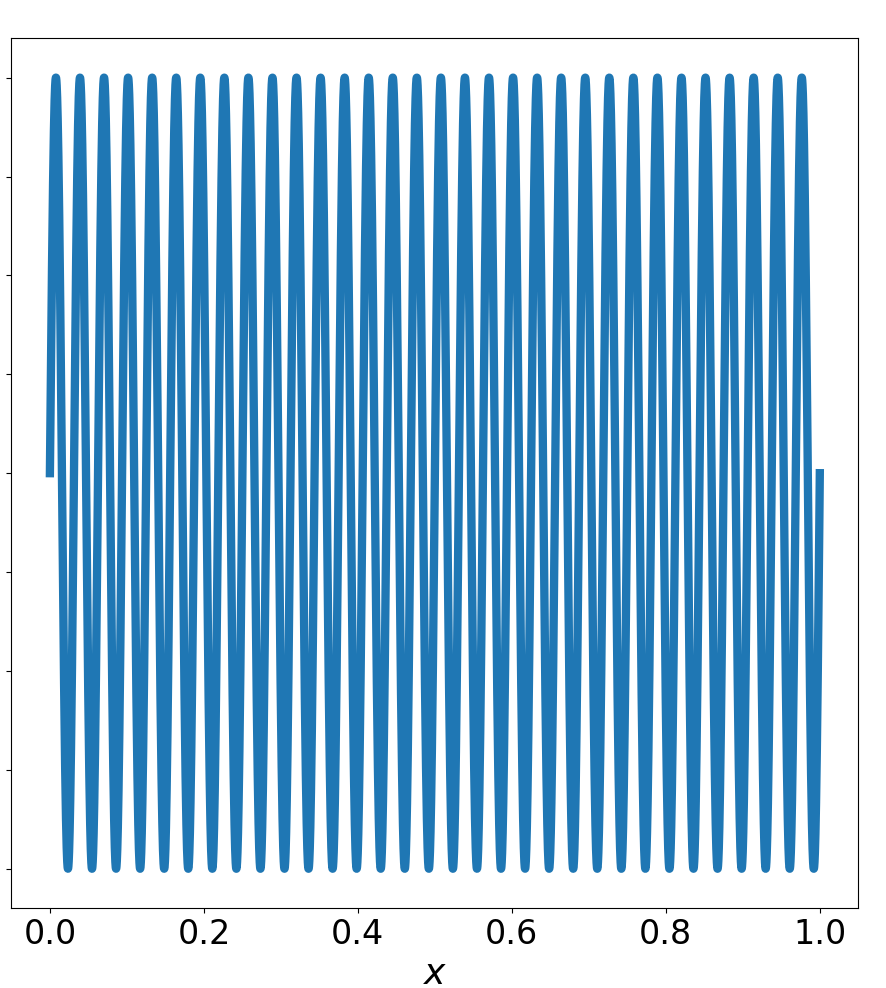}
      \label{fig.1d3}
    \end{subfigure}\\
    \begin{subfigure}[t]{.9\linewidth}
      \centering
      \caption{}
      \includegraphics[width=1\linewidth]{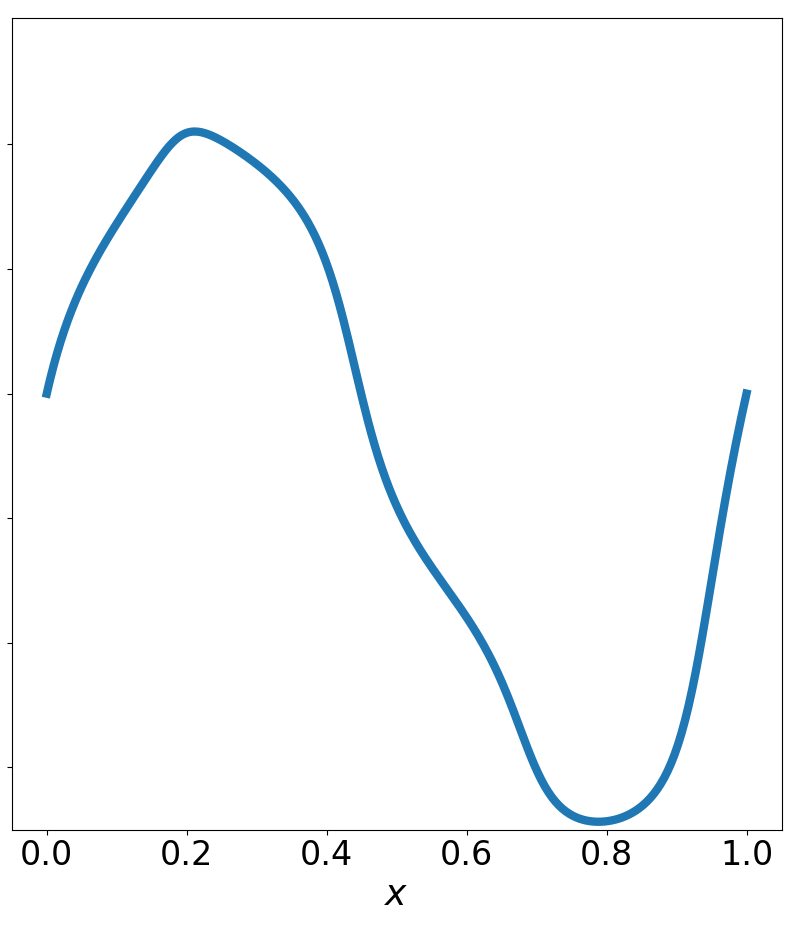}
      \label{fig.1d4}
    \end{subfigure}
    & 
    \begin{subfigure}[t]{.9\linewidth}
      \centering
      \caption{}
      \includegraphics[width=1\linewidth]{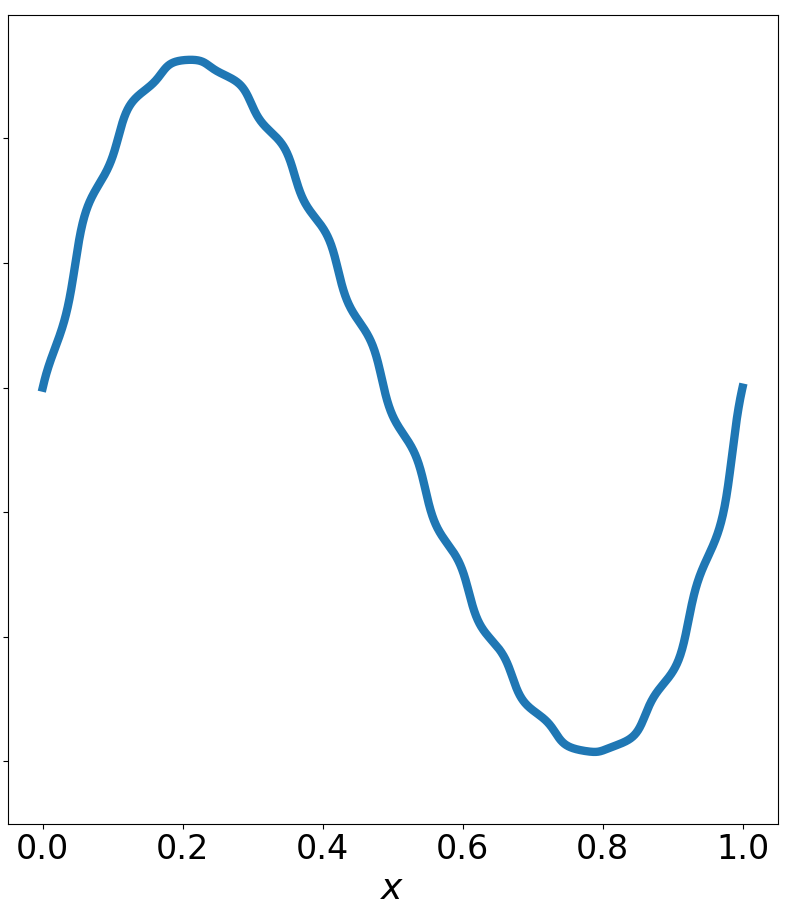}
      \label{fig.1d5}
    \end{subfigure}
    & 
    \begin{subfigure}[t]{.9\linewidth}
      \centering
      \caption{}
      \includegraphics[width=1\linewidth]{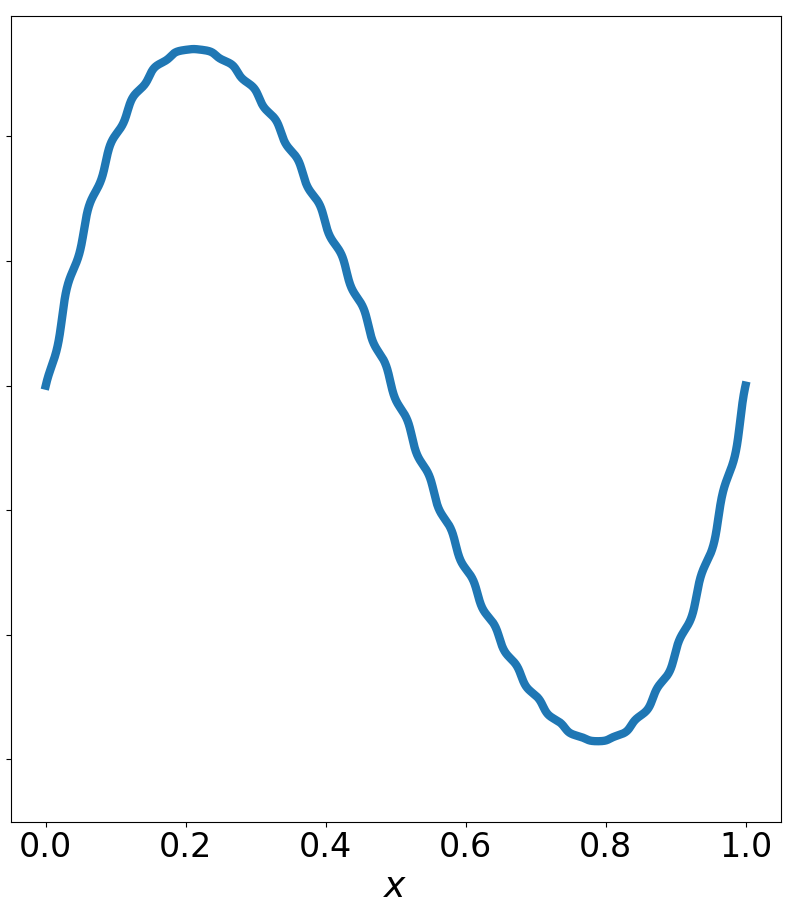}
      \label{fig.1d6}
    \end{subfigure}\\
  \end{tabular}
  \caption{\ref{fig.1d1} -- \ref{fig.1d3} show the rapidly oscillating coefficient field $a(\frac{x}{\e})=2+\sin(2\pi\frac{x}{\e})$ for $\e\in\{\frac14,\frac{1}{16},\frac{1}{32}\}$. \ref{fig.1d4} -- \ref{fig.1d6} show the solutions to to \eqref{Int:Eq1} and \eqref{Int:Eq2} with $f(x)=-3(2x-1)$.}
  \label{fig.1d}
\end{figure}

We are going to prove the following homogenization result:
\begin{itemize}
\item For all $\e>0$ equations \eqref{Int:Eq1}, \eqref{Int:Eq2} admit a unique smooth solution $u_\e$.
\item As $\e\downarrow 0$, $u_\e$ converges to a smooth function $u_0$.
\item The limit $u_0$ is the unique solution to the equation
\begin{align}
  -\partial_x(a_0\partial_x u_0)&=f\qquad\text{in }O, \label{Int:Eq3}\\
  u_0&=0\qquad\text{on }\partial O, \label{Int:Eq4}
\end{align}
where $a_0\in\R$ denotes the harmonic mean of $a$, i.e.
\begin{equation*}
  a_0=\left(\int_0^1a^{-1}(y)\,dy\right)^{-1}.
\end{equation*}
\end{itemize}
\begin{problem}\label{Int:P1}
  Show that \eqref{Int:Eq1} and \eqref{Int:Eq2} admit a unique, smooth solution. 
\end{problem}
The solution to this and all subsequent problems in this introduction can be found in Appendix~\ref{appendix-solution}. We have an explicit presentation for the solution:
\begin{equation}\label{Int:Eq5}
  u_\e(x)=\int_0^xa^{-1}_\e\left(x'\right)\left(c_\e-\int_0^{x'}f\left(x''\right)\,\dd x''\right)\,\dd x'.
\end{equation}
where
\begin{equation*}
  c_\e=\left(\int_0^La^{-1}_\e\left(x'\right)\,\dd x'\right)^{-1}\int_0^L\int_0^{x'}a^{-1}_\e\left(x'\right)f\left(x''\right)\,\dd x''\,\dd x'.
\end{equation*}
In order to pass to the limit $\e\downarrow 0$ in the representation \eqref{Int:Eq5}, we need to understand the limit of functions of the form
\begin{equation*}
  x\mapsto \frac{1}{a\left(\tfrac{x}{\e}\right)}\int_0^xf\left(x'\right)\,\dd x'.
\end{equation*}
This function rapidly oscillates on scale $\e$ and the amplitude of the oscillations is of unit order. Hence, the expression does not converge uniformly (or in a strong sense).
Nevertheless, we have the following result:

\begin{lemma}\label{Int:P2}
Let $F(y,x)$ be a smooth function that is periodic in $y\in\R$ and assume that $F$ and $\partial_xF$ are bounded. Show that
\begin{equation*}
  \lim\limits_{\e\downarrow 0}\int_a^b F\left(\tfrac{x}{\e},x\right)\,\dd x=\int_a^b \bar F(x)\,\dd x,\qquad \bar F(x)=\int_0^1F(y,x)\,\dd y.
\end{equation*}
Furthermore, show that there exists a constant $C$ (only depending on $F$) such that
\begin{equation*}
  |\int_a^b F(\tfrac{x}{\e},x)-\bar F(x)\,\dd x|\leq C(|b-a|+1)\e.
\end{equation*}
\end{lemma}
\begin{proof}
Consider the functions
\begin{equation*}
  G(y,x)=\int_0^y(F(y',x)-\bar F(x))\,\dd y',\qquad g_\e(x):=\e G\left(\tfrac{x}{\e},x\right).
\end{equation*}
Note that $G(y,x)$ and $\partial_xG(y,x)$ are periodic in $y$; indeed, we have
\begin{equation*}
  G(y+1,x)-G(y,x)=\int_y^{y+1}F(y',x)-\bar F(x)\,\dd y'=\bar F(x)-\bar F(x)=0,
\end{equation*}
and the same is true for $\partial_x G$. Furthermore, $G$ and $\partial_xG$ are smooth and bounded, and we have
\begin{equation*}
  \partial_xg_\e(x)=\e\partial_x G\left(\tfrac{x}{\e},x\right)+\partial_y G\left(\tfrac{x}{\e},x\right)= \e\partial_x G\left(\tfrac{x}{\e},x\right)+ \left(F\left(\tfrac{x}{\e},x\right)-\bar F(x)\right),
\end{equation*}
and thus
\begin{align*}
  \int_a^b F\left(\tfrac{x}{\e},x\right)-\bar F(x)\,\dd x&=\int_a^b \partial_x g_\e(x)-\e\partial_x G\left(x,\tfrac{x}{\e}\right)\,\dd x\\
  &=\e \left(G\left(\tfrac{b}{\e},b\right)-G\left(\tfrac{a}{\e},a\right)-\int_a^b\partial_x G\left(\tfrac{x}{\e},x\right)\,\dd x\right).
\end{align*}
The expression in the brackets is bounded uniformly in $\e$ (by smoothness and periodicity of $G$ and $\partial_x G$), and thus the statement follows.
\end{proof}

\begin{problem}\label{Int:P3}
  Show that $\max_{x\in O}|u_\e(x)-u_0(x)|\leq C\e$ where $C$ only depends on $O$, $f$ and $a$.
\end{problem}
\bigskip

The physical interpretation of the result of Lemma~\ref{Int:P3} is the following: While the initial problem \eqref{Int:Eq1} \& \eqref{Int:Eq2} describes a heterogeneous, microstructured material (a periodic composite with period $\e$), the limiting equation \eqref{Int:Eq3} \& \eqref{Int:Eq4} describes a homogeneous material with conductivity $a_0$. Hence, Problem~\ref{Int:P2} states that if we observe a material with a rapidly  oscillating conductivity $a\left(\tfrac{\cdot}{\e}\right)$ on a macroscopic length scale, then it behaves like a homogeneous material with effective conductivity given by $a_0$. We therefore call \eqref{Int:Eq3} \& \eqref{Int:Eq4} the homogenized problem. It is much simpler than the heterogeneous initial problem \eqref{Int:Eq1} \& \eqref{Int:Eq2}:
\begin{problem}\label{Int:P4}
Let $f\equiv 1$. Show that a solution to 
\begin{align*}
  -\partial_x(a\partial_x u)&=1\qquad\text{in }O,\\
  u&=0\qquad\text{on }\partial O.
\end{align*}
is a quadratic function, if and only if the material is homogeneous, i.e.\ iff $a$ does not depend on $x$. 
\end{problem}
The homogenization result shows that $u_\e\to u_0$ as $\e\downarrow 0$. Hence, for $\e\ll 1$ the function $u_0$ is a consistent approximation to the solution to \eqref{Int:Eq1} \& \eqref{Int:Eq2}. We even have a rate: $u_\e=u_0+O(\e)$. Thanks to the homogenization result certain properties of the difficult equation \eqref{Int:Eq1} \& \eqref{Int:Eq2} can be studied by analyzing the simpler problem \eqref{Int:Eq3} \& \eqref{Int:Eq4}:
\begin{problem}\label{Int:P5}
  Let $f\equiv 1$ and $O=(0,1)$. Show that $M_\e:=\max_{\bar O}u_\e=\tfrac{1}{8a_0}+O(\e)$.
\end{problem}
What can be said about the convergence of the gradient $\partial_xu_\e$?
\begin{problem}\label{Int:P6}
Show that $\limsup\int_O|\partial_x u_\e-\partial_x u_0|^2>0$ (unless the initial material is homogeneous). Show on the other hand, that for all smooth functions $\varphi:\R\to\R$ we have
\begin{equation*}
  \int_O u'_\e(x)\varphi(x)\,\dd x\to\int_O u'_0(x)\varphi(x)\,\dd x
\end{equation*}
i.e.\ we have weak convergence, but not strong convergence.
\end{problem}
Yet, we can modify $u_0$ by adding oscillations, such that the gradient of the modified functions converges:
\begin{lemma}[Two-scale expansion]\label{Int:L2}
Let $a,f$ be smooth, $O=(0,1)$. Let $\phi:\R\to\R$ denote a $1$-periodic solution to
\begin{equation}\label{corr:1dim}
  \partial_y(a(y)(\partial_y \phi(y)+1))=0
\end{equation}
with $\phi(0)=0$. Let $u_0$ and $u_\e$ be as above. Consider
\begin{equation*}
  v_\e(x):=u_0(x) + \e\phi\left(\tfrac{x}{\e}\right)\partial_x u_0(x).
\end{equation*}
Then there exists a constant $C>0$ such that for all $\e>0$ with $\frac{1}{\e}\in\N$ we have
\begin{equation*}
  \int_{O}|u_\e-v_\e|^2+|\partial_xu_\e-\partial_x v_\e|^2\leq (\frac{4}{\lambda}\max|\phi|^2)\e^2\int_{O}|\partial^2_x u_0|^2.
\end{equation*}
\end{lemma}
\begin{proof}
To ease notation we write
\begin{equation*}
  a_\e(x):=a\left(\tfrac{x}{\e}\right),\qquad \phi_\e(x):=\phi\left(\tfrac{x}{\e}\right).
\end{equation*}

\step{1}
\smallskip

It can be easily checked (by direct calculations) that
\begin{equation*}
  \phi(y):=\int_0^y\frac{a_0}{a(t)}-1\,\dd t
\end{equation*}
and that $\phi$ is smooth and bounded. Note that
\begin{equation*}
  a_0=a(y)(\partial_y\phi(y)+1)\qquad\text{for all }y\in\R.
\end{equation*}
Indeed, by the corrector equation \eqref{corr:1dim} and the definition of $a_0$ the difference of both functions is constant and has zero mean. (This is only true in the  one-dimensional case!)

\step{2}
\smallskip

Set $z_\e:=u_\e-v_\e$. Since $\frac{1}{\e}\in\N$ we have $\phi\left(\frac{1}{\e}\right)=0$. Combined with the boundary conditions imposed on $u_\e$ and $\phi_\e$ we conclude that $z_\e(0)=z_\e(1)=0$. We claim that
\begin{equation*}
  \int_{O}|z_\e|^2\leq \int_O|\partial_x z_\e|^2.
\end{equation*}
Indeed, since $O=(0,1)$ and $z_\e=0$ on $\partial O$, this follows by Poincar\'e's inequality:
\begin{equation*}
  \int_0^1|z_\e|^2=\int_0^1\left(\int_0^x\partial_x z_\e\right)^2\leq \int_0^1|\partial_x z_\e|^2.
\end{equation*}
Hence,
\begin{equation*}
  \int_O |z_\e|^2+|\partial_x z_\e|^2\leq 2\int_O|\partial_x z_\e|^2\leq \tfrac{2}{\lambda}\int_O|\partial_x z_\e|^2a_\e,
\end{equation*}
where we used that $a_\e\geq\lambda$ by assumption. Since $z_\e=0$ on $\partial O$, we may integrate by parts and get
\begin{equation*}
  \int_O |z_\e|^2+|\partial_x z_\e|^2\leq \tfrac{2}{\lambda}\int_Oz_\e(-\partial_x(a_\e\partial_x z_\e)).
\end{equation*}

\smallskip
\step{3} We compute $(-\partial_x(a_\e\partial_x z_\e))$:
\begin{align*}
  \partial_x z_\e&=\partial_x u_\e-\left(\partial_y\phi\left(\tfrac{x}{\e}\right)+1\right)\partial_x u_0-\e\phi_\e\partial_x^2 u_0\\
  &\qquad \text{use }a_0=a_\e(\partial_y\phi(\tfrac{\cdot}{\e})+1)\\
  a_\e\partial_x z_\e&=a_\e\partial_x u_\e - a_0\partial_x u_0-\e a_\e\phi_\e\partial_x^2u_0\\
  -\partial_x(a_\e\partial_x z_\e)&=-\partial_x(a_\e\partial_x u_\e)+\partial_x(a_0\partial_x u_0)+\partial_x(\e a_\e\phi_\e\partial_x^2 u_0).
\end{align*}
The first two terms on the right-hand side are equal to the left-hand side of the PDEs for $u_\e$ and $u_0$. Hence, these two terms evaluate to $f-f=0$:
\begin{equation*}
  -\partial_x(a_\e\partial_x z_\e)=\partial_x \left(\e a_\e\phi_\e\partial_x^2 u_0\right).
\end{equation*}
Combined with the estimate of Step 2 we deduce that
\begin{eqnarray*}
  \int_O |z_\e|^2+|\partial_x z_\e|^2&\leq& \tfrac{2}{\lambda}\int_Oz_\e\partial_x \left(\e a_\e\phi_\e\partial_x^2 u_0\right)\\
  &&\qquad\text{integration by parts}\\
  &=&
  \int_O\partial_xz_\e\left(\e\phi_\e a_\e\partial_x^2 u_0\right)\\
  &&\qquad\text{Cauchy-Schwarz and Young's inequality}\\
  &&\qquad\text{in the form $ab\leq \frac{\delta}{2}a^2+\frac{1}{2\delta}b^2$ with $\delta=\frac{\lambda}{2}$}\\
  &\leq&
  \tfrac{1}{2}\int_O|\partial_xz_\e|^2+\tfrac{2}{\lambda^2}\e^2\int_O|\phi_\e|^2|a_\e|^2|\partial_x^2 u_0|^2,
\end{eqnarray*}
and thus
\begin{equation*}
  \int_O |z_\e|^2+|\partial_x z_\e|^2\leq \frac{4}{\lambda^2}\e^2\int_O|\phi_\e|^2|\partial_x^2 u_0|^2.
\end{equation*}
\end{proof}
%
%
In this lecture we extend the previous one-dimensional results to
\begin{itemize}
\item higher dimensions -- the argument presented above heavily relies on the fact that we have an explicit representation for the solutions. In higher dimensions such a representation is not available and the argument will be more involved. In particular, we require some input from the theory of \textit{partial differential equations} and \textit{functional analysis} such as the notion of distributional solutions, the existence theory for elliptic equations in divergence form in Sobolev spaces, the Theorem of Lax-Milgram, Poincar\'e's inequality, the notion of weak convergence in $L^2$-spaces, and the Theorem of Rellich-Kondrachov, e.g. see the textbook on functional analysis by Brezis \cite{Brezis}. 
\item periodic and random coefficients -- to treat the later we require some input from \textit{ergodic \& probability theory}.
\end{itemize}
Moreover, we discuss
\begin{itemize}
\item the two-scale expansion in higher dimension and in the stochastic case, and explain
\item quantitative results for stochastic homogenization in a discrete setting.
\end{itemize}
%


\section{Qualitative homogenization of elliptic equations}

In this section we discuss the homogenization theory for elliptic operators of the form $-\nabla\cdot (a\nabla)$ with uniformly elliptic coefficients.  We say that $a:\R^d\to\R^{d\times d}$ is {\bf uniformly elliptic} with ellipticity constant $\lambda>0$, and write $a\in M(\R^d,\lambda)$, if $a$ is measurable, and for a.e.\ $x\in\R^d$ we have
\begin{equation}\label{Hom:Eq1}
  \forall \xi\in\R^d\,:\qquad \xi\cdot a(x)\xi\geq\lambda|\xi|^2\text{ and }|a(x)\xi|\leq |\xi|.
\end{equation} 
A standard result (that invokes the Lax-Milgram Theorem) yields existence of weak solutions to the associated elliptic boundary value problem.
\begin{problem}\label{P:apriori}
  Let $a\in M(\R^d,\lambda)$, $O\subset\R^d$ open and bounded, $f\in L^2(O)$, $F\in L^2(O,\R^d)$. Show that there exists a unique solution $u\in H^1_0(O)$ to the equation
  \begin{equation}\label{T1:elliptic}
    -\nabla\cdot(a\nabla u)=f-\nabla\cdot F\qquad\text{in }\mathcal D'(O).
  \end{equation}
  It satisfies the \textit{a priori estimate}
  \begin{equation}\label{T1:apriori}
    \|u\|_{H^1(O)}\leq C(\lambda,d,\operatorname{diam}(O))\left(\|f\|_{L^2(O)}+\|F\|_{L^2(O)}\right).
  \end{equation}
\end{problem}
In this section we study a classical problem of elliptic homogenization: Given a family of coefficient fields $(a_\e)\subset M(\R^d,\lambda)$, consider the weak solution $u_\e\in H^1_0(O)$ to the equation $-\nabla\cdot (a_\e\nabla u_\e)=f-\nabla\cdot F$ in $\mathcal D'(O)$. A prototypical homogenization result states that under appropriate conditions on $(a_\e)$,
\begin{itemize}
\item $u_\e$ weakly converges to a limit $u_0$ in $H^1_0(O)$ as $\e\downarrow 0$.
\item The limit $u_0$ can be characterized as the unique weak solution in $H^1_0(O)$ to a homogenized equation $-\nabla\cdot(a_{\hom}\nabla u_0)=f-\nabla\cdot F$.
\item The homogenized coefficient field $a_{\hom}$ can be computed from $(a_\e)$ by a homogenization formula.
\end{itemize}
We discuss two types of structural conditions on the coefficient fields $(a_\e)$ that allow to prove such a result. In the first case, which usually is referred to as \textit{periodic homogenization}, the coefficient fields are assumed to be periodic, i.e.\ $a_\e(\cdot)=a_0(\frac{\cdot}{\e})$, where $a_0$ is periodic in the following sense:
\begin{definition}
  We call a measurable function $f$ defined on $\R^d$ {\bf $L$-periodic}, if for all $z\in\Z^d$ we have
  \begin{equation*}
    f(\cdot+Lz)=a(\cdot)\qquad\text{a.e.\ in }\R^d.
  \end{equation*}
\end{definition}
In the second case, called \textit{stochastic homogenization},  the coefficient fields are supposed to be stationary and ergodic random coefficients. We discuss the stochastic case in more detail in Section~\ref{S:stochastic}.
\medskip

Both cases (the periodic and the stochastic case) can be analyzed by a common approach that relies on \textit{Tartar's method of oscillating test function}, see \cite{MT97}. In the following we present the approach in the periodic case in a form that easily adapts to the stochastic case.

\subsection{Periodic homogenization}

In this section we prove the following classical and prototypical result of periodic homogenization.
\begin{theorem}[e.g. see textbook Bensoussan, Lions and G. Papanicolaou \cite{Bensoussan}]\label{Hom:T1}
  Let $\lambda>0$ and $a\in M(\R^d,\lambda)$ be $1$-periodic. Then there exists a constant, uniformly elliptic coefficient matrix $a_{\hom}$ such that:
  \smallskip

  For all $O\subset\R^d$ open and bounded, for all $f\in L^2(O)$ and $F\in L^2(O,\R^{d})$, and $\e>0$, the unique weak solution $u_\e\in H^1_0(O)$ to
  \begin{equation*}
    -\nabla\cdot(a(\tfrac{x}{\e}) \nabla u_\e)=f-\nabla\cdot F\qquad\text{in }\mathcal D'(O)
  \end{equation*}
  weakly converges in $H^1(O)$ to the unique weak solution $u_0\in H^1_0(O)$ to
  \begin{equation*}
    -\nabla\cdot(a_{\hom}\nabla u_0)=f-\nabla\cdot F\qquad\text{in }\mathcal D'(O).
  \end{equation*}
\end{theorem}
A numerical illustration of the theorem is depicted in Figure \ref{fig.periodic}.
\medskip

The main difficulty in the proof of the theorem is to pass to the limit in expressions of the form 
\begin{equation*}
  \int a(\tfrac{x}{\e})\nabla u_\e(x)\cdot \eta(x) e_i\,dx\qquad (\eta\in C^\infty_c(O)),
\end{equation*}
since the integrand is a \textit{product of weakly convergent} terms. 
\begin{figure}[!]
  \bigskip

  \begin{tabular}[t]{p{.4\linewidth} p{.4\linewidth}}
    \begin{subfigure}[t]{.9\linewidth}
      \centering
      \includegraphics[width=1.24\linewidth]{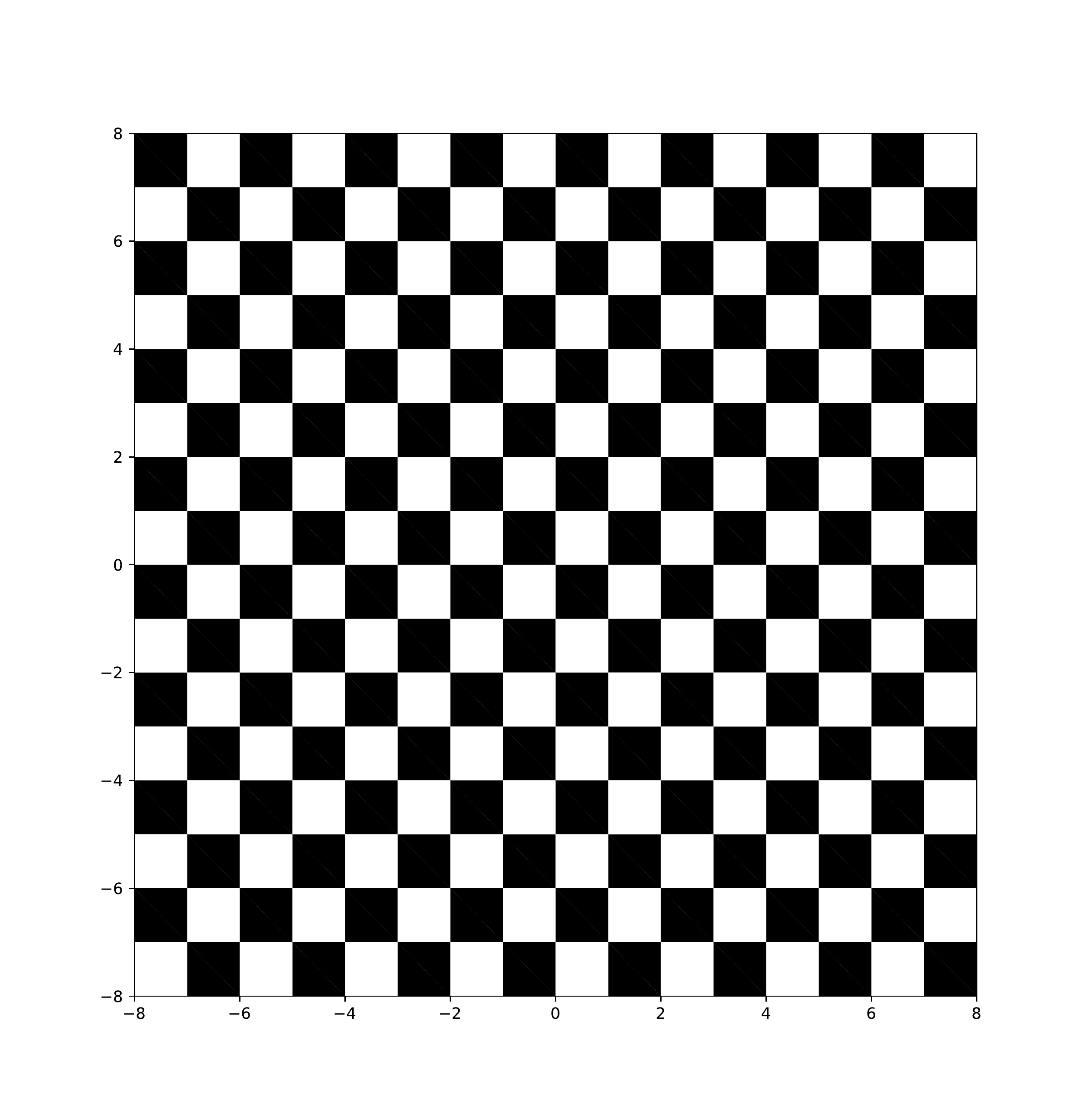}
      \caption{}
      \label{fig.checker1}
    \end{subfigure}
    & 
    \begin{subfigure}[t]{.99\linewidth}
      \centering
      \includegraphics[width=1.1\linewidth]{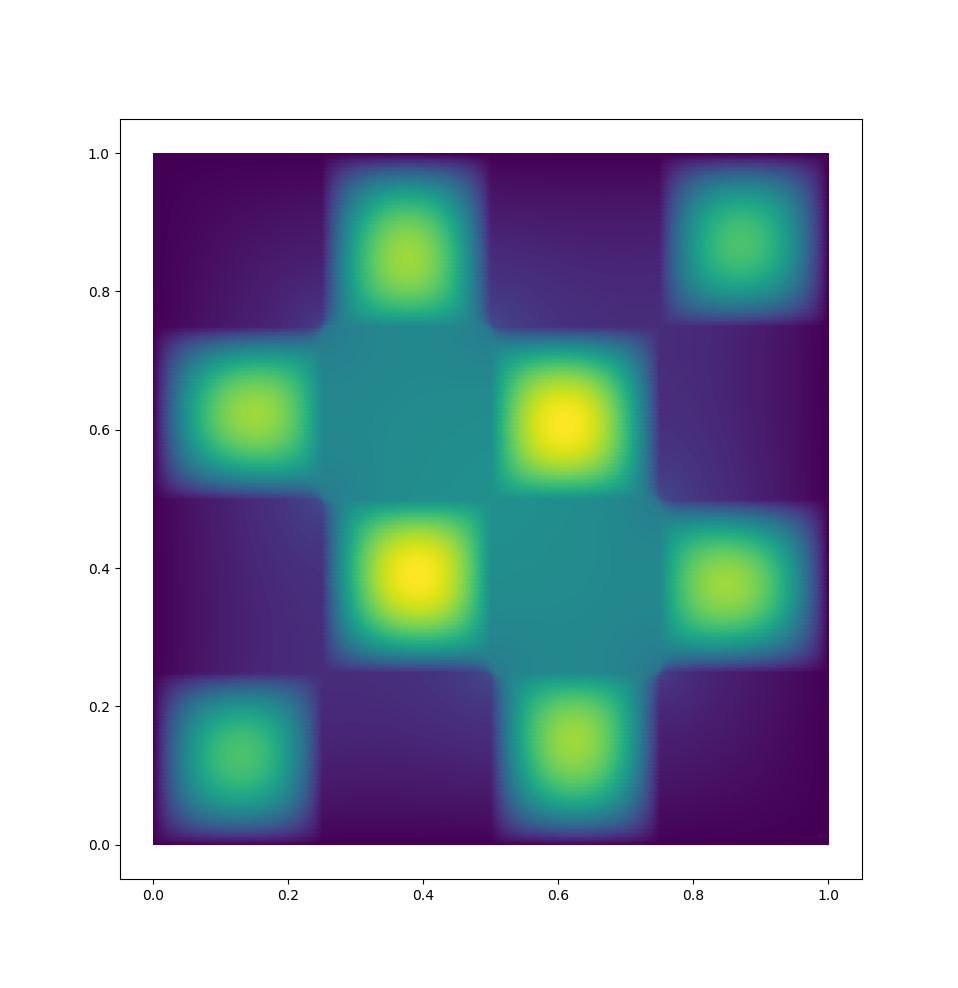}
      \caption{}
      \label{fig.checker2}
    \end{subfigure}\\%
    \begin{subfigure}[t]{.99\linewidth}
      \centering
      \includegraphics[width=1.1\linewidth]{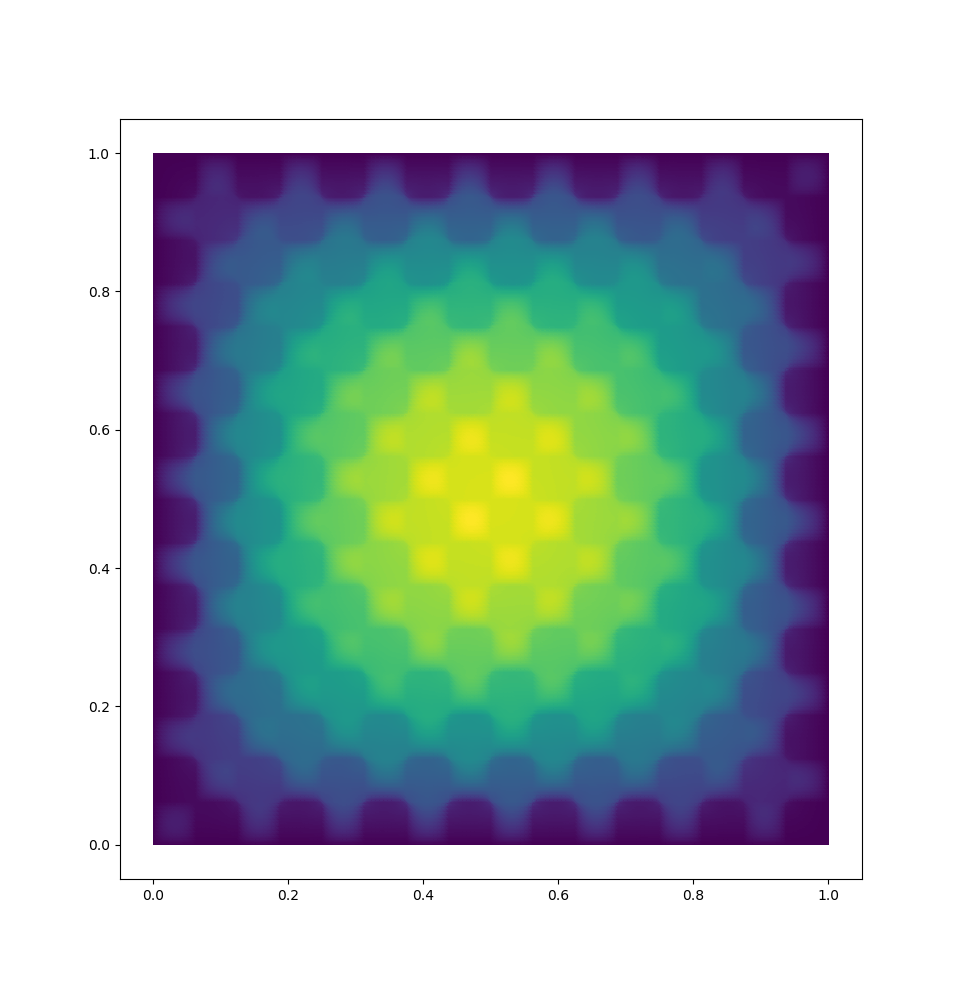}
      \caption{}
      \label{fig.checker3}
    \end{subfigure}
    & 
    \begin{subfigure}[t]{.99\linewidth}
      \centering
      \includegraphics[width=1.1\linewidth]{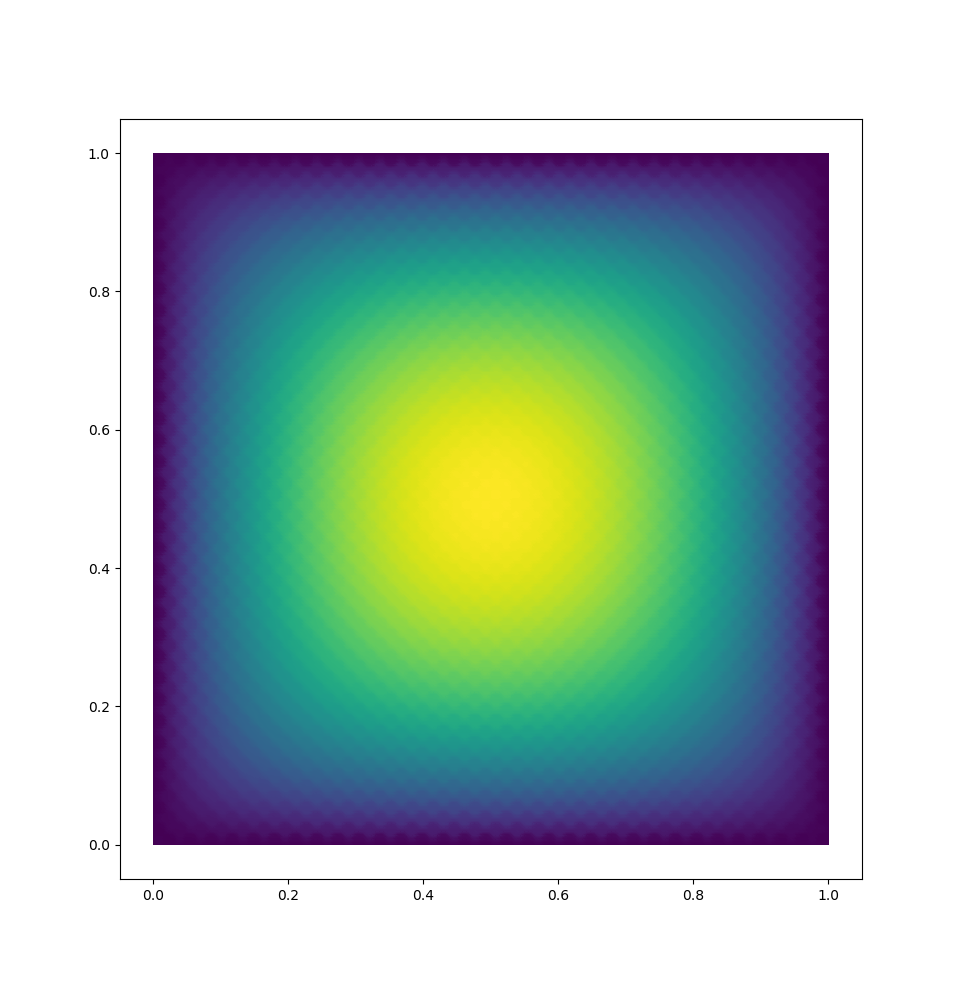}
      \caption{}
      \label{fig.checker4}
    \end{subfigure}
  \end{tabular}
  \caption{\textit{Illustration of Theorem~\ref{Hom:T1} in the periodic, two-dimensional case}. (a) shows a periodic checkerboard-like coefficient field $a(\cdot)$. (b) -- (c) show the solution to the equation $-\nabla\cdot(a(\tfrac{\cdot}{\e})\nabla u_\e)=1$ on the unit cube with homogeneous Dirichlet boundary values for $\e\in\{\frac{1}{2}, \frac{1}{8}, \frac{1}{32}\}$.
}\label{fig.periodic}
\end{figure}
In a nutshell Tartar's method relies on the idea to approximate the test field $\eta(x) e_i$ by some gradient field $\nabla(\eta g_{i,\e})$, where $g_{i,\e}$ denotes an \textit{oscillating test function} with the property that $-\nabla\cdot a^t(\tfrac{\cdot}{\e})\nabla g_{i,\e}\to -\nabla\cdot a^t_{\hom}e_i$ in $H^{-1}(O)$.  We can then pass to the limit by appealing to the following special form of Murat \& Tartar's celebrated div-curl lemma, see \cite{MT97}:
\begin{lemma}\label{L:divcurl}
  Consider $(u_\e)\subset H^1_0(O)$ and $(F_\e)\subset L^2(O,\R^d)$. Suppose that
  \begin{itemize}
  \item  $u_\e\wto u_0$ weakly in $H^1_0(O)$,
  \item  $F_\e\wto F_0$ weakly in $L^2(O,\R^d)$ and  $\int_O F_\e\cdot\nabla\eta_\e\to \int_O F\cdot \nabla\eta$ for any sequence $(\eta_\e)\subset H^1_0(O)$ with $\eta_\e\wto \eta$ weakly in $H^1(O)$.
  \end{itemize}
  Then for any $\eta\in C^\infty_c(O)$ we have
  \begin{equation*}
    \int_O \eta(\nabla u_\e\cdot F_\e)\to \int_O \eta(\nabla u_0\cdot F_0).
  \end{equation*}
\end{lemma}
\begin{proof}
  \begin{eqnarray*}
    \int_O\eta(\nabla u_\e\cdot F_\e)=\int_O\nabla(u_\e\eta)\cdot F_\e-\int_Ou_\e\nabla\eta\cdot F_\e.
  \end{eqnarray*}
  Since $u_\e\eta \wto u_0\eta$ weakly in $H^1_0(O)$, and $u_\e\nabla\eta\to u_0\nabla\eta$ strongly in $L^2(O)$ (by the Rellich-Kondrachov Theorem), we find that the right-hand side converges to 
  \begin{equation*}
    \int_O\nabla(u_0\eta)\cdot F_0-\int_O u_0\nabla\eta\cdot F_0=    \int_O\eta(\nabla u_0\cdot F_0).
  \end{equation*}
\end{proof}

It turns out that the homogenization result holds, whenever we are able to construct an oscillating test function $g_{i,\e}$. This motivates the following definition:
\begin{definition}\label{D:oscillating}
  We say that $(a_\e)\subset M(\R^d,\lambda)$ admits homogenization if there exists an elliptic, constant coefficient matrix $a_{\hom}$, called ``the homogenized  coefficients'', such that the following properties hold: For $i=1,\ldots,d$ there exist oscillating test functions  $(g_{i,\e})\subset H^1_{\loc}(\R^d)$ such that
  \begin{align*}
    \tag{C1}\label{Hom:Eq3}
    -\nabla\cdot a^t_\e\nabla g_{i,\e}&\,=0\qquad\text{in }\mathcal D'(\R^d),\\
    \tag{C2}\label{Hom:Eq4}
    g_{i,\e}&\,\wto x_i\qquad\text{weakly in }H^1_{\loc}(\R^d),\\
    \tag{C3}\label{Hom:Eq5}
    a_\e^t\nabla g_{i,\e}&\,\wto a_{\hom}^te_i\qquad\text{weakly in }L^2_{\loc}(\R^d).
  \end{align*}
\end{definition}
Based on \eqref{Hom:Eq3} -- \eqref{Hom:Eq5} and the div-curl lemma we obtain the following general homogenization result:
\begin{lemma}\label{L:hom-gen}
  Suppose $(a_\e)\subset M(\R^d,\lambda)$ admits homogenization with homogenized coefficients $a_{\hom}$. Then for all $O\subset\R^d$ open and bounded, for all $f\in L^2(O)$ and $F\in L^2(O,\R^{d})$, and $\e>0$, the unique weak solution $u_\e\in H^1_0(O)$ to
  \begin{equation*}
    -\nabla\cdot(a_\e \nabla u_\e)=f-\nabla\cdot F\qquad\text{in }\mathcal D'(O)
  \end{equation*}
  weakly converges in $H^1(O)$ to the unique weak solution $u_0\in H^1_0(O)$ to
  \begin{equation*}
    -\nabla\cdot(a_{\hom}\nabla u_0)=f-\nabla\cdot F\qquad\text{in }\mathcal D'(O).
  \end{equation*}
  Moreover, we have
  \begin{equation*}
    a_\e\nabla u_\e\wto a_{\hom}\nabla u_0\qquad\text{weakly in }L^2(O,\R^d).
  \end{equation*}
\end{lemma}
\begin{proof}
  \step{1} Compactness.
  
  We denote the flux by
  \begin{equation*}
    j_\e:=a_\e\nabla u_\e.
  \end{equation*}
  By the a priori estimates of Problem~\ref{P:apriori} we have
  \begin{equation*}
    \int_O |u_\e|^2+|\nabla u_\e|^2+|j_\e|^2\leq C\int_O |f|^2+|F|^2
  \end{equation*}
  where $C$ does not depend on $\e$. Since bounded sets in $L^2(O,\R^d)$ and $H^1(O)$ are precompact in the weak topology, and since $H^1(O)\Subset L^2_{\loc}(O)$ is compactly embedded (by the Rellich-Kondrachov Theorem), there exist $u_0\in H^1_0(O)$ and $j_0\in L^2(O)$ such that, for a subsequence (that we do not relabel), we have
  \begin{align*}
    &u_\e\wto u_0\qquad\text{weakly in }H^1(O),\\
    &u_\e\to u_0\qquad\text{in }L^2_{loc}(O),\\
    &j_\e\wto j_0\qquad\text{weakly in }L^2(O).
  \end{align*}
  We claim that
  \begin{equation}\label{Hom:Eq6}
    -\nabla\cdot j_0=f-\nabla\cdot F\qquad\text{in }\mathcal D'(O).
  \end{equation}
  Indeed, for all $\varphi\in C^\infty_c(O)$ we have
  \begin{equation*}
    \int j_0\cdot \nabla\varphi\leftarrow \int j_\e\cdot \nabla\varphi=\int a_\e \nabla u_\e\cdot \nabla\varphi=\int f\cdot\varphi+F\cdot \nabla\varphi.
  \end{equation*}

  \step{2}  Identification of $j_0$.
  
  We first argue that it suffices to prove the identity
  \begin{equation}\label{Hom:Eq6b}
    j_0=a_{\hom}\nabla u_0.
  \end{equation}
  Indeed, the combination of \eqref{Hom:Eq6b} and \eqref{Hom:Eq6} shows that 
  \begin{equation*}
    -\nabla\cdot (a_{\hom}\nabla u_0)=f-\nabla\cdot F\qquad\text{in }\mathcal D'(O).
  \end{equation*}
  Since this equation has a unique solution (recall that $a_{\hom}$ is assumed to be elliptic), we deduce that $u_0$ and $j_0$ (which were originally obtained as a weak limits of $(u_\e)$ and $(j_\e)$ along a subsequence), are independent of the subsequence. Hence, we get $u_\e\wto u_0$ weakly in $H^1(O)$ and $j_\e\wto a_{\hom}\nabla u_0$ weakly in $L^2(O,\R^d)$ along the entire sequence, and thus the claimed statement follows.

  It remains to prove \eqref{Hom:Eq6b}. By the fundamental lemma of the calculus of variations, it suffices to show: For all  $\eta\in C^\infty_c(O)$ and $i=1,\ldots,d$ we have
  \begin{equation}\label{Hom:Eq7}
    \int_O \eta(j_0-a_{\hom}\nabla u_0)\cdot e_i=0.
  \end{equation}
  For the argument let $g_{i,\e}$ denote the oscillating test function of Definition~\ref{D:oscillating}, and note that for any sequence $\eta_\e\wto\eta_0$ weakly in $H^1_0(O)$ we have
  \begin{equation*}
    \int_Oj_\e\cdot \nabla\eta_\e=\int_O f\eta_\e+F\cdot\nabla\eta_\e\to \int_Of\eta_0+F\cdot \eta_0=\int_O j\cdot \nabla\eta_0.
  \end{equation*}
  Hence, an application of the div-curl lemma, see Lemma~\ref{L:divcurl}, and property \eqref{Hom:Eq4} yield
  \begin{equation*}
    \int_O \eta(j_\e\cdot\nabla g_{i,\e})\to \int_O\eta(j_0\cdot e_i).
  \end{equation*}
  On the other hand, by \eqref{Hom:Eq3},\eqref{Hom:Eq5} and the convergence of $u_\e$, the div-curl lemma also yields
  \begin{equation*}
    \int_Q\eta(j_\e\cdot\nabla g_{i,\e})=    \int_Q\eta(\nabla u_\e\cdot a^t_\e\nabla g_{i,\e})\to \int_Q\eta(\nabla u_0\cdot a^t_{\hom}e_i),
  \end{equation*}
  and thus \eqref{Hom:Eq7}.
\end{proof}

With Lemma~\ref{L:hom-gen} at hand, the proof of Theorem~\ref{Hom:T1} reduces to the construction of the oscillating test functions $g_{i,\e}$. In the periodic, case the construction is based on the notion of the \textit{periodic corrector}. Before we come to its definition we introduce a Sobolev space of periodic functions: Let $\Box:=(-\frac12,\frac12)^d$ denote the unit box in $\R^d$. For $L>0$ set
\begin{eqnarray*}
  H^1_{\#}(L\Box)&:=&\big\{u\in H^1_{\loc}(\R^d)\,:\,u\text{ is $L$-periodic.}\,\big\}.
\end{eqnarray*}

\begin{problem}
Show that 
\begin{itemize}
\item $H^1_{\#}(L\Box)$ with the inner product of $H^1(L\Box)$ is a Hilbert space (and can be identified with a closed linear subspace of $H^1(L\Box)$). 
\item The space of smooth, $L$-periodic functions on $\R^d$ is dense in $H^1_{\#}(L\Box)$.
\item For any $F\in H^1_\#(L\Box,\R^d)$ we have the integration by parts formula
\begin{equation*}
  \int_{L(z+\Box)}\nabla\cdot F=0\qquad\text{for all }L\in\N\text{ and }z\in\R^d.
\end{equation*}
\end{itemize}
\end{problem}
\begin{lemma}[Periodic corrector]\label{Hom:L1}
  Let $a\in M(\R^d,\lambda)$ be $1$-periodic.
 \begin{enumerate}[(a)]
 \item For $i=1,\ldots,d$ there exists a unique $\phi_i\in H^1_{\#}(\Box)$ with $\fint_\Box\phi_i=0$ s.t. 
   \begin{equation}\label{L1:corr-weak}
     \fint_{\Box}a(\nabla\phi_i+e_i)\cdot \nabla\eta\,=0\qquad\text{for all }\eta\in H^1_{\#}(\Box).
    \end{equation}
  \item $\phi_i$ can be characterized as the unique, $1$-periodic function $\phi_i\in H^1_{\loc}(\R^d)$ with $\fint_\Box\phi_i=0$ and 
    \begin{equation}\label{Hom:Eq2a}
      -\nabla\cdot(a(\nabla\phi_i+e_i))=0\qquad\text{in }\mathcal D'(\R^d).
    \end{equation}
  \end{enumerate}
\end{lemma}
\begin{proof}[Proof of Lemma~\ref{Hom:L1} part (a)]
  Through out the proof $\ell$ denotes a non-negative integer. We set $\Box_{\ell}:=(-\frac{2\ell+1}{2},\frac{2\ell+1}{2})^d$ for $\ell\in\N_0$ and note that
  \begin{equation*}
    \Box_\ell=\dot\cup_{x\in\Z^d\cap\Box_\ell}(x+\Box)\qquad\text{up to a null-set}.
  \end{equation*}
  We first remark that the problem
  \begin{equation}\label{e.L1.1}
    \int_{\Box_\ell}a(\nabla\phi^\ell+e_i)\cdot \nabla\eta=0\qquad\text{for all }\eta\in H^1_{\#}(\Box_\ell),
  \end{equation}
  admits a unique solution $\phi^\ell\in H^1_{\#}(\Box_\ell)$ satisfying $\int_{\Box_\ell}\phi^\ell=0$, as follows from the Lax-Milgram theorem and Poincar\'e's inequality. In particular, for $\ell=0$ this proves (a). We claim that $\phi^\ell=\phi^0$ for any $\ell\in\N$. Indeed, let $\eta$ denote a test function in $H^1_{\#}(\Box_\ell)$. Then by $1$-periodicity of $a(\nabla\phi^0+e_i)$ we have
  \begin{eqnarray*}
    \int_{\Box_\ell}a(\nabla\phi^0+e_i)\cdot\nabla\eta&=&\sum_{x\in\Z^d\cap\Box_\ell}\int_{x+\Box}a(\nabla\phi^0+e_i)\cdot\nabla\eta\\
                                                      &=&\sum_{x\in\Z^d\cap\Box_\ell}\int_{\Box}a(\nabla\phi^0+e_i)\cdot\nabla\eta(\cdot+x)\\
                                                      &=&\int_{\Box}a(\nabla\phi^0+e_i)\cdot\nabla\tilde\eta,
  \end{eqnarray*}
  where $\tilde\eta:=\sum_{x\in\Z^d\cap\Box_\ell}\eta(\cdot+x)$. By construction we have $\tilde\eta\in H^1_{\#}(\Box)$, and thus the right-hand side is zero (by appealing to the equation for $\phi^0$). Hence, we deduce that $\phi^0$ solves \eqref{e.L1.1} and the condition $\int_{\Box_\ell}\phi^0=0$. Since \eqref{e.L1.1} admits a unique solution, we deduce that $\phi^0=\phi^\ell$.
  
  We are now in position to prove the equivalence of the problems  \eqref{L1:corr-weak} and \eqref{Hom:Eq2a}. For the direction ``$\Rightarrow$'' it suffices to show that for arbitrary $\eta\in C^\infty_c(\R^d)$ we have
  \begin{equation*}
    \int a(\nabla\phi^0+e_i)\cdot\nabla\eta=0.
  \end{equation*}
  For the argument, choose $\ell\in\N$ sufficiently large such that $\eta=0$ outside $\Box_\ell$. Then $\eta$ can be extended to a periodic function  $\eta^\ell\in H^1_{\#}(\Box_\ell)$, and we conclude that (since $\phi^0=\phi^\ell$),
  \begin{equation*}
    \int a(\nabla\phi^0+e_i)\cdot\nabla\eta=\int_{\Box_\ell}a(\nabla\phi^\ell+e_i)\cdot\nabla\eta^\ell=0.
  \end{equation*}
  For the other direction let $\phi$ denote the solution to \eqref{Hom:Eq2a}. It suffices to show that for arbitrary $\eta\in H^1_{\#}(\Box)$ we have
  \begin{equation}\label{Hom:Eq2aa}
    \fint_{\Box}a(\nabla\phi+e_i)\cdot\nabla\eta=0.
  \end{equation}
  By periodicity, we have for any $\rho_\ell\in C^\infty_c(\Box_\ell)$,
  \begin{eqnarray*}
    \fint_{\Box}a(\nabla\phi+e_i)\cdot\nabla\eta &=&    \fint_{\Box_\ell}a(\nabla\phi+e_i)\cdot\nabla\eta\\
                                                &=&    \fint_{\Box_\ell}a(\nabla\phi+e_i)\cdot\nabla(\eta\rho_\ell)+
                                                    \fint_{\Box_\ell}a(\nabla\phi+e_i)\cdot(\nabla\eta(1-\rho_\ell)-\eta\nabla\rho_\ell)\\
    &=&\fint_{\Box_\ell}a(\nabla\phi+e_i)\cdot(\nabla\eta(1-\rho_\ell)-\eta\nabla\rho_\ell),
  \end{eqnarray*}
  where the last identity holds thanks to \eqref{Hom:Eq2a}. Since $\dist(\Box_{\ell-1},\R^d\setminus \Box_\ell)=1$, we can find a cut-off function $\rho_\ell\in C^\infty_c(\Box_{\ell})$ such that $0\leq\rho_\ell\leq 1$, $\rho_\ell=1$ on $\Box_{\ell-1}$ and $|\nabla\rho_\ell|\leq C$ with $C$ independent of $\ell$. We thus conclude that
  \begin{eqnarray*}
    &&
       \big|\fint_{\Box_\ell}a(\nabla\phi+e_i)\cdot(\nabla\eta(1-\rho_\ell)-\eta\nabla\rho_\ell)\big|\\
    &\leq&
        (C+1)\frac{|\Box_{\ell}\setminus\Box_{\ell-1}|}{|\Box_{\ell}|}\fint_{\Box_\ell\setminus\Box_{\ell-1}}|a(\nabla\phi+e_i)|(|\nabla\eta|+|\eta|)\\
    &=&
        (C+1)\frac{|\Box_{\ell}\setminus\Box_{\ell-1}|}{|\Box_{\ell}|}\fint_{\Box}|a(\nabla\phi+e_i)|(|\nabla\eta|+|\eta|),
  \end{eqnarray*}
  where the last identity holds by $1$-periodicity of the integrand. In the limit $\ell\to \infty$, the right-hand side converges to $0$, and thus \eqref{Hom:Eq2aa} follows.
\end{proof}
\begin{definition}[Periodic corrector and homogenized coefficient]
  Let $a\in M(\R^d,\lambda)$ be $1$-periodic. The solution $\phi_i$ to \eqref{L1:corr-weak} is called the (periodic) \textbf{corrector} in direction $e_i$ (associated with $a$). The matrix $a_{\hom}\in\R^{d\times d}$ defined by
  \begin{equation*}
    a_{\hom}e_i:=\fint_\Box a(\nabla\phi_i+e_i)\qquad (i=1,\ldots,d)
  \end{equation*}
  is called the homogenized coefficient (associated with $a$).
\end{definition}

\begin{lemma}[Properties of the homogenized coefficients]\label{L:ahom}
Let $a\in M(\R^d,\lambda)$ be $1$-periodic and denote by $a_{\hom}$ the associated homogenized coefficients.
\begin{enumerate}[(a)]
\item (ellipticity). For any $\xi\in\R^d$ we have
  \begin{equation*}
    \xi\cdot a_{\hom}\xi\geq \lambda|\xi|^2.
  \end{equation*}
\item (invariance under transposition). Let $\phi_i^t$ denote the corrector associated with the transposed matrix  $a^t$. Then
  \begin{equation*}
    (a_{\hom})^te_i=\fint_{\Box}a^t(\nabla\phi_i^t+e_i).
  \end{equation*}
\item (symmetry). If $a$ is symmetric (a.e.\ in $\R^d$), then $a_{\hom}$ is symmetric.
\end{enumerate}
\end{lemma}
\begin{proof}
  For $\xi\in\R^d$ set $\phi_\xi:=\xi_i\phi_i$ and note that $\phi_\xi$ is the unique solution in $H^1_{\#}(\Box)$ with $\fint_{\Box}\phi_\xi=0$ to
  \begin{equation*}
    \fint_{\Box}a(\nabla\phi_\xi+\xi)\cdot\nabla\eta=0\qquad\text{for all }\eta\in H^1_{\#}(\Box).
  \end{equation*}
  Hence,
  \begin{eqnarray*}
    \frac{1}{\lambda}\xi\cdot a_{\hom}\xi=    \frac{1}{\lambda}\fint_{\Box}(\xi+\nabla\phi_\xi)\cdot a(\xi+\nabla\phi_\xi)\geq \fint_{\Box}|\xi+\nabla\phi_\xi|^2=|\xi|^2+\fint_{\Box}|\nabla\phi_\xi|^2,
  \end{eqnarray*}
  where we used that $\fint_{\Box}\nabla\phi_\xi=0$ by periodicity. This proves the ellipticity of $a_{\hom}$. For (b) note that
  \begin{eqnarray*}
    (a_{\hom})^te_i\cdot\xi&=&e_i\cdot a_{\hom}\xi=e_i\cdot \fint_{\Box}a(\nabla\phi_\xi+\xi)=\fint_{\Box}(\nabla\phi_i^t+e_i)\cdot a(\nabla\phi_\xi+\xi)\\
    &=&\fint_{\Box}a^t(\nabla\phi_i^t+e_i)\cdot (\nabla\phi_\xi+\xi)=\fint_{\Box}a^t(\nabla\phi_i^t+e_i)\cdot \xi.
  \end{eqnarray*}
  Since this is true for arbitrary $\xi\in\R^d$, (b) follows. If $a$ is symmetric, then $\phi^t_i=\phi_i$, and thus $a^t(\nabla\phi^t_i+e_i)=a(\nabla\phi_i+e_i)$. In this case (b) simplifies to
  \begin{equation*}
    (a_{\hom})^te_i=\fint_{\Box}a(\nabla\phi_i+e_i)=a_{\hom}e_i,
  \end{equation*}
  and thus $a_{\hom}$ is symmetric.
\end{proof}
We finally give the construction of the oscillating test function and establish the properties \eqref{Hom:Eq3} -- \eqref{Hom:Eq5}:
\begin{lemma}[Construction of the oscillating test function]\label{L:key}
  Let $a\in M(\R^d,\lambda)$ be $1$-periodic, let $a_{\hom}$ denote the associated homogenized coefficient and denote by $\phi_1^t,\ldots\phi_d^t$ the periodic correctors associated with the transposed matrix $a^t$. Then $(a(\tfrac{\cdot}{\e}))$ admits homogenization with homogenized coefficients $a_{\hom}$, and the oscillating test function can be defined as
  \begin{equation*}
    g_{i,\e}(x):=x_i+\e\phi_i(\tfrac{x}{\e}).
  \end{equation*}
\end{lemma}
For the proof we need to pass to the limit in sequences of rapidly oscillating functions:
\begin{proposition}[Rapidly oscillating functions]\label{P:rapid}
  Let $g\in L^2_{\loc}(\R^d)$ be $1$-periodic. Consider $g_\e(x):=g(\tfrac{x}{\e})$, $\e>0$. Then
  \begin{equation*}
    g_\e\wto \bar g:=\fint_{\Box}g\qquad\text{weakly in }L^2(O),
  \end{equation*}
  for any $O\subset\R^d$ open and bounded.
\end{proposition}
\begin{proof}
  Fix $O\subset\R^d$ open and bounded. It suffices to prove:
  \begin{align}
    \label{P:rapid:eq1}
    &\limsup\limits_{\e\downarrow 0}\|g_\e\|_{L^2(O)}^2\leq \operatorname{diam}(O)^d\int_{\Box}|g|^2,\\
    \label{P:rapid:eq2}
    &\forall Q\subset\R^d\text{ cube}\,:\,\fint_Qg_\e\to\bar g.
  \end{align}
  Indeed, this is sufficient, since \eqref{P:rapid:eq1} yields boundedness of $(g_\e)$, and for weak convergence in $L^2(O)$ it suffices  to test with a class of test functions that is dense in $L^2(O)$, e.g. $D:=\operatorname{span}\{\,\mathbf{1}_Q\,\text{ indicator function of a cube } Q\subset O\,\}$. Now, by linearity of the integral and by \eqref{P:rapid:eq2}, we have
  \begin{equation*}
    \int_O g_\e v\to \int_O\bar g v\qquad\text{for all }v\in D.
  \end{equation*}

  \step 1 Argument for \eqref{P:rapid:eq1}
  
  W.l.o.g. let $O$ be a cube. Set $Q_{z,\e}:=z+\e\Box$, and set $Z_\e:=\{\,z\in\e\Z^d\,:\,Q_{z,\e}\cap O\neq\emptyset\,\}$. Then
  \begin{equation*}
    \|g_\e\|^2_{L^2(O)}\leq \sum_{z\in Z_\e}\underbrace{\int_{Q_{z,\e}}|g(\frac{x}{\e})|^2\,dx}_{=\e^d\int_{\Box}|g|^2}=\e^d\# Z_\e\|g\|_{L^2(\Box)}^2,
  \end{equation*}
  and the conclusion follows since $\e^d\# Z_\e\to |O|=\operatorname{diam}(O)^d$.

  \step 2 Argument for \eqref{P:rapid:eq2}

  Let $Q$ denote a cube, set $Z_\e:=\{\,z\in\e\Z^d\,:\,Q_{z,\e}\subset Q\,\}$, and $Q_\e:=\cup_{z\in Z_\e}Q_{z,\e}$, so that $|Q_\e|\to |Q|$. Then
  \begin{equation*}
    |\int_Q g_\e-\int_{Q_\e}g_\e|\leq \int_{Q\setminus Q_\e}|g_\e|\leq |Q\setminus Q_\e|^\frac12 \|g_\e\|_{L^2(Q)}\to 0,
  \end{equation*}
  and $\int_{Q_{z,\e}} g_\e=\e^d\int_{\Box} g$, and thus
  \begin{equation*}
    \int_{Q_\e}g_\e=\sum_{z\in Z_\e}\int_{Q_{z,\e}}g_\e=\e^d\# Z_\e\int_{\Box} g=|Q_\e|\bar g\to |Q|\bar g.
  \end{equation*}
\end{proof}
\begin{proof}[Proof of Lemma~\ref{L:key}]
To ease notation we simply write $g_\e$, $\phi^t$ and $e$ instead of $g_{i,\e}$, $\phi^t_i$ and $e_i$. We only need to check \eqref{Hom:Eq3}--\eqref{Hom:Eq5}.

\step 1 Argument for \eqref{Hom:Eq3} and \eqref{Hom:Eq5}

Consider the periodic function $j:=a^t(\nabla\phi^t+e)\in L^2_{\loc}(\R^d)$, and note that we have $-\nabla\cdot j=0$ in $\mathcal D'(\R^d)$ by Lemma~\ref{Hom:L1} (b). Scaling yields $-\nabla\cdot j(\frac{\cdot}{\e})=0$ in $\mathcal D'(\R^d)$, and thus \eqref{Hom:Eq3}. Since $j$ is periodic, Proposition~\ref{P:rapid} yields $j(\frac{\cdot}{\e})\wto \fint_{\Box}j=a_{\hom}^te$ weakly in $L^2_{\loc}(\R^d)$, and thus \eqref{Hom:Eq5}.

\step 2 Argument for \eqref{Hom:Eq4}.

Since $\phi^t$ and $\nabla\phi^t$ are periodic functions, we conclude from Proposition~\ref{P:rapid} that $\phi^t(\tfrac{\cdot}{\e})$ and $\nabla\phi^t(\tfrac{\cdot}{\e})$ weakly converge in $L^2_{\loc}(\R^d)$, and thus $g_{\e}(x)=x_i+\e\phi(\frac{x}{\e})\wto x_i$ in  $H^1_{\loc}(\R^d)$.
\end{proof}

\subsection{Stochastic homogenization}
\label{S:stochastic}
\def\prob{\mathbb P}

\paragraph{Description of Random Coefficients}
In stochastic homogenization we only have ``uncertain'' or ``statistical'' information about the coefficient matrix $a$ (which models the microstructure of the material).  Hence, $\{a(x)\}_{x\in\R^d}$ has to be considered as a family of matrix valued random variables. For stochastic homogenization the random field $a$ is required to be stationary in the sense that for any finite number of points $x_1,\ldots,x_k$ and shift $z\in\R^d$ the random variable $(a(x_1+z),\ldots,a(x_k+z))$ has a distribution independent of $z$, i.e.\ the coefficients are \textit{statistically homogeneous}. In addition, for homogenization towards a deterministic limit, the random field $a$ is required to be ergodic in the sense that spatial averages of $a$ over cubes of size $R$ converge to a deterministic constant as $R\uparrow\infty$; one could interpret this by saying that a typical sample of the coefficient field already carries all information about the statistics of the random coefficients.
For our purpose it is convenient to work within the following mathematical framework:
\begin{itemize}
\item We introduce a \textit{configuration space} of admissible coefficient fields
  \begin{equation*}
    \Omega:=\Big\{\,a:\R^d\to\R^{d\times d}_{\sym}\,\text{ is measurable and uniformly elliptic in the sense of }\eqref{Hom:Eq1}\Big\}
  \end{equation*}
\item We introduce a probability measure $\prob$ on $\Omega$ (which we equip with a canonical $\sigma$-algebra). We write $\expec{\cdot}$ for the associated expectation.
\end{itemize}
The measure $\prob$ describes ``the frequency of seeing a certain microstructure in our random material''. The assumption of stationarity and ergodicity can be phrased as follows:
\paragraph{Assumption (S).} Let $(\Omega,\mathcal F,\mathbb P)$ denote a probability space equipped with the (spatial) ``shift operator''
\begin{equation*}
  \tau:\R^d\times\Omega\to\Omega,\qquad \tau(z,a):=a(\cdot+z)\qquad(=:\tau_za),
\end{equation*}
which we assume to be measurable. We assume that the following properties are satisfied:
\begin{itemize}
\item (Stationarity). For all $z\in\R^d$ and any random variable $f\in L^1(\Omega,\mathbb P)$ we have
  \begin{equation*}
    \expec{f\circ\tau_z}=\expec{f}.
  \end{equation*}
\item (Ergodicity). For any $f\in L^1(\Omega,\mathbb P)$ we have
  \begin{equation}\label{ergodic}
    \lim\limits_{R\uparrow\infty}\fint_{R\Box}f(\tau_{z}a)\,dz=\expec{f}\qquad\text{for $\prob$-a.e.\ $a\in\Omega$}.
  \end{equation}
\end{itemize}
\begin{remark}
  Assumption (S) can be rephrased by saying that $(\Omega,\mathcal F,\prob, \tau)$ forms a $d$-dimensional ergodic, measure-preserving dynamical system. Ergodicity is usually defined as follows: For any $E\subset\Omega$ (measurable) we have
  \begin{equation*}
    E\text{ is shift-invariant }\qquad\Rightarrow\qquad P(E)\in\{0,1\}.
  \end{equation*}
  Here, a (measurable) set $E\subset\Omega$ is called shift invariant, if $\tau_zE= E$ for all $z\in\R^d$. The fact that this definition of ergodicity implies \eqref{ergodic} is due to {Birkhoff's pointwise ergodic theorem}, e.g. see Ackoglu \& Krengel \cite{AK81} for reference that covers the multidimensional case. \end{remark}

  \begin{figure}[h]
    \begin{subfigure}[t]{0.5\textwidth}
      \centering
      \includegraphics[width=7cm]{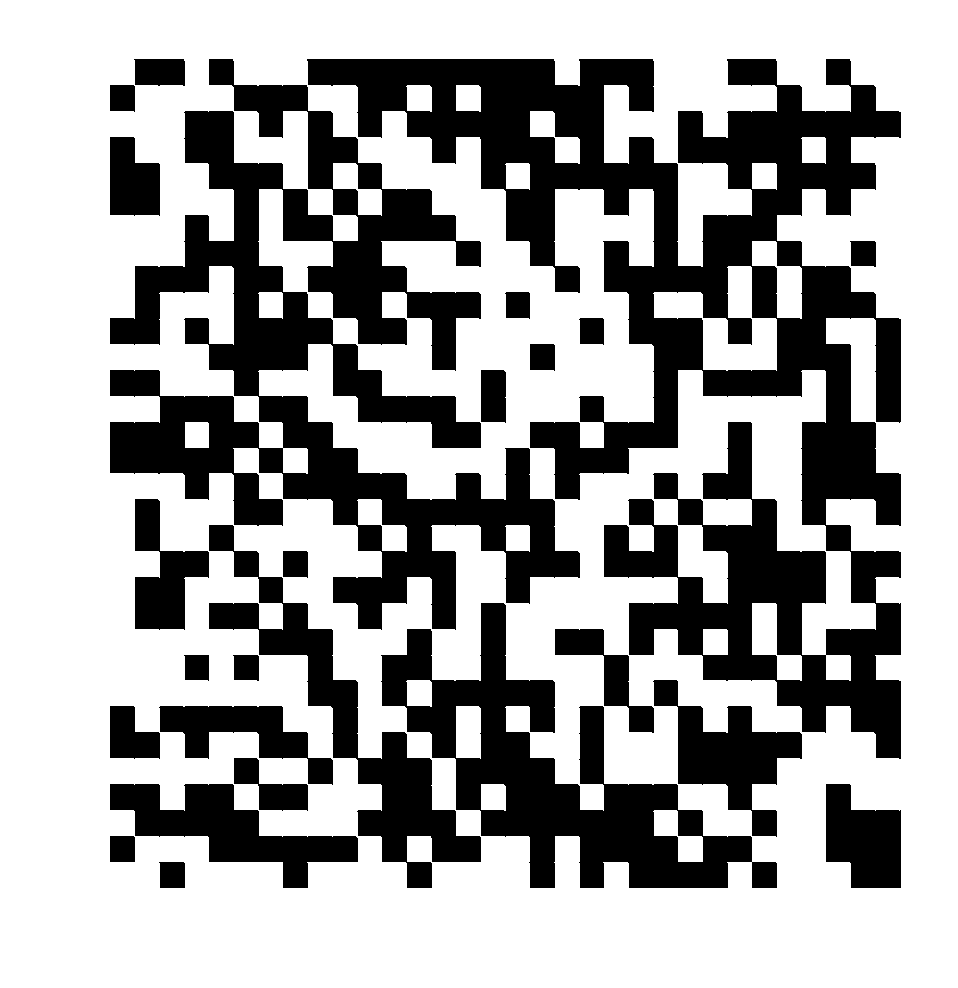}
      \label{fig.sto1}
      \caption{independent and identically distributed tiles}
    \end{subfigure}
    \begin{subfigure}[t]{0.5\textwidth}
      \centering
      \includegraphics[width=7cm]{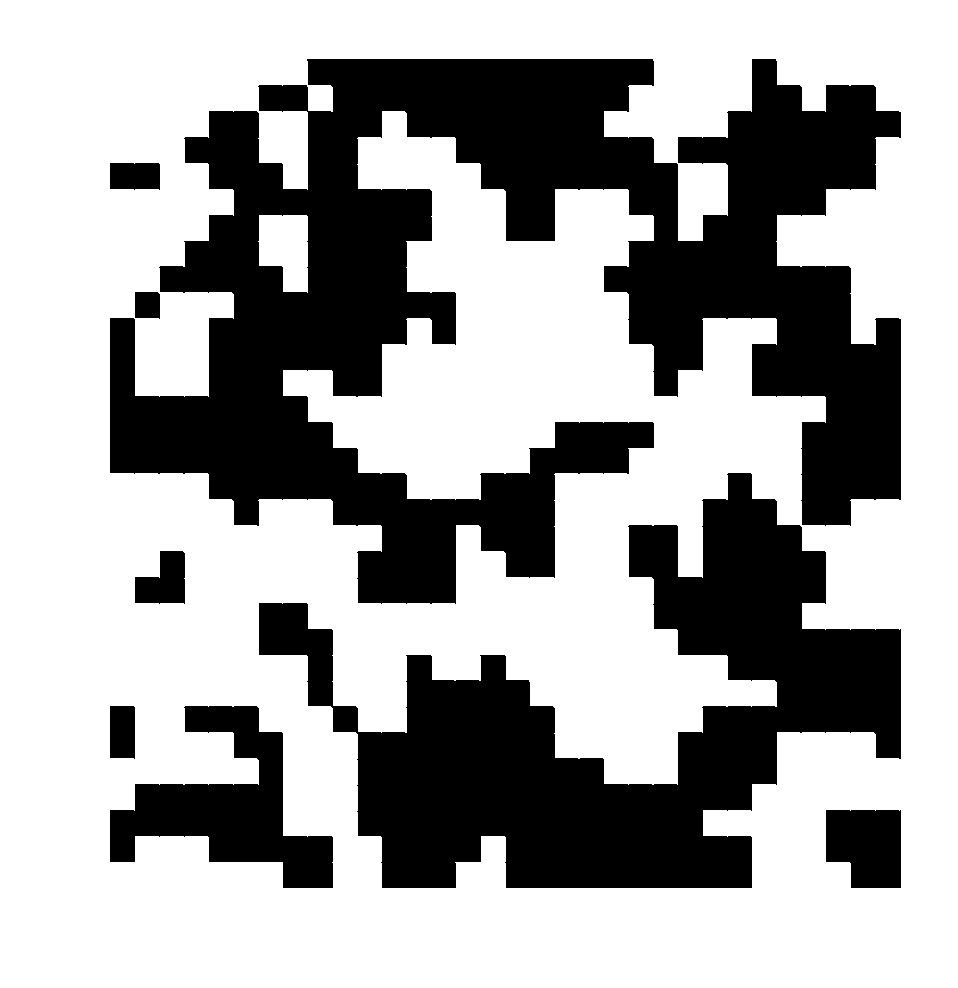}
      \label{fig.sto2}
      \caption{correlated tiles}
    \end{subfigure}
    \caption{Typical sample of a stationary, ergodic random
      checkerboard type coefficient field that takes two values with
      the same probability.}\label{fig.stochchecker}
  \end{figure}

\begin{example}[Random Checkerboard]
  Let $z\in(0,1)^d$ denote a random vector with uniform distribution, and $\{a_k\}_{k\in\Z^d}$ a family of independent, identically distributed random matrices in $\Omega_0:=\{a_0\in\R^{d\times d}\,:\, a_0\text{ satisfies }\eqref{Hom:Eq1}\}$. Then 
  \begin{equation*}
    a:\R^d\to\R^{d\times d},\qquad a(x):=\sum_{k\in\Z^d}1_{k+z+\Box}(x)a_k,
  \end{equation*}
  defines random field in $\Omega$ whose distribution is stationary and ergodic. Figures~\ref{fig.stochchecker} (a) shows an example for a sample of such a random field in the case, when $a$ only takes two values, say $a_{\text{white}}$ and $a_{\text{black}}$. More precisely, the construction of $\prob$ is as follows: We start with the probability space $(\Omega_0,\mathcal F_0,\prob_0)$ where $\mathcal F_0$ denotes the Borel-$\sigma$-algebra on $\Omega_0\subset\R^{d\times d}$, and $\prob_0$ describes the distribution on a single tile. Then consider the product space
  \begin{equation*}
    (\Omega',\mathcal F',\prob'):=\big(\Omega_0^{\Z^d}\times\Box,\mathcal F_0^{\otimes \Z^d}\otimes\mathcal B(\Box),\prob_0^{\otimes\Z^d}\otimes\mathcal L\big),
  \end{equation*}
  where $\mathcal L$ denotes the Lebesgue measure on $\Box$, and the map
  \begin{equation*}
    \pi:\Omega'\to\Omega,\qquad \pi(a,z):=\sum_{k\in\Z^d}1_{k+z+\Box}(\cdot)a_k.
  \end{equation*}
  The probability measure $\prob$ is then obtained as the push-forward of $\prob'$ under $\pi$ and yields a stationary and ergodic measure. Note that the associated coefficients have a finite range of dependence, in the sense that if we take $x,x'\in\R^d$ with $|x-x'|>\diam(\Box)$, then the random variables $a(x)$ and $a(x')$ are independent, and \eqref{ergodic} is a consequence of the law of large numbers. We might vary the example by considering the convolution 
  \begin{equation*}
    \Lambda_\varphi:\Omega_0^{\Z^d}\to\Omega_0^{\Z^d},\qquad \Lambda_\varphi(a)_k:=\sum_{j\in\Z^d}\varphi(j-k)a_j,
  \end{equation*}
  with some non-negative convolution kernel  $\varphi:\Z^d\to\R_{\geq 0}$ satisfying $\sum_{k\in\Z^d}\varphi(k)=1$. If we define $\prob$ as the push-forward of $P'$ under the mapping
  \begin{equation*}
    \pi_{\varphi}:\Omega'\to\Omega,\qquad \pi(a,z):=\sum_{k\in\Z^d}1_{k+z+\Box}(\cdot)\Lambda_\varphi(a)_k,
  \end{equation*}
  we obtain again a stationary and ergodic measure. If $\varphi$ is not compactly supported, then $a(x)$ and $a(x')$ are always correlated (even for $|x-x'|\gg 1$), yet they decorrelate on large distances, i.e.\ for $x,x'\in\Z^d$ we have
  \begin{eqnarray*}
    \expec{(a(x)-\expec{a})(a(x')-\expec{a})}&=&\sum_{j,j'\in\Z^d}\varphi(j-x)\varphi(j'-x)\operatorname{Cov}(a_{j},a_{j'})\\
    &\leq&\operatorname{Var}(a_0)\sum_{j\in\Z^d}\varphi(j-x)\varphi(j-x')\\
    &=&\sum_{j\in\Z^d}\varphi(j)\varphi(j+x-x')\to 0\qquad\text{as }|x-x'|\to \infty.
  \end{eqnarray*}
  Figure~\ref{fig.stochchecker} (b) shows a typical sample of a coefficient field obtained in this way (with a kernel $\varphi$ that exponentially decays).
\end{example}

\begin{example}[Gaussian random fields]
  Let $\{\xi(x)\}_{x\in\R^d}$ denote a centered, stationary Gaussian random field with covariance function $C(x)=\operatorname{Cov}(\xi(x),\xi(0))$. Roughly speaking this means that for any $x_1,\ldots,x_N\in\R^d$ the random vector $(\xi(x_1),\ldots,\xi(x_N))$ has the distribution of a multivariate Gaussian with mean zero and covariance matrix $\Sigma_{ij}=C(x_i-x_j)$. Suppose that $|C(x)|\leq (|x|+1)^{-\alpha}$ for some $\alpha>0$ (i.e.\ at least some algebraic decay of correlations). Let $\Lambda:\R\to\Omega_0$ denote a Lipschitz function. Then $a(x):=\Lambda(\xi(x))$ defines a stationary and ergodic ensemble of coefficient fields.  
\end{example}

\begin{problem}[Periodic coefficients]
  Let $a_\#$ denote a periodic coefficient field in $\Omega$. Show that there exists a stationary and ergodic measure $\prob$ on $\Omega$, s.t. 
  for any open set $O\subset\Box$ we have
  \begin{equation*}
    \prob(\{a_\#(\cdot+z)\,:\,z\in O\})=|O|.
  \end{equation*}
  (Hence, with full probability a sample $a$ is a translation of $a_{\#}$. In this sense periodic coefficients can be recast into the stochastic framework).
\end{problem}

\paragraph{Homogenization in the stochastic case.}

The analogue to Theorem~\ref{Hom:T1} in the stochastic case is the following:

\begin{theorem}[Papanicolaou \& Varadhan '79 \cite{PV79}, Kozlov '79 \cite{Kozlov-79}]\label{Hom:T2}
  Suppose Assumption~(S).  There exists a (uniformly elliptic) constant coefficient tensor $a_{\hom}$ such that for $\prob$-a.e.\ $a\in\Omega$ we have:
  \smallskip

  For all $O\subset\R^d$ open and bounded, for all $f\in L^2(O)$, $F\in L^2(O,\R^{d})$, and $\e>0$, the unique weak solution $u_\e\in H^1_0(O)$ to
  \begin{equation*}
    -\nabla\cdot(a(\tfrac{x}{\e}) \nabla u_\e)=f-\nabla\cdot F\qquad\text{in }O
  \end{equation*}
  weakly converges in $H^1(O)$ to the  weak solution $u_0\in H^1_0(O)$ to
  \begin{equation*}
    -\nabla\cdot(a_{\hom}\nabla u_0)=f-\nabla\cdot F\qquad\text{in }O,
  \end{equation*}
  and we have
  \begin{equation*}
    a(\tfrac{\cdot}{\e})\nabla u_\e\wto a_{\hom}\nabla u_0\qquad\text{weakly in }L^2(O,\R^d).
  \end{equation*}
\end{theorem}
Except for the assumption on $a$, the statement is similar to Theorem~\ref{Hom:T1}. Since $a\in\Omega$ is random, the solutions $u_\e\in H^1_0(O)$ (which depend in a nonlinear way on $a$) are random quantities. In contrast, the homogenized coefficient matrix $a_{\hom}$ is deterministic and only depends on $\prob$, but not on the individual sample $a$, the domain $O$ or the right-hand side. Therefore, the limiting equation and thus $u_0$ is deterministic. Hence, in the theorem we pass from an elliptic equation with random, rapidly oscillating coefficients to a deterministic equation with constant coefficients, which is a huge reduction of complexity. A numerical illustration of the result is given in Figure~\ref{fig.stoch}.
\medskip

\begin{figure}[!]
  \bigskip

  \begin{tabular}[t]{p{.4\linewidth} p{.1\linewidth} p{.4\linewidth}}
    \begin{subfigure}[t]{.9\linewidth}
      \centering
      \includegraphics[width=1.24\linewidth]{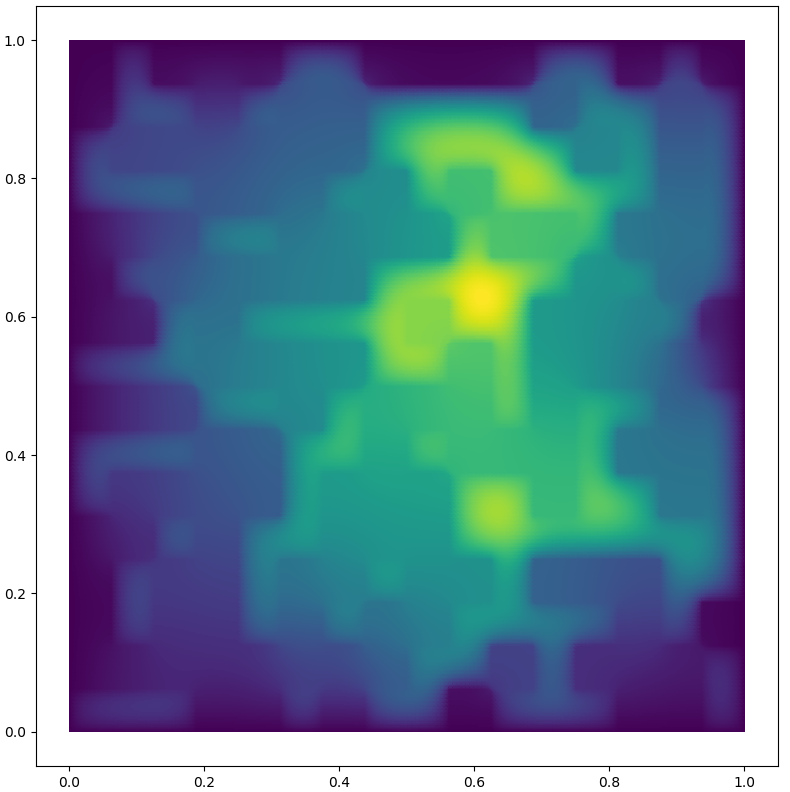}
      \caption{}
      \label{fig.stoch1}
    \end{subfigure}
    & &
    \begin{subfigure}[t]{.99\linewidth}
      \centering
      \includegraphics[width=1.1\linewidth]{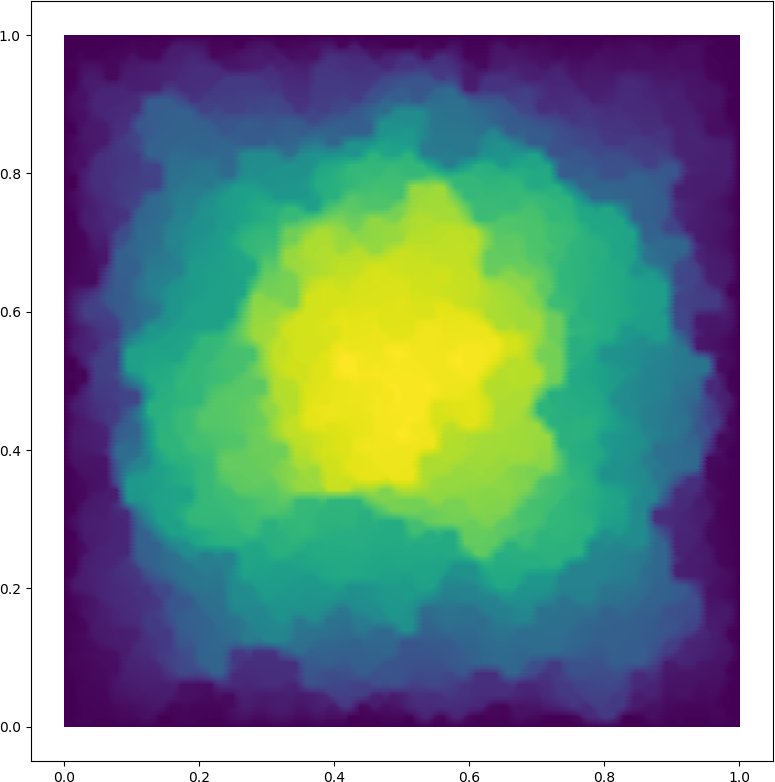}
      \caption{}
      \label{fig.stoch2}
    \end{subfigure}
  \end{tabular}
  \caption{\textit{Illustration of Theorem~\ref{Hom:T2}} in the case of a random checkerboard-like coefficient field with independent and identically distributed tiles. (a) and (b) show realizations of the solutions to $-\nabla\cdot (a(\frac{\cdot}{\e}\nabla u_\e)=1$ in $H^1_0((0,1)^2)$ for $\e\in\{\frac{1}{8},\frac{1}{32}\}$.
}\label{fig.stoch}
\end{figure}

As in the periodic case, the core of the proof is the construction of a corrector (which is then used to define the oscillating test functions in Definition~\ref{D:oscillating}).
\begin{proposition}[The corrector in stochastic homogenization]\label{P:corr-stoch}
  Suppose that Assumption (S) is satisfied. For any $\xi\in\R^d$ there exists a unique random field $\phi:\Omega\times\R^d\to\R$, called the corrector associated with $\xi$, such that:
  \begin{enumerate}[(a)]
  \item For $\prob$-a.e.\ $a\in\Omega$ the function $\phi(a,\cdot)\in H^1_{\loc}(\R^d)$ is a distributional solution to 
    \begin{equation}\label{eq:random-corr}
      -\nabla\cdot a(\nabla\phi(a,\cdot)+\xi)=0\qquad\text{in }\mathcal D'(\R^d),
    \end{equation}
    with sublinear growth in the sense that
    \begin{equation}\label{sublin}
      \limsup\limits_{R\to\infty}\frac{1}{R^2}\fint_{R\Box}|\phi(a,\cdot)|^2=0,
    \end{equation}
    and $\phi(a,\cdot)$ is anchored in the sense that $\fint_{\Box}\phi(a,y)\,dy=0$.
  \item $\nabla\phi$ is stationary in the sense of Definition~\ref{stationary} below.
  \item $\expec{\fint_\Box|\nabla\phi|^2}\leq \frac{1-\lambda^2}{\lambda^2}\expec{\fint_{\Box}|a\xi|^2}$, and
    $\expec{\fint_{\Box}\nabla\phi}=0$.
  \end{enumerate}
\end{proposition}
Let us remark that the arguments that we are going to present extend verbatim to the caseof systems, see Remark~\ref{R:systems}. Before we discuss the proof of Proposition~\ref{P:corr-stoch}, we note that in combination with Lemma~\ref{L:key}, Proposition~\ref{P:corr-stoch} yields a proof of Theorem~\ref{Hom:T2}. In fact, we only need to show:
\begin{lemma}\label{Lstochhom}
 Suppose that Assumption (S) is satisfied. For $i=1,\ldots,e_i$ let $\phi_i$ (resp. $\phi_i^t$) denote the  corrector associated with $e_i$ and $a$ (resp. the transposed coefficient field $a^t$) from Proposition~\ref{P:corr-stoch} and consider the matrix $a_{\hom}\in\R^{d\times d}$ defined by
 \begin{equation*}
   a_{\hom}e_i:=\expec{\fint_{\Box}a(y)(\nabla\phi_i(a,y)+e_i)\,dy}.
 \end{equation*}
 Then $a_{\hom}$ is elliptic and for $\prob$-a.e.\ $a\in\Omega$ the family $(a(\frac{\cdot}{\e}))$ admits homogenization with homogenized coefficients given by $a_{\hom}$ and oscillating test functions given by 
 \begin{equation*}
   g_{i,\e}(x):=x_i+\e\phi_i^t(a,\tfrac{x}{\e}).
 \end{equation*}
\end{lemma}
The proof is a rather direct consequence of the properties of the corrector and ergodicity in form of \eqref{ergodic}. We present it in Section~\ref{S:auxproof}.
\medskip

The main part of this section is devoted to the proof of Proposition~\ref{P:corr-stoch}. We first remark that the sublinearity condition \eqref{sublin} is a natural ``boundary condition at infinity''. {Indeed, if the coefficient field $a$ is constant, then sublinearity implies that the solution to \eqref{eq:random-corr} is unique up to an additive constant.}
\begin{lemma}[A priori estimate for sublinear solutions]\label{L:unique-const}
  Let $a\in\Omega$. Suppose $u\in H^1_{\loc}(\R^d)$ has sublinear growth in the sense of \eqref{sublin} and solves
  \begin{equation*}
    -\nabla\cdot a\nabla u=0\qquad\text{in }\mathcal D'(\R^d).
  \end{equation*}
  Then {
    \begin{equation}\label{L:bound-grad}
      \limsup\limits_{R\to\infty}\fint_{R\Box}|\nabla u|^2=0.
    \end{equation}
    In particular, if the coefficient field $a$ is constant, then $u$ is constant.  }
\end{lemma}

\begin{proof}
    Let $\eta\in C^\infty_c(\R^d)$. By Leibniz' rule we have
    \begin{equation*}\nonumber
      \nabla(u\eta)\cdot a\nabla(u\eta)=\nabla(u\eta^2)\cdot a\nabla u+u^2\nabla\eta\cdot a\nabla\eta+u\big(\nabla(u\eta )\cdot a\nabla\eta-\nabla\eta\cdot a\nabla(u\eta )\big)
    \end{equation*}
  Note that by Young's inequality we have
  \begin{eqnarray*}
    |u\big(\nabla(u\eta)\cdot a\nabla\eta-\nabla\eta\cdot a\nabla(u\eta)\big)|&\leq& 2|u||\nabla(u\eta)||\nabla\eta|\leq \frac{\lambda}{2}|\nabla(u\eta)|^2+\frac{2}{\lambda}u^2|\nabla\eta|^2.
  \end{eqnarray*}
  and thus by ellipticity, and the equation for $u$, 
  \begin{equation*}
    \lambda\int|\nabla(u\eta)|^2\leq \int    \nabla(u\eta)\cdot a\nabla(u\eta)\leq (1+\frac{2}{\lambda})\int u^2|\nabla\eta|^2+\frac{\lambda}{2}\int|\nabla(u\eta)|^2.
  \end{equation*}
  We conclude that
  \begin{equation*}
    \int|\nabla(u\eta)|^2\leq C(d)\int u^2|\nabla\eta|^2.
  \end{equation*}
  We now specify the cut-off function: Let $\eta_1\in C^\infty_c(2\Box)$ satisfy $\eta_1=1$ on $\Box$. Then the above estimate applied with $\eta_1(\frac{\cdot}{R})$ yields
  \begin{eqnarray*}
    \fint_{R\Box}|\nabla u|^2\leq R^{-d}\int|\nabla(u\eta_1(\tfrac{\cdot}{R}))|^2&\leq& C(d,\lambda)R^{-d-2}\int u^2|(\nabla\eta_1)(\tfrac{\cdot}{R})|^2\\
&\leq& C(d,\lambda,\eta_1)(2R)^{-2}\fint_{2R\Box}|u|^2.
  \end{eqnarray*}
  By sublinearity, for $R\to\infty$ the right-hand side converges to $0${, which yields the first claim. If $a$ is constant, then by a standard interior estimate we have
    \begin{equation*}
      \|\nabla u\|_{L^\infty(R\Box)}^2\leq C\fint_{2R\Box}|\nabla u|^2,
    \end{equation*}
    for a constant $C$ that is independent of $R$. We conclude that $\nabla u=0$ a.e.~in $\R^d$ and thus $u$ is constant.}
\end{proof}
{Let us anticipate that the above estimate also yields uniqueness in the case of stationary and ergodic random coefficients for solutions with sublinear growth and stationary gradients, see Corollary~\ref{C:unique_sublin} below.}
On the other hand it is not clear at all that the equation $-\nabla\cdot (a\nabla \phi)=\nabla\cdot (a\xi)$ admits a sublinear solution. In fact, this is only true for ``generic'' coefficient fields $a\in\Omega$, in particular, we shall see that this is true for $\prob$-a.e.\ $a\in\Omega$ when $\prob$ is stationary and ergodic.  Our strategy is the following:
\begin{itemize}
\item Instead of the equation $-\nabla\cdot(a\nabla\phi)=\nabla\cdot (a\xi)=0$ we consider the \textbf{modified corrector equation}
  \begin{equation}\label{eq:modified-corr}
    \frac{1}{T}\phi_T-\nabla\cdot a(\nabla \phi_T)=\nabla\cdot(a\xi)\qquad\text{in }\R^d\qquad (T\gg 1),
  \end{equation}
  which turns out to be well-posed for all $a\in \Omega$ and yields  an a priori estimate of the form
  \begin{equation}\label{eq:modified}
    \forall R\geq T\,:\,  \fint_{\sqrt R\Box}\frac{1}{T}|\phi_T|^2+|\nabla\phi_T|^2\leq C(d,\lambda)|\xi|^2.
  \end{equation}
\item By {\bf stationarity} of $\prob$ we can turn \eqref{eq:modified} into an averaged estimate that on the level of $\nabla\phi_T$ is uniform in $T$,
  \begin{equation*}
    \expec{\fint_\Box|\nabla\phi_T|^2}\leq C(d,\lambda)|\xi|^2.
  \end{equation*}
\item This allows us to pass to the weak limit (for $T\uparrow\infty$) in an appropriate subspace of random fields. The limit $\phi$ is a solution to the corrector equation, its gradient is stationary, i.e.\ $\nabla\phi(a,x+z)=\nabla\phi(\tau_xa,y)$ and  satisfies
  \begin{equation*}
    \expec{\fint_{\Box}|\nabla\phi|^2}<\infty\qquad\text{and}\qquad\expec{\fint_{\Box}\nabla\phi}=0.
  \end{equation*}
\item Finally, by exploiting \textbf{ergodicity} and the property that $\expec{\fint_{\Box}\nabla\phi_T}=0$ we deduce sublinearity.
\end{itemize}
We start with the argument that establishes sublinearity, since the latter is the most interesting property of the corrector. In fact, the argument can be split into a purely deterministic argument (that we state next), and a non-deterministic part that exploits ergodicity, see proof of Corollary~\ref{C:existence_sublin} below.
\begin{lemma}[sublinearity]\label{L:sublinear}
  Let $u\in H^1_{\loc}(\R^d)$ satisfy $\fint_{\Box}u=0$, 
  \begin{align}\label{sublinear:ass1}
    \limsup\limits_{R\to\infty}\fint_{R\Box}|\nabla u|^2&\,<\infty,\qquad\text{and}\\\label{sublinear:ass2}
\limsup\limits_{R\to\infty}\fint_{R\Box}\nabla u(y)\cdot F(Rx)&\,=0\qquad\text{for all }F\in L^2(\Box,\R^d).
  \end{align}
  Then we have
  \begin{eqnarray}
    \label{sublinear:1}
    \limsup_{R\to\infty}\Big(\frac{1}{R^2}\fint_{R\Box}\big|u-\fint_{R\Box}u\big|^2\Big)^\frac12&=&0,\\
    \label{sublinear:2}
    \limsup_{R\to\infty}\Big(\frac{1}{R^2}\fint_{R\Box}\big|u\big|^2\Big)^\frac12&=&0.
  \end{eqnarray}
\end{lemma}
\begin{proof}
  \step{1} Proof of \eqref{sublinear:1}.

  We appeal to a scaling argument: Consider $u_R(x):=\frac{1}{R}(u(Rx)-\fint_{R\Box}u)$ and note that
  \begin{equation*}
    \nabla u_R(x)=\nabla u(Rx),\qquad \int_{\Box}|u_R|^2=\frac{1}{R^2}\fint_{R\Box}\big|u-\fint_{R\Box}u\big|^2.
  \end{equation*}
  Hence, it suffices to show that $u_R\to 0$ strongly in $L^2(\Box)$. Since $\fint_\Box u_R=0$, Poincar\'e's inequality yields
  \begin{equation*}
    \int_\Box u_R^2\leq \int_\Box |\nabla u_R|^2=\fint_{R\Box}|\nabla u|^2,
  \end{equation*}
  and \eqref{sublinear:ass1} implies that $(u_R)$ is bounded in $H^1(\Box)$. By weak compactness of bounded sequences in $H^1(\Box)$, we find $u_\infty\in H^1(\Box)$ such that $u_{R}\wto u_\infty$ weakly in $H^1(\Box)$ (for a subsequence that we do not relabel). Since $H^1(\Box)\subset L^2(\Box)$ is compactly embedded (by the Theorem of Rellich-Kondrachov), we may assume w.l.o.g. that we also have $u_{R}\to u_\infty$ strongly in $L^2(\Box)$. We claim that $u_\infty=0$ (which then also implies that the convergence holds for the entire sequence). Indeed, from \eqref{sublinear:ass2} we deduce that
  \begin{equation*}
    \fint_{\Box}\nabla u_\infty\cdot F=\lim\limits_{R\to\infty}\fint_{\Box}\nabla u_{R}\cdot F=\lim\limits_{R\to\infty}\fint_{\Box}\nabla u(Rx)\cdot F(x) = 0.
  \end{equation*}
  Hence, $\nabla u_\infty=0$, and thus $u_\infty$ is constant. Since $\fint_{\Box}u_\infty=0$, $u_\infty=0$ follows.

  \step 2 Proof of \eqref{sublinear:2}.

  Set $J(t)=\fint_{t\Box}u=\fint_{\Box}u(tx)\dd x$. We have
  \begin{align*}
    \partial_tJ(t)=\fint_{\Box}\nabla u(tx)\cdot x\dd x.
  \end{align*}
  Let $R\gg T\geq1$. Then
  \begin{align*}
    \big|\fint_{R\Box}u-\fint_{\Box}u\big|&=\big|\int_{1}^{R}\partial_{t}J(t)\dd t\big|\leq\int_{1}^{T}|\partial_{t}J(t)|\dd t+\int_{T}^{R}|\partial_{t}J(t)|\dd t\\
    &=\int_{1}^{T}\big|\fint_{\Box}\nabla u(tx)\cdot x\big|\dd t+\int_{T}^{R}\big|\fint_{\Box}\nabla u(tx)\cdot x\big|\dd t\\
    &\leq C(d)\int_{1}^{T}\left(\fint_{\Box}|\nabla u(tx)|^2\dd x\right)^\frac12+(R-T)\sup_{t\geq T}\big|\fint_{\Box}\nabla u(tx)\cdot x\big|.
  \end{align*}
  By \eqref{sublinear:ass1}
  \begin{equation*}
    \int_{1}^{T}\left(\fint_{\Box}|\nabla u(tx)|^2\dd x\right)^\frac12=    \int_{1}^{T}\left(\fint_{t\Box}|\nabla u(x)|^2\dd x\right)^\frac12<\infty.
  \end{equation*}
  Hence, dividing  by $R$ and taking the limit $R\to\infty$ yields
  \begin{align*}
    \limsup\limits_{R\to\infty}\frac{1}{R}\big|\fint_{R\Box}u-\fint_{\Box}u\big|&\leq \sup_{t\geq T}\big|\fint_{\Box}\nabla u(tx)\cdot x\big|.
  \end{align*}
  By \eqref{sublinear:ass2} (applied with $F(x)=x$), in the limit $T\to\infty$, the last expression converges to $0$.
  We conclude
  \begin{align*}
    \left(\frac{1}{R^2}\fint_{R\Box}|u|^2\right)^\frac12&\leq\underbrace{\Big(\frac{1}{R^2}\fint_{R\Box}\big|u-\fint_{R\Box}u\big|^2\Big)^\frac12}_{\to0} + \underbrace{\frac{1}{R}\big|\fint_{R\Box}u-\fint_{\Box}u\big|}_{\to0}+\underbrace{\frac{1}{R}\big|\fint_{\Box}u\big|}_{=0}.
  \end{align*}
\end{proof}
\begin{remark}
Let us anticipate that in the proof of Proposition~\ref{P:corr-stoch} we apply Lemma~\ref{L:sublinear} in the special situation where $u$ is a realization of  a random field $u:\Omega\times\R^d\to\R$, whose gradient is stationary and satisfies $\expec{\fint_{\Box}|\nabla u|^2}<\infty$ and $\expec{\fint_{\Box}\nabla u}=0$. Then, properties \eqref{sublinear:ass1} and \eqref{sublinear:ass2} hold for $u(a,\cdot)$ for  $\prob$-a.e.\ $a\in\Omega$ as we will prove by appealing to  ergodicity \eqref{ergodic}.
\end{remark}
Another argument that is purely deterministic is the existence theory for the modified corrector. Note that the right-hand side of \eqref{eq:modified-corr} is a divergence of a vector field $F:\R^d\to\R^d$ that is not integrable (yet bounded). For the deterministic a priori estimate it is convenient to consider the weighted norm
\begin{equation}\label{eq:normtheta}
  \|F\|_{\theta}^2:=\int|F(x)|^2\theta(x)\,dx,
\end{equation}
where $\theta:\R^d\to\R$ denotes a positive, exponentially decaying weight to be specified below. In the following various estimates are localized on cubes. We use the notation
\begin{equation*}
  \mathcal Q:=\big\{ Q=x+r\Box\,:\,x\in\R^d,\,r>0\,\big\},\qquad \Box:=(-\frac12,\frac12)^d.
\end{equation*}
\begin{lemma}\label{L:resolvent_equation}
  There exists a positive, exponentially decaying weight $\theta$ with $\int_{\R^d}\theta=1$ (that only depends on $d$ and $\lambda$) and a constant $C=C(d,\lambda)$, such that the following properties hold:
  Let $a\in\Omega$, $T>0$, $F\in L^2_{\loc}(\R^d,\R^d)$ with $\|F\|_{\theta}<\infty$ (cf. \eqref{eq:normtheta}). Then there exists a unique solution $u\in H^1_{\loc}(\R^d)$ to
  \begin{equation}\label{resolvent_equation}
    \frac{1}{T}u-\nabla\cdot a\nabla u=\nabla\cdot F\qquad\text{in }\mathcal D'(\R^d),
  \end{equation}
  such that $R\mapsto \fint_{R\Box}|u|^2$ grows at most polynomially for $R\to \infty$. The solution satisfies the a priori estimate
  \begin{equation}\label{resolvent_apriori}
    \forall R\geq T\,:\,\fint_{\sqrt R\Box}\frac{1}{T}|u|^2+|\nabla u|^2\,\leq C(d,\lambda)\|F(\sqrt R\cdot)\|_{\theta}^2.
  \end{equation}
\end{lemma}

\begin{proof}
  \step 1 Proof of the a priori estimate.

 We claim that $u$ satisfies \eqref{resolvent_apriori}. For the argument let $\eta\in C^\infty_c(\R^d)$. By Leibniz' rule we have
  \begin{equation*}\nonumber
    \nabla(u\eta)\cdot a\nabla(u\eta)=\nabla(u\eta^2)\cdot a\nabla u+u^2\nabla\eta\cdot a\nabla\eta+u\big(\nabla(u\eta )\cdot a\nabla\eta-\nabla\eta\cdot a\nabla(u\eta )\big)
  \end{equation*}
  Note that 
  \begin{equation*}
    |u\big(\nabla(u\eta)\cdot a\nabla\eta-\nabla\eta\cdot a\nabla(u\eta)\big)|\leq 2|u||\nabla(u\eta)||\nabla\eta|.
  \end{equation*}
  Thus, integration, ellipticity, and \eqref{resolvent_equation} yield
  \begin{eqnarray*}
\int    \frac{1}{T}|u\eta|^2+\lambda|\nabla(u\eta)|^2&\leq &     \int\frac{1}{T}|u\eta|^2+\nabla(u\eta)\cdot a\nabla(u\eta)\\
    &\leq&\int \frac{1}{T}(u\eta^2)u+\nabla(\eta^2u)\cdot a\nabla u+2\int |\nabla(u\eta)||u||\nabla\eta|\\
                                                      &=&\int F\cdot \nabla(u\eta^2)+2\int |\nabla(u\eta)||u||\nabla\eta|\\
    &\leq&\int |F||\eta||\nabla(u\eta)|+|F||u||\eta||\nabla\eta|+2\int |\nabla(u\eta)||u||\nabla\eta|.
  \end{eqnarray*}
  With Young's inequality in form of
  \begin{eqnarray*}
    |F||u||\eta||\nabla\eta|&\leq& \frac{1}{2}|F|^2\eta^2+\frac{1}{2}u^2|\nabla\eta |^2,\\
    |F||\eta||\nabla(u\eta)|&\leq& \frac{1}{\lambda}|F|^2\eta^2+\frac{\lambda}{4}|\nabla(\eta u)|^2,\\
    2|\nabla(u\eta)||u||\nabla\eta|&\leq& \frac{4}{\lambda}u^2|\nabla\eta|^2+\frac{\lambda}{4}|\nabla(\eta u)|^2,
  \end{eqnarray*}
  we get
  \begin{eqnarray*}
    \frac{1}{T}\int |u\eta|^2+\frac{\lambda}{2}|\nabla(u\eta)|^2
    &\leq&c_\lambda\left(\int |F|^2\eta^2+\int u^2|\nabla\eta|^2\right),\qquad c_\lambda:=(\frac{4}{\lambda}+\frac{1}{2}).
  \end{eqnarray*}
  Let $R\geq T$. By an approximation argument (that exploits that $R\mapsto\fint_{R\Box}|u|^2$ grows at most polynomially), this estimate extends to the exponential cut-off function $\eta(x)=\exp(-c_0\frac{|x|}{\sqrt R})$, where  $c_0:=\frac{1}{2\sqrt{dc_\lambda}}$ to the effect of
  \begin{equation*}
    \frac{c_\lambda|\nabla\eta|^2}{\eta^2}\leq \frac{dc_\lambda c_0^2}{R}\leq \frac{1}{2R}\leq \frac{1}{2T}.
  \end{equation*}
  We conclude that
  \begin{eqnarray*}
    \int\frac{1}{2T}|u\eta|^2+\frac{\lambda}{2}|\nabla(u\eta)|^2
    &\leq&c_\lambda\int |F|^2\eta^2,
  \end{eqnarray*}
  and thus
  \begin{equation}\label{apriori:modified}
    \int (\frac{1}{T}|u|^2+|\nabla u|^2)\eta^2 \leq C(d,\lambda)\int |F|^2\eta^2.
  \end{equation}
  Since $\min_{\sqrt R\Box}\eta^2\geq \exp(-2c_0)>0$, we deduce that
  \begin{equation*}
    \fint_{\sqrt R\Box} (\frac{1}{T}|u|^2+|\nabla u|^2) \leq C(d,\lambda)R^{-\frac{d}{2}}\int |F|^2\eta^2.
  \end{equation*}
  On the other hand,  with $$\theta(x):=\left(\int\exp(-2c_0|y|)\,dy\right)^{-1}\exp(-\eta(-2c_0|x|),$$ we may estimate the right-hand side of the previous estimate by
  \begin{equation*}
    \int |F|^2\eta^2\leq C(d,\lambda)R^{\frac{d}{2}}\int |F(\sqrt R\cdot)|^2\theta(\cdot),
  \end{equation*}
and thus obtain \eqref{resolvent_apriori}.

  \step 2 Conclusion.

  Consider a general right-hand side $F$ with $\|F\|_{\theta}<\infty$. For $k\in\N$ set $F_k(x):={\bf 1}(|x|<k)F(x)$, which is a vector field in $L^2(\R^d,\R^d)$. Therefore, by the theorem of Lax-Milgram we find $u_k\in H^1(\R^d)$ that solves
  \begin{equation*}
    \frac{1}{T}u_k-\nabla\cdot a\nabla u_k=\nabla\cdot F_k,
  \end{equation*}
  and satisfies the standard a priori estimate,
  \begin{equation}\label{standardenergy}
    \frac{1}{T}\int u^2_k+\frac{\lambda}{2}|\nabla u_k|^2\leq \frac{2}{\lambda}\int|F_k|^2.
  \end{equation}
  In particular, $R\mapsto \fint_{R\Box}|u_k|^2$ is bounded and thus $u_k$ satisfies the a priori estimate of Step~1,
  \begin{equation*}
    \forall R\geq T\,:\qquad \fint_{\sqrt R\Box} (\frac{1}{T}|u_k|^2+|\nabla u_k|^2) \leq C(d,\lambda)\|F_k\|_{\theta}^2\leq C(d,\lambda)\|F\|_{\theta}^2,
  \end{equation*}
  which is uniform in $k$. Consider the nested sequence of cubes $Q_\ell:=2^\ell\sqrt T\Box$, $\ell\in\N_0$. By the a priori estimate we conclude that $(u_k)$ is bounded in $H^1(Q_\ell)$ for any $\ell\in\N_0$. Since $Q_\ell\subset Q_{\ell+1}$ and $Q_\ell\uparrow\R^d$, we conclude that there exists $u\in H^1_{\loc}(\R^d)$ such that $u_k\wto u$ weakly in $H^1(Q_\ell)$ for any $\ell\in\N_0$. Consequently $u$ solves \eqref{resolvent_equation} in a distributional sense. Thanks to the lower-semicontinuity of the norm, we deduce that $u$ satisfies the a priori estimate \eqref{resolvent_apriori}. This proves the existence of the solution. Uniqueness of $u$ is a consequence of the a priori estimate.
\end{proof}
\medskip

As already mentioned, in order to obtain an estimate that is uniform in $T$, we need to exploit stationarity of $\prob$ and random fields in the following sense:
\begin{definition}[Stationary random field]
A measurable function $u:\Omega\times\R^d\to\R$ is called a stationary $L^1$-random field (or short: stationary), if $\expec{\fint_Q|u|}<\infty$ for all cubes $Q\subset\mathcal Q$ and if for $\prob$-a.e.\ $a\in\Omega$,
\begin{equation}\label{stationary}
  \int_{x+Q}u(a,y)\,dy=\int_Qu(\tau_xa,y)\,dy\qquad\text{for all cubes }Q\in\mathcal Q\text{ and }x\in\R^d.
\end{equation}
\end{definition}
A prototypical example of a stationary random variable is as follows: Take $u_0\in L^1(\Omega)$ and consider $u(a,x):=u_0(\tau_x a)$. Then $u$ is a stationary $L^1$-random field,  called the stationary extension of $u_0$. One can easily check that for any $A\subset\R^d$ open and bounded we have
\begin{equation*}
  \expec{\fint_A u(a,y)\,dy}=\fint_A\expec{u_0(\tau_ya)}=\expec{u_0},
\end{equation*}
where the last identity holds by stationarity of $\prob$. In particular, we deduce that the value of $\expec{\fint_{Q} u(a,y)\,dy}$, $Q\in\mathcal Q$, is independent of $Q$. The same properties are true for general stationary $L^1$-random fields (except for the difference that we need to invoke an average w.r.t. $\R^d$-component to obtain well-defined quantities):
\begin{lemma}\label{L:avstat}
  Suppose $\prob$ is stationary. Let $f$ denote a stationary $L^1$-random field. 
  \begin{enumerate}[(a)]
  \item For any $A\subset \R^d$ open and bounded we have
    \begin{equation*}
      \expec{\fint_A f}=\expec{\fint_\Box f}.
    \end{equation*}
  \item Let $\rho>0$ and set $f_\rho(a):=\fint_{\rho\Box}f(a,y)\,dy$. Then $f_\rho\in L^1(\Omega)$ and for $\prob$-a.e.\ $a\in\Omega$,
    \begin{equation*}
      \fint_{\rho\Box}f(a,x+y)\,dy=f_\rho(\tau_xa)\qquad\text{for all }x\in\R^d.
    \end{equation*}
  \end{enumerate}
\end{lemma}
We postpone the proof to Section~\ref{S:auxproof}.
As a consequence of \eqref{ergodic} (i.e.\ Birkhoff's ergodic theorem), we obtain the following variant for stationary fields:
\begin{lemma}[Variant of Birkhoff's ergodic theorem]\label{L:ergodic}
  Let $f$ denote a stationary $L^1$-random field, then for $\prob$-a.e.\ $a\in\Omega$ we have
   \begin{equation*}
     \lim\limits_{R\to\infty}\fint_{R\Box}f(a,x)\,dx=\expec{\fint_\Box f}.
   \end{equation*}
  Moreover, if additionally $\expec{\fint_{\Box}|f|^2}<\infty$, then
   \begin{equation*}
     \lim\limits_{R\to\infty}\fint_{\Box}f(a,Rx)\eta(x)\,dx=\expec{\fint_\Box f}\fint_\Box\eta\,dx\qquad\text{for all }\eta\in L^2(\Box).
   \end{equation*}
\end{lemma}
We postpone the proof to Section~\ref{S:auxproof}.
\medskip

We turn back to the modified corrector equation which corresponds to the equation \eqref{resolvent_equation} with right-hand side $F(a,x):=a(x)\xi$. This random field (and by uniqueness the associated solution) is stationary. Hence, as a corollary of Lemma~\ref{L:resolvent_equation} and Lemma~\ref{L:avstat} we obtain:
\begin{corollary}\label{C:existence_modified}
  Suppose $\prob$ is stationary. Let $T\geq 1$ and $\xi\in\R^d$. Then there exists a unique stationary random field $\phi_T$ that solves the modified corrector equation
  \begin{equation}\label{eq:modcorr}
    \frac{1}{T}\phi_T-\nabla\cdot (a(\nabla\phi_T+\xi))=0\qquad\text{in }\mathcal D'(\R^d),\,\prob\text{-a.s.},
  \end{equation}
  and which satisfies the a priori estimate
  \begin{equation*}
    \expec{\fint_{\Box}\frac{1}{T}\phi_T^2+|\nabla\phi_T|^2}\leq C(d,\lambda)|\xi|^2.
  \end{equation*}
  Moreover, we have
  \begin{equation*}
    \expec{\fint_{\Box}\nabla\phi_T}=0.
  \end{equation*}
\end{corollary}
\begin{proof}
  
  \step 1 Existence and a priori estimate.

  For $a\in\Omega$ let $\phi_T(a,\cdot)$ denote the solution to \eqref{resolvent_equation} with $F=F(a,\cdot)=a(\cdot)\xi$ of Lemma~\ref{L:resolvent_equation}. By uniqueness of the solution we deduce that $\phi_T$ is stationary. Hence, by Lemma~\ref{L:avstat},
  \begin{equation*}
    \expec{\fint_{\Box}\frac{1}{T}\phi_T^2+|\nabla\phi_T|^2}\leq\expec{\fint_{\sqrt T\Box}\frac{1}{T}\phi_T^2+|\nabla\phi_T|^2}\leq C(d,\lambda)\expec{\|F(\sqrt T\cdot)\|^2_{\theta}},
  \end{equation*}
  and the estimate follows, since
  \begin{equation*}
    \expec{\|F(\sqrt T\cdot)\|^2_{\theta}}=\int\expec{|a(\sqrt T x)\xi|^2}\theta(x)\,dx\leq|\xi|^2.
  \end{equation*}

  \step 2 Zero expectation of the gradient.
  
  This is in fact a general property of stationary random fields $u$ satisfying $\expec{\fint_{\Box}u^2+|\nabla u|^2}<\infty$. Indeed, by stationarity and the divergence theorem we have
  \begin{eqnarray*}
    I&:=&\expec{\fint_{\Box}\partial_iu(a,x)\,dx} = \fint_{\Box}\expec{\fint_{\Box}\partial_iu(a,x+y)\,dx}\,dy\\
     &=&\frac{1}{|\Box|}\fint_{\Box}\expec{\int_{\partial\Box}u(a,x+y)\nu_i(x)\,dS(x)}\,dy,
  \end{eqnarray*}
  where $\nu(x)$ denotes the outer unit normal at $x\in\partial\Box$. By Fubini we may switch the order of the integration and get
  \begin{eqnarray*}
    I&=&\frac{1}{|\Box|}\int_{\partial\Box}\expec{\fint_{\Box}u(a,x+y)\,dy}\nu_i(x)\,dS(x)=\frac{1}{|\Box|}\expec{\fint_{\Box}u(a,y)\,dy}\int_{\partial\Box}\nu_i(x)\,dS(x)\\
     &=&0,
  \end{eqnarray*}
  where in the second last step we used stationarity, and in the last step $\int_{\partial\Box}\nu_i(x)\,dS(x)=0$.

\end{proof}
The estimate on $\nabla\phi_T$ of Corollary~\ref{C:existence_modified} is uniform $T$. Motivated by this we introduce a suitable function space in which we can pass to the limit $T\to\infty$. Since we can only pass to the limit on the level of the gradient, it is convenient to consider  $u_T=\phi_T-\fint_{\Box}\phi_T$, which satisfies $\fint_{\Box}u_T=0$, and  thus is uniquely determined by $\nabla u_T=\nabla\phi_T$. 
\begin{lemma}\label{L:spaceH}
  Suppose $\prob$ is stationary. Consider the linear space
  \begin{align*}
    \mathcal H:=\Big\{u:\Omega\times\R^d\to\R\,:\,&\expec{\fint_{\Box}|u|^2+|\nabla u|^2}<\infty,\,\expec{|\fint_{\Box}u|}=0,\,\nabla u\text{ is stationary}\, \Big\}.
  \end{align*}
  Then,
  \begin{enumerate}[(a)]
  \item for any cube $Q\in\mathcal Q$ we have
    \begin{equation*}
      \expec{\fint_Q|u|^2+|\nabla u|^2}\leq C(d,Q)\expec{\fint_{\Box}|\nabla u|^2}.
    \end{equation*}

  \item $\mathcal H$ equipped with the inner product
    \begin{equation*}
      (u,v)_{\mathcal H}:=\expec{\fint_{\Box}\nabla u\cdot\nabla v}
    \end{equation*}
    is a Hilbert space.
  \end{enumerate}
\end{lemma}
\begin{proof}
  
  \step 1 Proof of (a).

  We start with a deterministic estimate. Consider the dyadic family of cubes $Q_n=2^n\Box$, $n=0,1,\ldots$. We claim that
  \begin{equation}\label{spaceH:eq:1}
    \left(\fint_{Q_{n}}|u|^2\right)^{\frac{1}{2}}
    \leq C(d)\sum_{\ell=1}^n2^{\ell}\left(\fint_{Q_{\ell}}|\nabla u|^2\right)^{\frac{1}{2}}.
  \end{equation}
  Indeed, 
  \begin{eqnarray*}
    \left(\fint_{Q_{n}}|u|^2\right)^{\frac{1}{2}}
    &\leq&\left(\fint_{Q_{n}}|u-\fint_{Q_{n-1}} u|^2\right)^{\frac{1}{2}}+\big|\fint_{Q_{n-1}}u\big|\\
    &\leq&
           \left(\fint_{Q_{n}}|u-\fint_{Q_{n-1}} u|^2\right)^{\frac{1}{2}}+\left(\fint_{Q_{n-1}}|u|^2\right)^\frac12\\
    &\leq&
           \sum_{\ell=1}^n\left(\fint_{Q_{\ell}}|u-\fint_{Q_{\ell-1}} u|^2\right)^{\frac{1}{2}}+\underbrace{\big|\fint_{\Box}u\big|}_{=0}\\
    &\leq&C(d)\sum_{\ell=1}^n2^{\ell}\left(\fint_{Q_{\ell}}|\nabla u|^2\right)^{\frac{1}{2}}.
  \end{eqnarray*}
  Suppose $u\in\mathcal H$.  Taking the square and expectation of \eqref{spaceH:eq:1}, and exploiting stationarity in form of
  \begin{equation*}
    \expec{\fint_{Q_\ell}|\nabla u|^2}=\expec{\fint_{\Box}|\nabla u|^2},
  \end{equation*}
  yields
  \begin{equation}\label{spaceH:eq:2}
    \expec{\fint_{Q_{n}}|u|^2}
    \leq C(d)n\sum_{\ell=1}^n2^{2\ell}\expec{\fint_{Q_{\ell}}|\nabla u|^2}\leq C(d,n)\expec{\fint_{\Box}|\nabla u|^2}.
  \end{equation}
  Now, let $Q$ denote an arbitrary cube. Then we have $Q\subset Q_n$ for some $n\in\N$, and thus
  \begin{equation*}
    \expec{\fint_Q|u|^2+|\nabla u|^2}=\expec{\fint_Q|u|^2}+\expec{\fint_{\Box}|\nabla u|^2}\leq C'(d,n)\expec{\fint_{\Box}|\nabla u|^2}.
  \end{equation*}
  
  \step 2  $\mathcal H$ is Hilbert.

  Obviously $(\cdot,\cdot)_{\mathcal H}$ turns $\mathcal H$ into an inner product space and the definiteness of the norm follows from (a). We argue that  $(\mathcal H,\|\cdot\|_{\mathcal H})$ is complete. First note that by stationarity of $\nabla u$, we have for all $n\in\N$,
  \begin{equation}\label{L:spaceH:eq1}
    \expec{\fint_{Q_n}|\nabla u|^2}=    \expec{\fint_{\Box}|\nabla u|^2}=\|u\|_{\mathcal H}^2,
  \end{equation}
  and thus by Step 1,
  \begin{equation*}
    \expec{\fint_{Q_n}|u|^2+|\nabla u|^2}\leq C(d,n)\|u\|_{\mathcal H}^2.
  \end{equation*}
  Let $(u_k)$ denote a Cauchy sequence in $\mathcal H$. Then the previous estimate implies that $(u_k)$ is Cauchy in any of the spaces $L^2(\Omega,H^1(Q_n))$, $n\in\N$. Thus, $u_k\to u^{(n)}$ in $L^2(\Omega,H^1(Q_n))$ for all $n\in\N_0$. For $\ell\leq n$, we have $Q_{\ell}\subset Q_{n}$, and thus $u^{(\ell)}=u^{(n)}$ on $\Omega\times Q_\ell$.  We conclude that there exists a random field $u$ with $u\in L^2(\Omega,H^1(Q))$ for all cubes $Q\in\mathcal Q$, and $u_k\to u$ in $L^2(\Omega, H^1(Q))$ for any $Q\in\mathcal Q$. This in particular implies that $\expec{|\fint_{\Box}u|}=0$. To conclude $u\in\mathcal H$ it remains to argue that $\nabla u$ is stationary. It suffices to show for any $\varphi\in L^2(\Omega)$, $Q\in\mathcal Q$ and $x\in\R^d$,
  \begin{equation*}
    \expec{\fint_{x+Q}\partial_iu(a,y)\,dy\varphi(a)}=    \expec{\fint_{Q}\partial_iu(\tau_xa,y)\,dy\varphi(a)}.
  \end{equation*}
  Since $\partial_iu_k$ is stationary, this identity is satisfied for $u$ replaced by $u_k$. Since $\partial_iu_k\to \partial_iu$ in $L^2(\Omega\times Q)$ for any $Q\in\mathcal Q$, the identity also holds for $\partial_iu$. 
\end{proof}

As a corollary of Lemma~\ref{L:spaceH}, Corollary~\ref{C:existence_modified} and Lemma~\ref{L:sublinear} we obtain the existence { and uniqueness} of the sublinear corrector:
{
\begin{corollary}[Uniqueness of the sublinear corrector]\label{C:unique_sublin}
  Suppose $\prob$ is stationary and ergodic. Then there exists at most one $\phi\in\mathcal H$ satisfying the corrector equation \eqref{eq:random-corr} and the sublinear growth condition \eqref{sublin} $\prob$-a.s. 
\end{corollary}
\begin{proof}
  Let $\phi,\phi'\in\mathcal H$ be two sublinear solutions to the corrector equation and consider $u:=\phi-\phi'$. Then $\prob$-a.s. $u$ satisfies the assumptions of Lemma~\ref{L:unique-const} and we conclude
  \begin{equation*}
    \lim_{R\to\infty}\fint_{R\Box}|\nabla u|^2 = 0.
  \end{equation*}
  On the other hand, by stationarity of $\nabla u$ and ergodicity we have
  \begin{equation*}
    \expec{\fint_{\Box}|\nabla u|^2}=\lim_{R\to\infty}\fint_{R\Box}|\nabla u|^2 = 0,
  \end{equation*}
  and thus $u$ is constant $\prob$-a.s. Since $\fint_\Box u=0$, we conclude that $u=0$.
\end{proof}
}
\begin{corollary}[Existence of the sublinear corrector]\label{C:existence_sublin}
  Suppose $\prob$ is stationary and ergodic. Let $\phi_T$ denote the solution to the modified corrector equation \eqref{eq:modcorr} of Corollary~\ref{C:existence_modified}. Then there exists $\phi\in\mathcal H$ such that $u_T:=\phi_T-\fint_{\Box}\phi_T\wto \phi$ weakly in $\mathcal H$ (for $T\to\infty$), and $\phi$ is the unique solution to the corrector equation in the sense {of Corollary~\ref{C:unique_sublin}}.
\end{corollary}
\begin{proof}
  By Corollary~\ref{C:existence_modified} $(u_T)$ is a bounded sequence in $\mathcal H$. Since $\mathcal H$ is Hilbert, we may pass to a subsequence (not relabeled) such that $u_T\wto \phi$ weakly in $\mathcal H$. We claim that $\phi$ solves the corrector equation \eqref{eq:random-corr}. Let $\eta\in C^\infty_c(\R^d)$ and $\varphi\in L^2(\Omega)$ denote test functions and let $Q\in\mathcal Q$ denote a cube centered at $0$ with $\supp\eta\subset Q$. From $u_T\wto\phi$ weakly in $\mathcal H$, we infer that $\nabla\phi_T\wto \nabla\phi$ weakly in $L^2(\Omega\times Q)$, and thus
  \begin{equation*}
    \expec{\varphi\int a(\nabla\phi+\xi)\cdot\nabla\eta}=\lim\limits_{T\to\infty}\expec{\varphi\int a(\nabla\phi_T+\xi)\cdot\nabla\eta}= -\lim\limits_{T\to\infty}\expec{\varphi\int \frac{1}{T}\phi_T\eta}.
  \end{equation*}
  Note that
  \begin{equation}\label{eq:solve-corr}
    |\int \frac{1}{T}\phi_T\eta|\leq|Q|\left(\frac{1}{T}\fint_Q\phi_T^2\right)^{\frac12}\left(\frac{1}{T}\fint_Q\eta^2\right)^\frac12.
  \end{equation}
  By stationarity and the a priori estimate of Corollary~\ref{C:existence_modified} we have
  \begin{equation*}
    \expec{\varphi\int \frac{1}{T}\phi_T\eta}\leq T^{\frac12}|Q|\left(\fint_Q\eta^2\right)^\frac12\expec{\varphi^2}^\frac12 \expec{\frac{1}{T}\fint_\Box\phi_T^2}^{\frac12}\to 0,
  \end{equation*}
  and thus we deduce with \eqref{eq:solve-corr} that
  \begin{equation*}
    \expec{\varphi\int a(\nabla\phi+\xi)\cdot\nabla\eta}=0.
  \end{equation*}
  Since the test functions are arbitrary, \eqref{eq:random-corr} follows. Since $\phi\in \mathcal H$, we have $\expec{\fint_{\Box}|\nabla\phi|^2}<\infty$, $\expec{\fint_{\Box}\nabla\phi}=0$, and $\expec{|\fint_{\Box}\phi|}=0$. By ergodicity, which we use in form of Lemma~\ref{L:ergodic}, we find that the assumptions of Lemma~\ref{L:sublinear} are satisfied $\prob$-a.s. Hence, $\phi(a,\cdot)$ is sublinear in the sense of \eqref{sublin} $\prob$-a.s., {and a solution to \eqref{eq:random-corr}. By uniqueness  of the solution (cf.~Corollary~\ref{C:unique_sublin}) we conclude that} $\phi$ is independent of the subsequence, and we deduce that $u_T\wto\phi$ in $\mathcal H$ for the entire sequence. 
\end{proof}
Note that Corollary~\ref{C:existence_sublin} proves Proposition~\ref{P:corr-stoch} except for the a priori estimate
\begin{equation}\label{apriori:inH}
  \expec{\fint_{\Box}|\nabla\phi|^2}\leq\frac{1-\lambda^2}{\lambda^2}\expec{\fint_{\Box}|a\xi|^2},
\end{equation}
whose argument we postpone to the end of this section. In fact, the estimate $\expec{\fint_{\Box}|\nabla\phi|^2}\leq C(d,\lambda)|\xi|^2$ (for some constant $C(d,\lambda)<\infty$) follows (by lower semicontinuity) directly from the a priori estimate in Corollary~\eqref{C:existence_modified}.
 The sublinear corrector of Proposition~\ref{P:corr-stoch} can alternatively be characterized as the unique solution to an abstract variational problem in the Hilbert space $\mathcal H$. (This formulation also entails a short argument for \eqref{apriori:inH}). In the rest of this section, we discuss this alternative formulation. We start with the observation that the space of stationary $H^1$-random fields forms a Hilbert space:
\begin{lemma}\label{L:spaceS}
  Suppose $\prob$ is stationary.
  Consider the linear space
  \begin{equation*}
    \mathcal S:=\Big\{u\text{ is a stationary random field with }\expec{\fint_\Box|u|^2+|\nabla u|^2}<\infty\, \Big\}.
  \end{equation*}
  Then $\mathcal S$ with inner product
  \begin{equation*}
    (u,v)_{\mathcal S}:=\expec{\fint_{\Box}uv+\nabla u\cdot\nabla v}
  \end{equation*}
  is a Hilbert space. Moreover, for any $u\in\mathcal S$ we have $\expec{\fint_{\Box}\nabla u}=0$. 
\end{lemma}
\begin{proof}
  Obviously $(\cdot,\cdot)_{\mathcal S}$ turns $\mathcal S$ into an inner product space.  We argue that $(\mathcal S,\|\cdot\|_{\mathcal S})$ is complete and first note that
  for any $u\in\mathcal S$, the stationarity of $u$ implies stationarity of $\nabla u$, and thus for all $Q_n:=2^n\Box$, $n\in\N$, we have
  \begin{equation}\label{L:spaceS:eq1}
    \expec{\fint_{Q_n}u^2+|\nabla u|^2}=    \expec{\fint_{\Box}u^2+|\nabla u|^2}=\|u\|_{\mathcal S}^2.
  \end{equation}
  The remaining argument for completeness is similar to the proof of Lemma~\ref{L:spaceH}. The fact that gradients of stationary random fields are mean-free has already been proven in Step~2 in the proof of Corollary~\ref{C:existence_modified}.
\end{proof}
Next we observe that on the level of the gradient any function $u\in \mathcal H$ can be approximated by functions in $\mathcal S$. With help of this observation we can pass from distributional equations on $\R^d$ to problems in $\mathcal H$ (and vice versa):
\begin{lemma}\label{L:spaceH0}
  Suppose $\prob$ is stationary and ergodic.
  \begin{enumerate}[(a)]
  \item For any $u\in\mathcal H$ we can find a sequence $u_T\in\mathcal S$  such that  $u_T-\fint_\Box u_T\wto u$ weakly in $\mathcal H$.
  \item Let $F$ be a stationary random vector field with $\expec{\fint_\Box|F|^2}<\infty$. Then the following are equivalent
    \begin{align}
      \label{L:sapceH01}
        \expec{\fint_{\Box}F\cdot\nabla \varphi}=&\,0\qquad\text{for all }\varphi\in\mathcal H,\\
      \label{L:sapceH02}
        -\nabla\cdot F=&\,0\qquad\text{in }\mathcal D'(\R^d),\
        \prob\text{-a.s.}.
    \end{align}
  \end{enumerate}
\end{lemma}
\begin{proof}[Proof of Lemma~\ref{L:spaceH0}]
  \step 1

  Let $F$ denote a stationary vector field with $\expec{\fint_\Box|F|^2}<\infty$, let $T\geq 1$. We claim that there exists a unique $u_T\in\mathcal S$ such that
  \begin{equation}\label{L:spaceH0:dist}
    \frac{1}{T}u_T-\triangle u_T=\nabla\cdot F\qquad\text{in }\mathcal D'(\R^d),\ \prob\text{-a.s.},
  \end{equation}
  and that $u_T$ is characterized by the weak equation
  \begin{equation}\label{L:spaceH0:weak}
    \expec{\fint_{\Box}\frac{1}{T}u_T\varphi+\nabla u_T\cdot\nabla\varphi}=-\expec{\fint_{\Box}F\cdot\nabla\varphi}\qquad\text{for all }\varphi\in\mathcal S.
  \end{equation}
  We first argue that a solution $u_T\in\mathcal S$ to \eqref{L:spaceH0:dist} exists.
  Note that by stationarity we have for all $R\geq 1$,
  \begin{equation*}
    \expec{\|F(\sqrt R\cdot)\|_{\theta}^2}=\expec{\fint_{\Box}|F|^2}.
  \end{equation*}
  Thus, by Lemma~\ref{L:resolvent_equation}, there exists a unique random field $u_T$ that satisfies \eqref{L:spaceH0:dist} and the a priori bound \eqref{resolvent_apriori} $\prob$-a.s. Since $F$ is stationary, $u_T$ and $\nabla u_T$ are stationary, and thus the a priori bound turns into
  \begin{equation}\label{eq:Leq3}
\expec{\fint_{\Box}\frac{1}{T}u_T^2+|\nabla u_T|^2}\leq C(d,\lambda)\expec{\fint_{\Box}|F|^2}.
  \end{equation}
  On the other hand, the Lax-Milgram Theorem yields a unique solution $v_T\in\mathcal S$ to the weak formulation \eqref{L:spaceH0:weak}. In order to conclude that both formulations are equivalent, it suffices to show that $u_T$ solves \eqref{L:spaceH0:weak}. For the argument let $\varphi\in\mathcal S$ and $\eta\in C^\infty_c(\Box)$ be arbitrary test functions. It suffices to show
  \begin{equation*}
    I:=\expec{\fint_{\Box}\frac{1}{T}u_T\varphi+(\nabla u_T+F)\cdot\nabla\varphi}=0.
  \end{equation*}
For $R\geq 1$ set $\varphi_R:=\frac{1}{R}\varphi(R\cdot)$, $u_{T,R}:=\frac{1}{R}u_{T,R}(R\cdot)$, and $F_R:=F(R\cdot)$. Then by stationarity and scaling we have
  \begin{equation*}
    I=\expec{\fint_{\Box}\frac{R^2}{T}u_{T,R}\varphi_R+(\nabla u_{T,R}+F_R)\cdot\nabla\varphi_R},
  \end{equation*}
  and by \eqref{L:spaceH0:dist},
  \begin{equation*}
    \expec{\fint_{\Box}\frac{R^2}{T}u_{T,R}(\varphi_R\eta)+(\nabla u_{T,R}+F_R)\cdot\nabla(\varphi_R\eta)}=0.
  \end{equation*}
  The difference of the previous two equations is given by
  \begin{equation*}
    \expec{\fint_{\Box}\frac{R^2}{T}u_{T,R}\varphi_R(1-\eta)}+\expec{(\nabla u_{T,R}+F_R)\cdot\big(\nabla\varphi_R-\nabla(\varphi_R\eta)}=:II+III.    
  \end{equation*}
  By Cauchy-Schwarz, stationarity, and the a priori estimate \eqref{eq:Leq3},
  \begin{eqnarray*}
    |II|&\leq& \frac{1}{\sqrt T}\expec{\fint_{\Box}\frac{1}{T}|u_{T}(R\cdot)|^2}^{\frac12}\expec{\fint_{\Box}|\varphi(R\cdot)|^2|1-\eta|^2}^{\frac12}\\
    &\leq&C(d)\frac{1}{\sqrt T}\expec{\fint_\Box|F|^2}^\frac12\expec{\fint_{\Box}\varphi^2}^\frac12\left(\int_\Box|1-\eta|^2\right)^\frac12\\
    &\leq&C(d,T,F,\varphi)\|1-\eta\|_{L^2(\Box)}.
  \end{eqnarray*}
  Regarding $III$ we note that
  \begin{eqnarray*}
    |III|&\leq&\expec{\fint_{\Box}|\nabla u_R+F_R|(|\nabla \varphi_R||1-\eta|+|\varphi_R|||\nabla \eta|)}
  \end{eqnarray*}
Arguing as above, we deduce that
\begin{equation*}
  |III|\leq C(d,F,\varphi)\left(\|1-\eta\|_{L^2(\Box)}+\|\nabla\eta\|_{L^\infty(\Box)}\expec{\|\varphi_R\|_{L^2(\Box)}^2}^\frac12\right).
\end{equation*}
Note that by stationarity we have
\begin{equation*}
  \expec{\|\varphi_R\|_{L^2(\Box)}^2}=R^{-2}\expec{\fint_{R\Box}|\varphi|^2}=R^{-2}\expec{\fint_{\Box}|\varphi|^2}\to 0.
\end{equation*}
In conclusion we deduce that
\begin{equation*}
|I|\leq   \limsup\limits_{R\to\infty}(|II|+|III|)\leq C(d,T,F,\varphi)\|1-\eta\|_{L^2(\Box)}.
\end{equation*}
Since $\eta$ is arbitrary, the right-hand side can be made arbitrarily small, and thus $I=0$.

  \step 2 Proof of (a).

  Let $u\in\mathcal H$, set $F(a,x):=-\nabla u(a,x)$, and let $u_T$ denote the unique solution in $\mathcal S$ to \eqref{L:spaceH0:dist}. From \eqref{L:spaceH0:weak} we obtain the a priori estimate
  \begin{equation}\label{a:apriori}
    \expec{\fint_{\Box}\frac{1}{T}|u_T|^2+\frac{1}{2}|\nabla u_T|^2}\leq \frac{1}{2}\expec{\fint_{\Box}|\nabla u|^2},
  \end{equation}
  which for the gradient is uniform in $T\geq 1$. We conclude that $v_T:=u_T-\fint_{\Box} u_T$ defines a bounded sequence in $\mathcal H$. Let $v\in\mathcal H$ denote a weak limit of $(v_T)$ along a subsequence $T\to\infty$ (that we do not relabel). We claim that $v=u$ (which implies that the convergence holds for the entire sequence). First notice that it suffices to show that for all $\varphi\in L^2(\Omega)$ and $\eta\in C^\infty_c(\R^d)$ we have
  \begin{equation}\label{eq:Leq2}
    \expec{\varphi\int(\nabla v-\nabla u)\cdot\nabla\eta}=0.
  \end{equation}
  Indeed, this implies that $w=v-u$ satisfies $-\triangle w=0$ in $\mathcal D'(\R^d)$, $\prob$-a.s. Since $w\in\mathcal H$ has sublinear growth, we conclude with Lemma~\ref{L:unique-const} that $w$ is constant. Since $\fint_{\Box}w=0$ by construction, we deduce that $w=v-u=0$. We prove \eqref{eq:Leq2}. Since $\nabla v$ is a weak limit of $\nabla v_T=\nabla u_T$, it suffices to show that
  \begin{equation*}
    I:=\expec{\varphi\fint_{Q}(\nabla u_T-\nabla u)\cdot\nabla\eta}\to 0 \qquad\text{for }T\to\infty,
  \end{equation*}
  where $Q\in\mathcal Q$ is a cube that contains the support of $\eta$. Since $u_T$ solves \eqref{L:spaceH0:dist} with $F=-\nabla u$, we have
  \begin{equation*}
    I=-\expec{\varphi\fint_{Q}\frac{1}{T}u_T\eta},
  \end{equation*}
  which for $T\to\infty$ converges to $0$, thanks to the a priori estimate \eqref{a:apriori} and stationarity.
  \medskip
  
  \step 4 Proof (b).

  First note that \eqref{L:sapceH01}, thanks to (a), is equivalent to
  \begin{equation}\label{eq:eq55}
    \expec{\fint_{\Box}F\cdot\nabla\varphi}=0\qquad\text{for all }\varphi\in\mathcal S.
  \end{equation}
  Let $u_T\in\mathcal S$ denote the unique solution to 
  \begin{equation*}
    \frac{1}{T}u_T-\triangle u_T=-\nabla\cdot F\qquad\text{in }\mathcal D'(\R^d),\,\prob\text{-a.s.},
  \end{equation*}
  which exists thanks to Step 2, and is equivalent to
  \begin{equation}\label{eq:eq44}
    \expec{\fint_{\Box}\frac{1}{T}u_T\varphi+(\nabla u_T-F)\cdot\nabla\varphi}=0\qquad\text{for all }\varphi\in\mathcal S.
  \end{equation}
  Then for all $T\geq 1$, \eqref{L:sapceH02} is equivalent to $u_T=0$. On the other hand, in view of \eqref{eq:eq44}, $u_T=0$ implies \eqref{eq:eq55}, and \eqref{eq:eq55} implies
  \begin{equation*}
    \expec{\fint_{\Box}\frac{1}{T}u_T\varphi}=0\qquad\text{for all }\varphi\in\mathcal S,
  \end{equation*}
  and thus $u_T=0$. 
\end{proof}
Finally, we present the characterization of $\phi$ and $\phi_T$ by means of variational problems in the Hilbert spaces $\mathcal H$ and $\mathcal S$:
\begin{lemma}
  Suppose $\prob$ is stationary and ergodic. Let $\phi$ denote the sublinear corrector associated with $\xi\in\R^d$ of Proposition~\ref{P:corr-stoch}, and $\phi_T$ the unique modified corrector associated with $\xi\in\R^d$ of Corollary~\ref{C:existence_modified}. Then $\phi\in\mathcal H$ and $\phi_T\in\mathcal S$ are uniquely characterized by
  \begin{eqnarray}\label{corr:variational}
    \expec{\fint_{\Box}a(\nabla\phi+\xi)\cdot\nabla\varphi}&=&0\qquad\text{for all }\varphi\in \mathcal H,\\
    \label{corr-mod:variational}
    \expec{\fint_{\Box}\frac{1}{T}\phi_T\varphi+a(\nabla\phi_T+\xi)\cdot\nabla\varphi}&=&0\qquad\text{for all }\varphi\in \mathcal S,
  \end{eqnarray}
  and we have
  \begin{equation*}
    \lim\limits_{T\to\infty}\expec{\fint_{\Box}|\nabla\phi_T-\nabla\phi|^2}= 0.
  \end{equation*}
  Moreover, \eqref{apriori:inH} holds.
\end{lemma}
\begin{proof}
  First note that the variational equations for $\phi$ and $\phi_T$ in $\mathcal H$ and $\mathcal S$, respectively, admit a unique solution by the Theorem of Lax-Milgram. The equivalence of the formulations for $\phi$ follows from Lemma~\ref{L:spaceH0} (b). The equivalence of the formulation for $\phi_T$ follows by the argument in Step~1 in the proof of Lemma~\ref{L:spaceH0}. (We only need to replace $-\triangle$ by $-\nabla\cdot(a\nabla)$ and $F$ by $a\xi$). For the convergence statement it is convenient to work with the variational equations:
  \begin{eqnarray*}
    &&\expec{\fint_{\Box}(\nabla\phi-\nabla\phi_T)\cdot a(\nabla\phi-\nabla\phi_T)}=\expec{\fint_{\Box}(\nabla\phi-\nabla\phi_T)\cdot a(\nabla\phi+\xi)}\\
    &&
    -\expec{\fint_{\Box}(\nabla\phi\cdot a(\nabla\phi_T+\xi)}+\expec{\fint_{\Box}(\nabla\phi_T\cdot a(\nabla\phi_T+\xi)}\\
    &=&
    -\expec{\fint_{\Box}(\nabla\phi\cdot a(\nabla\phi_T+\xi)}-\frac{1}{T}\expec{\fint_{\Box}\phi_T^2}.
  \end{eqnarray*}
  Hence, 
  \begin{eqnarray*}
    &&\limsup\limits_{T\to\infty}\expec{\fint_{\Box}(\nabla\phi-\nabla\phi_T)\cdot a(\nabla\phi-\nabla\phi_T)}\leq -\lim\limits_{T\to\infty}\expec{\fint_{\Box}(\nabla\phi\cdot a(\nabla\phi_T+\xi)}\\
    &=&-\expec{\fint_{\Box}\nabla\phi\cdot a(\xi+\nabla\phi)}=0,
  \end{eqnarray*}
  and the claim follow by ellipticity of $a$. The a priori estimate for $\nabla\phi$ easily follows from the variational formulation of the corrector equation:
  We first note that by ellipticity and \eqref{corr:variational}, we have
\begin{eqnarray*}
  &&\lambda\expec{\fint_{\Box}|\nabla\phi+\xi|^2} \leq\expec{\fint_{\Box}(\nabla\phi+\xi)\cdot a(\nabla\phi_T+\xi)} =\xi\cdot\expec{\fint_{\Box}a(\nabla\phi_T+\xi)}\\
  &\leq&\frac{1}{2\lambda}|\xi|^2+\frac{\lambda}{2}\expec{\fint_{\Box}|\nabla\phi_T+\xi|^2},
\end{eqnarray*}
and thus
\begin{equation*}
\frac{\lambda}{2}\expec{\fint_{\Box}|\nabla\phi+\xi|^2} \leq\frac{1}{2\lambda}|\xi|^2.
\end{equation*}
On the other hand,
\begin{eqnarray*}
  \expec{\fint_{\Box}|\nabla\phi+\xi|^2}=    \expec{\fint_{\Box}|\nabla\phi|^2}+|\xi|^2,
\end{eqnarray*}
since the cross-term $\expec{\fint_{\Box}\nabla\phi\cdot \xi}=0$, thanks to $\expec{\fint_{\Box}\nabla\phi}=0$. Thus, \eqref{apriori:inH} follows from the combination of these estimates.
\end{proof}
Note that Corollary~\ref{C:existence_sublin} combined with \eqref{apriori:inH}, which follows from the previous lemma, completes the proof of Proposition~\ref{P:corr-stoch}.
As a corollary of the previous lemma, and in analogy to Lemma~\ref{L:ahom}, we have:
\begin{lemma}[Properties of the homogenized coefficients]\label{L:ahomstoch}
Suppose Assumption (S) is satisfied and let $\phi_1,\ldots,\phi_d$ denote the correctors associated with $e_1,\ldots,e_d$. Set
\begin{equation*}
  a_{\hom}e_i:=\expec{\fint_{\Box}a(\nabla\phi_i+e_i)}.
\end{equation*}
Then:
\begin{enumerate}[(a)]
\item (ellipticity). For any $\xi\in\R^d$ we have
  \begin{equation*}
    \xi\cdot a_{\hom}\xi\geq \lambda|\xi|^2.
  \end{equation*}
\item (invariance under transposition). Let $\phi_i^t$ denote the corrector associated with the transposed matrix  $a^t$. Then
  \begin{equation*}
    (a_{\hom})^te_i=\expec{\fint_{\Box}a^t(\nabla\phi_i^t+e_i)}.
  \end{equation*}
\item (symmetry). If $a$ is symmetric (a.e.\ in $\R^d$ and $\prob$-a.s.), then $a_{\hom}$ is symmetric.
\end{enumerate}
\end{lemma}
The proof is similar to the proof of Lemma~\ref{L:ahom}. We leave it to the reader.
\begin{remark}[Systems]\label{R:systems}
  The arguments that we presented in this section (in particular the construction of the sublinear corrector and the proof of Theorem~\ref{Hom:T2} extend to systems of the form
  \begin{equation*}
    -\nabla\cdot a \nabla u=F,
  \end{equation*}
  with $u:\R^d\to H$ taking values in a finite dimensional Euclidean space $H$. The matrix field $a:\R^d\to \operatorname{Lin}(H^d,H^d)$ is required to be bounded and uniformly elliptic in the integrated form of
  \begin{equation*}
  \int \nabla\zeta\cdot a\nabla\zeta\geq\lambda\int|D\zeta|^2,\qquad \text{for all }\zeta\in C^\infty_c(\R^d,H).
\end{equation*}
In particular, this includes the relevant case of linear elasticity, when $H=\R^d$ and $a:\R^d\to\operatorname{Lin}(\R^{d\times d},\R^{d\times d})$ is Korn-elliptic, i.e.
\begin{equation*}
  \xi\cdot a(x)\xi\geq|\sym \xi|^2\qquad\text{for a.e.\ }x\in\R^d\text{ and }\xi\in\R^{d\times d}.
\end{equation*}
\end{remark}

\subsubsection{Proof of Lemma~\ref{Lstochhom},  Lemma~\ref{L:avstat}, and Lemma~\ref{L:ergodic}}\label{S:auxproof}
\begin{proof}[Proof of Lemma~\ref{Lstochhom}]
Set $j_{i}:=a^t(\nabla\phi_i^t+e_i)$ and note that
\begin{equation*}
  j_i(\tfrac{x}{\e})=a^t_\e\nabla g_{i,\e}.
\end{equation*}
Note that by the corrector equation, we have $-\nabla\cdot j_i=0$, and thus property \eqref{Hom:Eq3} holds. Since $j_i$ a stationary random field with $\expec{\fint_{\Box}|j_i|^2}<\infty$, Birkhoff's ergodic theorem in form of Lemma~\ref{L:ergodic} implies that $j_i(\tfrac{\cdot}{\e})\wto \expec{\fint_{\Box}j_i}=a_{\hom}e_i$, and thus property \eqref{Hom:Eq5} is satisfied. Finally, since $\phi_i^t$ has sublinear growth, we deduce that
\begin{equation*}
  \fint_{Q}|\e\phi_{i}^t(\tfrac{\cdot}{\e})|^2\to 0\qquad\text{for all }Q\in\mathcal Q\text{ and }\prob\text{-a.s.},
\end{equation*}
and thus property \eqref{Hom:Eq4} is satisfied.
\end{proof}

\begin{proof}[Proof of Lemma~\ref{L:avstat}]
  \step 1 Proof of (a).

  We first claim that for any $Q\in\mathcal Q$ centered at $0$, and any odd $\ell\in\N$ we have
  \begin{equation}\label{L:avstat:eq1}
    \expec{\fint_{\ell Q}f}=    \expec{\fint_{Q}f}.
  \end{equation}
  Indeed, since with $s$ denoting the side length of $Q$, we have $\ell Q=\cup_{x\in s\Z^d\cap \ell Q}(x+Q)$, up to a set of zero measure, we get by stationarity of $f$ and $\prob$:
  \begin{eqnarray*}
    \expec{\fint_{\ell Q}f}&=&\sum_{x\in s\Z^d\cap \ell Q}\frac{|x+Q|}{|\ell Q|}\expec{\fint_{x+Q}f}=\ell^{-d}\sum_{x\in s\Z^d\cap \ell Q}\expec{\fint_{Q}f(\tau_xa,y)\,dy}\\
    &=&\ell^{-d}\sum_{x\in s\Z^d\cap \ell Q}\expec{\fint_Qf(a,y)\,dy}=\expec{\fint_Qf}.
  \end{eqnarray*}
  
  Next we prove (a) for any $Q\in\mathcal Q$ with $Q$ centered at $0$.
  W.l.o.g. we may assume that $f\geq 0$. (Otherwise decompose $f$ in its positive and negative part, which remain stationary). For $\ell\in\N_0$ let $\ell^-$ (and $\ell^+$) denote the largest (smallest) odd non-negative integer satisfying
  \begin{equation*}
    \ell^-\Box\subset \ell Q\subset \ell^+\Box,
  \end{equation*}
  and note that
  \begin{equation*}
    \frac{|\ell^{\pm}\Box|}{|\ell Q|}\to 1\qquad\text{as }\ell\to\infty.
  \end{equation*}
  Thus
  \begin{eqnarray*}
    \int_{\ell^-\Box}f\leq \int_{\ell Q} f\leq \int_{\ell^+\Box} f,
  \end{eqnarray*}
  dividing by $|\ell Q|$ and taking the expectation yields
  \begin{eqnarray*}
    \frac{|\ell^{-}\Box|}{|\ell Q|}\expec{\fint_{\ell^-\Box}f}\leq \expec{\fint_{\ell Q} f}\leq \frac{|\ell^{+}\Box|}{|\ell Q|}\expec{\int_{\ell^+\Box} f},
  \end{eqnarray*}
  By \eqref{L:avstat:eq1} we have
  \begin{equation*}
    \expec{\fint_{\ell^-\Box}f}=    \expec{\fint_{\ell^+\Box}f}=\expec{\fint_{\Box} f},
  \end{equation*}
  and
  \begin{equation*}
    \expec{\fint_{Q}f}=\expec{\fint_{\ell Q}f},
  \end{equation*}
  and thus
  \begin{eqnarray*}
    \frac{|\ell^{-}\Box|}{|\ell Q|}\expec{\fint_{\Box}f}\leq \expec{\fint_{Q} f}\leq \frac{|\ell^{+}\Box|}{|\ell Q|}\expec{\int_{\Box} f}.
  \end{eqnarray*}
  Taking the limit $\ell\to\infty$ yields (a) for centered cubes. The conclusion for an arbitrary cube $Q\in\mathcal Q$ follow directly from the stationarity of $f$ and $\prob$: If we choose $x\in\R^d$ such that $x+Q$ is centered, then
  \begin{equation*}
    \expec{\fint_Qf}=\expec{\fint_{x+Q}f(\tau_{-x}a,y)\,dy}=\expec{\fint_{\Box}f}.
  \end{equation*}
  The statement for an arbitrary open, bounded set $A\subset\R^d$ follows from Whitney's covering theorem: There exists a countable, disjoint family of cubes $Q_j$ s.t. $\cup_{j}\bar Q_j=A$, and thus
  \begin{equation*}
    \expec{\fint_A f}=\frac{1}{|A|}\sum_{j}|Q_j|\expec{\fint_{Q_j}f}=\expec{\fint_\Box f}\frac{1}{|A|}\sum_{j}|Q_j|=\expec{\fint_\Box f}.
  \end{equation*}

\step 2 Proof of (b).

By Fubini's theorem we have $f_\rho\in L^1(\Omega)$, and by stationarity we have
\begin{equation*}
  \fint_{\rho\Box} f(a,x+y)\,dy=  \fint_{x+\rho\Box} f(a,y)\,dy= \fint_{\rho\Box} f(\tau_xa,y)\,dy=f_\rho(\tau_xa).
\end{equation*}
\end{proof}

\begin{proof}[Proof of Lemma~\ref{L:ergodic}]
  \step 1 Proof of the first statement.

  W.l.o.g. we may assume that $f\geq 0$.
  Let $\rho>0$. Since $f_\rho(a):=\fint_{\rho\Box}f(a,y)\,dy\in L^1(\Omega)$, we have by \eqref{ergodic}, and Lemma~\ref{L:avstat} (a),
  \begin{equation}\label{L:ergodic:p1}
    \lim\limits_{R\to\infty}\fint_{R\Box}f_\rho(\tau_xa)\,dx = \expec{f_{\rho}}=\expec{\fint_{\rho\Box}f}=\expec{\fint_{\Box} f},
  \end{equation}
  for all $a\in\Omega'$ with $\prob(\Omega')=1$. From now on let $a\in\Omega'$ be fixed. For all $y\in\rho\Box$ we have
  \begin{equation*}
    \fint_{R\Box}f(a,x)\,dx\leq(\frac{(R+\rho)}{R})^d\fint_{(R+\rho)\Box}f(a,x+y)\,dx,
  \end{equation*}
  and thus applying $\fint_{\rho\Box}\cdot\,dy$ yields
  \begin{equation*}
    \fint_{R\Box}f(a,x),dx\leq(\frac{(R+\rho)}{R})^d\fint_{(R+\rho)\Box}\fint_{\rho\Box}f(a,x+y)\,dy\,dx.
  \end{equation*}
  By stationarity of $f$, we find that
  \begin{equation*}
    \fint_{\rho\Box}f(a,x+y)\,dy=    \fint_{x+\rho\Box}f(a,y)\,dy=\fint_{\rho\Box}f(\tau_xa,y)\,dy=f_\rho(\tau_x a),
  \end{equation*}
  and thus
  \begin{equation*}
    \fint_{R\Box}f(a,x),dx\leq(\frac{(R+\rho)}{R})^d\fint_{(R+\rho)\Box}f_\rho(\tau_x a)\,dx.
  \end{equation*}
  By a similar argument, we obtain
  \begin{equation*}
    \fint_{R\Box}f(a,x),dx\geq(\frac{(R-\rho)}{R})^d\fint_{(R-\rho)\Box}f_\rho(\tau_x a)\,dx.
  \end{equation*}
  Thanks to \eqref{L:ergodic:p1} the right-hand sides of the previous two equations converge to $\expec{\fint_{\Box}f}$, and thus we conclude that
  \begin{equation*}
    \lim\limits_{R\to\infty}\fint_{R\Box}f(a,x)\,dx = \expec{\fint_{\Box}f}.
  \end{equation*}
  \medskip

  \step 2 Proof of the second statement.

  Set $\mathcal Q':=\{\,Q=q+\ell\Box\,:\,Q\subset\Box,q\in\mathbb Q^d,\,\ell\in\mathbb Q_{>0}\,\}$, where $\mathbb Q$ denotes the field of rational numbers. By part (a) for any $Q\in\mathcal Q'$, say $Q=q+\ell\Box$, we have
  \begin{equation*}
    \fint_{\Box}f(a,Rx)\mathbf{1}_Q(x)\,dx=\frac{|Q|}{|\Box|}\fint_{q+\ell\Box}f(a,Rx)=\frac{|Q|}{|\Box|}\fint_{R\ell\Box}f(\tau_qa,x)\to\expec{\fint_{\Box}f}\fint_{\Box}\mathbf{1}_Q,
  \end{equation*}
  where $\mathbf{1}_Q$ denotes the indicator function of the set $Q$. The above convergence holds for all $a\in \Omega_{Q}$ with $\prob(\Omega_{Q})=1$. Since $\mathcal Q'$ is countable, we can find a set $\Omega'$ with $\prob(\Omega')=1$, such that the above convergence is valid for all $a\in\Omega'$, $Q\in\mathcal Q'$, and such that additional we have
  \begin{equation}\label{L:ergodic.p3}
    \fint_{R\Box}|f(a,x)|^2\to \expec{\fint_\Box|f|^2}<\infty.
  \end{equation}
  From now on let $a\in\Omega'$. We conclude by a density argument. By linearity, for any 
  \begin{equation*}
    \eta\in D:=\operatorname{span}\{\mathbf{1}_Q\,:\,Q\in\mathcal Q'\},
  \end{equation*}
  we get
  \begin{equation}\label{L:ergodic.p2}
    \fint_{\Box}f(a,Rx)\eta(x)\,dx\to\expec{\fint_{\Box}f}\fint_{\Box}\eta.
  \end{equation}
  Since $D\subset L^2(\Box)$ is dense, for any $\eta\in L^2(\Box)$ and $\delta>0$ we can find $\eta'\in D$ s.t. $\|\eta-\eta'\|_{L^2(\Box)}\leq\delta$, and thus
  \begin{equation}\label{L:ergodic.p4}
    \fint_{\Box}|f(a,Rx)(\eta(x)-\eta'(x)|\leq \left(\fint_{\Box}|f(a,Rx)|^2\right)^{\frac{1}{2}}\|\eta-\eta'\|_{L^2(\Box)}\leq \delta\left(\fint_{R\Box}|f(a,x)|^2\right)^\frac12.
  \end{equation}
  By the triangle inequality we have
  \begin{eqnarray*}
    &&\Big|\fint_{\Box}f(a,Rx)\eta(x)\,dx-\expec{\fint_{\Box}f}\fint_{\Box}\eta\Big|\\
    &&\leq\, \fint_{\Box}|f(a,Rx)(\eta(x)-\eta'(x))|\,dx+\Big|\fint_{\Box}f(a,Rx)\eta'(x)\,dx-\expec{\fint_{\Box}f}\fint_{\Box}\eta'|\\
    &&\,\,\,\,+\Big|\expec{\fint_{\Box}f}|\fint_{\Box}\eta-\fint_{\Box}\eta'|.
  \end{eqnarray*}
  In view of \eqref{L:ergodic.p3},\eqref{L:ergodic.p2} and \eqref{L:ergodic.p4}, we deduce that
  \begin{equation*}
    \limsup\limits_{R\to\infty}\Big|\fint_{\Box}f(a,Rx)\eta(x)\,dx-\expec{\fint_{\Box}f}\fint_{\Box}\eta\Big|\leq 2\delta\expec{\|f\|_{L^2(\Box)}^2}^\frac12.
  \end{equation*}
  Since $\delta>0$ is arbitrary, the claim follows.  
\end{proof}


\section{Two-scale expansion and homogenization error}
In this section we extend Lemma~\ref{Int:L2} to the multidimensional, stochastic case. Next to the corrector $\phi$ we require an additional flux corrector $\sigma$. It is a classical object in periodic homogenization, e.g. see \cite{Zhikov-book}. In the stochastic case it has been recently introduced in \cite{GNO4est}.

\begin{proposition}[extended corrector]\label{P:corr2}
  Suppose Assumption (S) is satisfied. For $i=1,\ldots,d$ there exists a unique triple $(\phi_i,\sigma_i,q_i)$ such that
  \begin{enumerate}[(a)]
  \item $\phi_i$ is a scalar field, $\sigma_i=\sigma_{ijk}$ is a matrix field, and $\phi_i,\sigma_{ijk}\in\mathcal H$, see Lemma~\eqref{L:spaceH0}.
  \item $\prob$-a.s. we have
    \begin{eqnarray*}
      -\nabla\cdot a(\nabla\phi_i+e_i)&=&0\\
      q_i&:=&a(\nabla\phi_i+e_i)-a_{\hom}e_i\\
      -\triangle \sigma_{ijk}&=&\partial_j q_{ik}-\partial_k q_{ij}
    \end{eqnarray*}
    in $\mathcal D'(\R^d)$.
  \item $\sigma_{i}$ is skew symmetric, and
    \begin{equation*}
      -\nabla\cdot\sigma_i=q_i\qquad\text{in }\mathcal D'(\R^d)
    \end{equation*}
    with the convention that $(\nabla\cdot \sigma_i)_j=\sum_{k=1}^d\partial_k\sigma_{ijk}$.
  \end{enumerate}
\end{proposition}
\begin{proof}[Sketch of the argument]
  The existence of $\phi_i$ and its properties are already explained in Proposition~\ref{P:corr-stoch}. The existence argument for $\sigma_i$ is similar and omitted here. We only sketch the argument for property (c). To that end apply $\triangle$  to $-\nabla\cdot\sigma_i$; then (in a distributional sense):
  \begin{eqnarray*}
    -\triangle(\nabla\cdot\sigma_i)_j&=&-\partial_k\triangle\sigma_{ijk}=\partial_k(\partial_j q_{ik}-\partial_k q_{ij})=\partial_j(\nabla\cdot q_i)-\triangle q_{ij}\\
    &=&-\triangle q_{ij}.
  \end{eqnarray*}
  Hence $\zeta:=(\nabla\cdot\sigma_i)_j+q_{ij}$ is harmonic in $\R^d$, and satisfies $\limsup_{R\to\infty}\fint_{B_R}|\zeta|^2<\infty$. By a variant of Liouville's theorem we conclude that $\zeta$ is equal to a constant. Since $\expec{q_{ij}}=\expec{(\nabla\cdot\sigma_i)_j}=0$, we conclude that $\zeta=0$ and thus $-\nabla\cdot\sigma_i=q_i$.
\end{proof}

\begin{remark}
  In dimension $d=1$, we simply have $\sigma=0$.
\end{remark}

\begin{theorem}\label{Sto:T3}
  Suppose $\mathcal P$ satisfies Assumption (S). Let $\alpha>0$ and $f\in L^{2}(\R^{d})$. For $a\in\Omega$ let $u_{\e}(a,\cdot),u_{0}\in H^1(\R^d)$ denote the unique weak solutions to
  \begin{align*}
    \alpha u_{\e}-\nabla\cdot a(\tfrac{x}{\e})\nabla u_{\e}(a,\cdot)&=f,\\
    \alpha u_{0}-\nabla\cdot a_{\hom}\nabla u_{0}&=f.
  \end{align*}
  Let $(\phi,\sigma)=(\phi_1,\ldots,\phi_d,\sigma_1,\ldots,\sigma_d)$ denote the extended corrector of Proposition~\ref{P:corr2}, and consider the two-scale expansion
  \begin{equation*}
    Z_{\e}(a,\cdot)=u_{\e}(a,\cdot)-\left(u_{0}+\e\sum_{i=1}^{d}\phi_i(a,\tfrac{\cdot}{\e})\partial_{i}u_{0}\right).
  \end{equation*}
  Then   $\prob$-a.s. we have
  \begin{align*}
    &\alpha\int|Z_{\e}|^{2}+\lambda\int|\nabla Z_{\e}|^{2}\\
 &\leq C(d,\lambda)\e^{2}\Big(\alpha \int|\phi\left(\tfrac{\cdot}{\e}\right)|^{2}|\nabla u_{0}|^{2}+\int\left(|\sigma\left(\tfrac{\cdot}{\e}\right)|^{2}+|a\left(\tfrac{\cdot}{\e}\right)|^{2}|\phi\left(\tfrac{\cdot}{\e}\right)|^{2}\right)|\nabla\nabla u_{0}|^{2}\Big).
  \end{align*}
\end{theorem}
\begin{proof}
  By a density argument, we may assume that $u_{0}$ is smooth.

\step{1} Shorthand: $a_\e:=a(\tfrac{\cdot}{\e})$, $q_{i,\e}(x)=q_i\left(\tfrac{x}{\e}\right)$, $q_{\e}(x)=\left(q_{1,\e}(x),\dots,q_{d,\e}(x)\right)$, and note that
\begin{equation*}
  q_\e=a_\e(I+\nabla\phi(\tfrac{\cdot}{\e})+I)-a_{\hom}.
\end{equation*}
We compute (using Einstein's summation convention),
\begin{eqnarray*}
  \nabla Z_{\e}&=&\nabla u_{\e}-\left(\nabla u_{0}+\nabla\phi_i\left(\tfrac{\cdot}{\e}\right)\partial_{i}u_{0}\right)-\e\phi_i\left(\tfrac{\cdot}{\e}\right)\nabla\partial_{i}u_{0}\\
&=&\nabla u_\e-(I+\nabla\phi(\tfrac{\cdot}{\e}))\nabla u_{0}-\e\phi_i\left(\tfrac{\cdot}{\e}\right)\nabla\partial_{i}u_{0}.
\end{eqnarray*}
Then (with $a_\e:=a(\tfrac{\cdot}{\e})$)
\begin{align*}
  a_{\e}\nabla Z_{\e}=&a_{\e}\nabla u_{\e}-a_{\hom}\nabla u_0-q_{\e}\nabla u_{0}-\e\phi_i\left(\tfrac{\cdot}{\e}\right)a_{\e}\nabla\partial_{i}u_{0}.
\end{align*}
Thus,
\begin{align}\label{Sto:Eq4}
\begin{split}
  \alpha\int Z_{\e}^{2}+\int \nabla Z_{\e}\cdot a_{\e}\nabla Z_{\e}=&\underbrace{\alpha\int Z_{\e}(u_{\e}-u_{0})+\int \nabla Z_{\e}\cdot a_\e\nabla u_{\e}-\nabla Z_{\e}\cdot a_{\hom}\nabla u_{0}}_{=0}\\
  &-\alpha\e\int\phi_i\left(\tfrac{\cdot}{\e}\right)\partial_{i}u_{0}Z_{\e}\\
&-\int q_{\e}\nabla u_{0}\cdot\nabla Z_{\e}\\
&-\e\int\phi_i\left(\tfrac{\cdot}{\e}\right)a\left(\tfrac{\cdot}{\e}\right)\nabla\partial_{i}u_{0}\cdot\nabla Z_{\e},\\
\end{split}
\end{align}
so we have three error terms.
\step{2} We now discuss the second error term of \eqref{Sto:Eq4}. Claim:
\begin{equation*}
  \forall v\in H^{1}(\R^d):\quad\int q_{\e}\nabla u_{0}\cdot\nabla v=\e\int\nabla\partial_{i}u_{0}\cdot\sigma_i\nabla v.
\end{equation*}
Indeed, in the sense of distribution we have
\begin{equation*}
  -\nabla\cdot(q_{i,\e}\partial_{i}u_{0})=-\underbrace{(\nabla\cdot q_{i,\e})}_{=0}\partial_{i}u_{0}-q_{i,\e}\cdot\nabla(\partial_{i}u_{0}).
\end{equation*}
Thanks to Proposition~\ref{P:corr2} (c) we have with $\sigma_{i,\e}=\sigma_i\left(\tfrac{\cdot}{\e}\right)$,
\begin{equation*}
  q_{i,\e}=q_i(\tfrac{\cdot}{\e})=-(\nabla \cdot\sigma_i)\left(\tfrac{\cdot}{\e}\right)=-\e \nabla \cdot\sigma_{i,\e}.
\end{equation*}
Therefore,
\begin{align*}
  -q_{i,\e}\cdot\nabla\partial_{i}u_{0}&=\e(\nabla\cdot\sigma_{i,\e})\cdot\nabla\partial_{i}u_{0}\\
  &=\e\partial_{k}\sigma_{ijk,\e}\partial_{j}\partial_iu_{0}\\
  &=\e\partial_{k}\left(\sigma_{ijk,\e}\partial_{j}\partial_{i}u_{0}\right)-\underbrace{\e\sigma_{ijk,\e}\partial_{jk}^2\partial_{i}u_{0}}_{=0},
\end{align*}
where we have used that $\sigma_{ijk,\e}$ is skew-symmetric and $\partial_{jk}^2\partial_i u_0$ is symmetric (w.r.t. $jk$). Overall we have (in a distributional sense)
\begin{align*}
  -\nabla\cdot(q_{i,\e}\partial_i u_0)&=\e\partial_{k}\left(\sigma_{ijk,\e}\partial_{j}\partial_{i}u_{0}\right)\\
                                      &=-\e\partial_{k}\left(\sigma_{ikj,\e}\partial_{j}\partial_{i}u_{0}\right)\\
                                      &=-\e\nabla\cdot(\sigma_{i,\e}\nabla\partial_{i}u_{0}).
\end{align*}
Now the claim follows by testing with $v$. 
\step{3} Conclusion.

For the first error term in \eqref{Sto:Eq4} we have,
\begin{align*}
 \alpha\e\int\phi_i\left(\tfrac{\cdot}{\e}\right)\partial_{i}u_{0}Z_{\e}&\leq 2\alpha\e^{2}\int\left|\phi\left(\tfrac{\cdot}{\e}\right)\right|^{2}|\nabla u_{0}|^{2}+\underbrace{\tfrac{\alpha}{2}\int Z_{\e}^{2}}_{\text{to be absorbed}}.
\end{align*}
For the second error term in \eqref{Sto:Eq4}, by Step~2 we have
\begin{equation*}
  -\int q_{\e}\nabla u_{0}\cdot\nabla Z_{\e}\leq \e\int|\sigma(\tfrac{\cdot}{\e})||\nabla^2u_0||\nabla Z_\e|\leq \frac{3}{\lambda}\e^2\int|\sigma(\frac{\cdot}{\e})|^2|\nabla^2u_0|^2+\frac{\lambda}{3}\int|\nabla Z_\e|^2.
\end{equation*}
The third error term in \eqref{Sto:Eq4} is estimate as
\begin{align*}
-\e\int\phi_i\left(\tfrac{\cdot}{\e}\right)a\left(\tfrac{\cdot}{\e}\right)\nabla\partial_{i}u_{0}\cdot\nabla Z_{\e}&\leq\int\sqrt{\tfrac{3}{\lambda}}\e\left|\phi\left(\tfrac{\cdot}{\e}\right)\right|\left|a\left(\tfrac{\cdot}{\e}\right)\right||\nabla\nabla u_{0}|\sqrt{\tfrac{\lambda}{3}}|\nabla Z_{\e}|\\
&\leq\tfrac{3}{\lambda}\e^{2}\int\left|\phi\left(\tfrac{\cdot}{\e}\right)\right|^{2}\left|a\left(\tfrac{\cdot}{\e}\right)\right|^{2}|\nabla\nabla u_{0}|^{2}+\tfrac{\lambda}{3}\int|\nabla Z_{\e}|^{2}.
\end{align*}
All together we get
\begin{align*}
 \tfrac{\alpha}{2}\int Z_{\e}^{2}+\tfrac{\lambda}{3}\int|\nabla Z_{\e}|^{2}\leq&\e^{2}\left(2\alpha\int\left|\phi\left(\tfrac{\cdot}{\e}\right)\right|^{2}|\nabla\nabla u_{0}|^{2}+\right.\\
&\left.\tfrac{6}{\lambda}\int\left(\left|\phi\left(\tfrac{\cdot}{\e}\right)\right|^{2}\left|a\left(\tfrac{\cdot}{\e}\right)\right|^{2}+\left|\sigma\left(\tfrac{\cdot}{\e}\right)\right|^{2}\right)|\nabla\nabla u_{0}|\right).
\end{align*}
\end{proof}

The estimate of Theorem~\ref{Sto:T3} reveals that the rate of convergence in the homogenization limit is encoded in the behavior of the correctors $(\phi,\sigma)$. In the periodic, scalar case it is relatively easy to conclude that the two-scale expansion satisfies the estimate $\|Z_\e\|_{H^1(\R^d)}\leq C(d,\lambda)\e \|f\|_{L^2(\R^d)}$ (which is optimal in terms of scaling in $\e$). The reason is that in the periodic case, say when $\prob$ concentrates on a $1$-periodic coefficient field (and its translations),  we have by Poincar\'e's inequality on $\Box$, the estimate
\begin{equation*}
\forall R\geq 1\,:\,  \fint_{R\Box}|(\phi,\sigma)|^2\leq C(d)\fint_{R\Box}|\nabla(\phi,\sigma)|^2\leq C(d,\lambda)\qquad(\prob\text{-a.s.}),
\end{equation*}
which combined with a Moser iteration yields
\begin{equation*}
  \|(\phi,\sigma)\|_{L^\infty(\R^d)}\leq C(d,\lambda).
\end{equation*}
Thus Theorem~\ref{Sto:T3} turns into the estimate
\begin{equation*}
  \alpha\int|Z_\e|^2+\lambda\int|\nabla Z_\e|^2\leq C(d,\lambda)\e^2\int f^2,
\end{equation*}
In the stochastic case we do not expect such a behavior. The sublinearity of $(\phi,\sigma)$ only yields (for $u_0$ sufficiently regular),
\begin{equation*}
  \e^{2}\int_{\R^{d}}\left(|\sigma\left(\tfrac{\cdot}{\e}\right)|^{2}+|a\left(\tfrac{\cdot}{\e}\right)|^{2}|\phi\left(\tfrac{\cdot}{\e}\right)|^{2}\right)|\nabla\nabla u_{0}|^{2}\Big)\to 0\qquad\prob\text{-a.s.},
\end{equation*}
and we do not expect a rate for the convergence in the general ergodic case. It turns out that we need to strengthen and quantify the assumption of ergodicity in order to see a rate in the convergence above. In fact, there is a subtle interplay between the space dimension $d$, the strength of the mixing condition, and the rate of convergence. In particular, in $d=2$, even under the strongest quantitative ergodicity assumptions, e.g. for coefficients with finite range of dependence, the rate is $\e\sqrt{\log\frac{1}{\e}}$ (and thus worse than in the periodic case). This has first been observed in \cite{GNO2} where a discrete situation is studied. We refer to the recent work \cite{GNO4est}, where the interplay of correlation and the decay in the two-scale expansion is discussed in full generality.


\section{Quantitative stochastic homogenization}

The goal in this section is to establish \textit{moment bounds} for the correctors $(\phi,\sigma)$, i.e.\ estimates on $\expec{\fint_{\Box+x}|\phi|^2+|\sigma|^2}$ that capture the optimal growth in $x\in\R^d$. The sublinearity of $(\phi,\sigma)$ yields only the behavior $\frac{1}{|x|^2}\expec{\fint_{\Box+x}|\phi|^2+|\sigma|^2}\to 0$ for $|x|\to\infty$, but not a quantitative growth rate. This is in contrast to the periodic case, where $\fint_{\Box+x}|\phi|^2+|\sigma|^2$ is bounded uniformly in $x\in\R^d$ -- a consequence of Poincar\'e's inequality on the unit cell of periodicity.
It turns out that in order to obtain a quantitative growth rate, we need to strengthen and quantify the assumption of ergodicity. In particular, we shall see that for $d\geq3$ we obtain an estimate that is uniform $|x|$ and for $d=2$ a logarithmic growth rate, provided $\prob$ satisfies a strong \textit{quantitative form} of ergodicity. Combined with the two-scale expansion Theorem~\ref{Sto:T3} such moment bounds yield error estimates for the homogenization error. Moreover, moment bounds on the corrector are at the basis to prove various quantitative results in stochastic homogenization, e.g.\ estimates on the approximation error of $a_{\hom}$ by representative volume elements of finite size, e.g.\ see \cite{Gloria-Otto-09, Gloria-Otto-09b, GNO1, GNO2, GNO4est, ArmInvent}.
\medskip

In the following we work in a discrete framework, i.e.\ $\R^d$ is replaced by $\Z^d$ and the elliptic operator $-\nabla\cdot (a\nabla)$ is replaced by an elliptic finite difference operator, $\nabla^*(a\nabla)$. We do this for several reasons:
\begin{itemize}
\item it is easy to define model problems of random coefficients satisfying a quantitative ergodicity assumption (e.g.\ i.i.d. coefficients),
\item some technicalities disappear: e.g.\ questions of regularity on small scales,
\item on the other hand: main difficulties are already present in full strength in the discrete case,
\item main concepts and results naturally extend to the continuum case,
\item the discrete framework is a natural setup in statistical mechanics and probability theory (e.g.\ random conductance models, see \cite{Biskup,Kumagai} for recent reviews).
\end{itemize}

\subsection{The discrete framework  and the discrete corrector}
We consider functions defined on the lattice $\Z^d$ and set for $1\leq p<\infty$,
\begin{equation*}
  \ell^p:=\{\,f:\Z^d\to\R\,:\,\|f\|_{\ell^p}:=\Big(\sum_{x\in\Z^d}|f(x)|^p\Big)^\frac1p<\infty\,\},
\end{equation*}
and
\begin{equation*}
  \ell^\infty:=\{\,f:\Z^d\to\R\,:\,\|f\|_{\ell^\infty}:=\sup_{x\in\Z^d}|f(x)|<\infty\,\}.
\end{equation*}

\paragraph{Discrete calculus.} Given a scalar field $f:\Z^d\to\R$, and a vector field $F=(F_1,\ldots,F_d):\Z^d\to\R^d$, we set 
\begin{gather*}
  \nabla_i f(x):=f(x+e_i)-f(x),\quad \nabla^*_if(x):=f(x-e_i)-f(x),\\
  \nabla f=(\nabla_1f,\ldots,\nabla_df),\qquad   \nabla^*F=\sum_{i=1}^d\nabla^*_iF_i\;.
\end{gather*}
It is easy to check that for $f\in \ell^p$ and $F_i\in \ell^q$ (with $p,q$ dual exponents) the integration by \textit{parts formula}
\begin{equation*}
  \sum_{x\in\Z^d}f(x)\nabla^*F(x)=  \sum_{x\in\Z^d}\nabla f(x)\cdot F(x),
\end{equation*}
holds.
Thus $\nabla^*$ is the adjoint of $\nabla$, and the discrete analogue to $-\nabla\cdot$.

\paragraph{Discrete elliptic operator and Green's function.} 
Recall that $\lambda\in(0,1)$ (the ellipticity ratio) is fixed. Define
\begin{eqnarray*}
  \Omega_0&:=&\Big\{\,\aa\in\R^{d\times d}\,:\,\aa=\operatorname{diag}(a_1,\ldots,a_d)\text{ with }a_i\in(\lambda,1)\,\Big\}\subset \R^{d\times d},\\
  \Omega&:=&\{\,\aa:\Z^d\to\Omega_0\,\}=\Omega^{\Z^d}.
\end{eqnarray*}
Then for any $a\in\Omega$ (and any $1\leq p\leq \infty$), $\nabla^*(a\nabla):\ell^p\to\ell^p$ is a bounded linear operator which is uniformly elliptic and satisfies a maximum principle. We denote the Green's function associated with $\nabla^*(a\nabla)$ by $G(a;x,y)$, i.e.\ $G(a;\cdot,y):\Z^d\to\R$ is the unique sublinear solution (resp. bounded solution if $d>2$) to
\begin{equation*}
  \nabla^*(a\nabla G(a;\cdot,y))=\delta(\cdot-y)\qquad\text{in }\Z^d,
\end{equation*}
where $\delta:\Z^d\to\{0,1\}$ denotes the Dirac function centered at $0$.

\paragraph{Random coefficients.}
Let  $\prob$ denote a probability measure on $(\Omega, \otimes_{\Z^d}\mathcal B(\R^{d\times d}))$.
We introduce the shift-operator
\begin{equation}
  \tau:\Z^d\times\Omega\to\Omega,\quad \tau_z\aa:=\aa(\cdot+z)
\end{equation}
and always assume \textit{stationarity} of $\prob$, i.e.\ for any $z\in\Z^d$ the mapping
\begin{equation}\label{A:stat:discrete}\tag{D1}
  \forall z\in\Z^d\,:\qquad \tau_z:\Omega\to\Omega\text{ preserves the measure }\prob.
\end{equation}
We say $\prob$ is \textit{ergodic}, if 
\begin{equation}\label{A:ergo:discrete}\tag{D2}
  A\subset\Omega\text{ is shift invariant}\qquad\Rightarrow\qquad \prob(A)\in\{0,1\}.
\end{equation}
Birkhoff's ergodic theorem then implies:
\begin{equation*}
  \lim\limits_{R\to\infty}R^{-d}\sum_{x\in R\Box\cap\Z)^d}f(\tau_x\aa)=\expec{f}
\end{equation*}
for a.e.\ $\aa$ and all $f\in L^1(\Omega)$. 

\paragraph{Stationary random fields and the ``horizontal'' differential calculus.} We say a function $u:\Omega\times\Z^d\to\R$ is a random field, if $u(\cdot,x)$ is measurable for all $x\in\Z^d$. We say that $u$ is a stationary random field, if
\begin{equation*}
  u(a,x)=u(\tau_xa,0)\qquad\text{for all }x\in\Z^d\text{ and }\prob\text{-a.e.\ }a\in\Omega.
\end{equation*}
For a stationary random field $u$ the value of $\expec{u(x)}$ is independent of $x\in\Z^d$, and thus we simply write $\expec{u}$. We consider the space
\begin{equation*}
  \mathcal S:=\Big\{\,u:\Omega\times\Z^d\to\R\,:\,u\text{ is stationary and }\expec{|u|^2}<\infty\,\Big\},
\end{equation*}
which with the inner product $(u,v)_{\mathcal S}:=\expec{uv}$ is a Hilbert space. For a random variable $u$ we set $\overline u(a,x):=u(\tau_xa)$. We call $\overline u$ the stationary extension of $u$, and note that the map
\begin{equation*}
  \overline{(\cdot)}:\,L^2(\Omega)\to \mathcal S,\qquad u\mapsto \overline u
\end{equation*}
is a linear isometric isomorphism (thanks to the stationarity of $\prob$). Note that  for any $u\in\mathcal S$, we have
\begin{equation*}
  \nabla_i u(a,x)=u(a,x+e_i)-u(a,x)=u(\tau_{e_i}a,x)-u(a,x).
\end{equation*}
Motivated by this, we define for a random variable $u:\Omega\to\R$ and a random vector $F:\Omega\to\R$ the ``horizontal'' derivatives
\begin{equation}\label{def:D}
  \begin{aligned}
    D_i u(a):=&\,u(\tau_{e_i}a)-u(a),\quad D^*_iu(a):=u(\tau_{-e_i}a)-u(a),\\
    D u=&\,(D_1f,\ldots,D_df),\qquad   D^*F=\sum_{i=1}^dD^*_iF_i,
  \end{aligned}
\end{equation}
and note that we have
\begin{equation*}
  \nabla\overline u(a,x)=\overline{(Du)}(a,x),\qquad \nabla^*\overline F(a,x)=\overline{(D^*F)}(a,x).
\end{equation*}
Moreover, for a random variable $u\in L^p(\Omega)$ and a random vector $F\in L^p(\Omega,\R^d)$ the integration by parts formula
\begin{equation*}
  \expec{uD^*F}=\expec{Du\cdot F}
\end{equation*}
holds, as a simple consequence of the stationarity of $\prob$.

\paragraph{Homogenization result  in the discrete case.}
As in the continuum case, homogenization in the random, discrete case relies on the notion of correctors. The correctors belong to the space
\begin{eqnarray*}
  \mathcal H:=\Big\{ u:\Omega\times\Z^d\to\R\,&:&\,u(\cdot,x)\text{ is measurable for all }x\in\Z^d,\\
  &&\nabla u\text{ is stationary, }\expec{|\nabla u|^2}<\infty,\text{ and }\expec{\nabla u}=0\,\Big\},
\end{eqnarray*}
which equipped with
\begin{equation*}
  (u,v)_{\mathcal H_0}:=\expec{\nabla u\cdot\nabla v}
\end{equation*}
is a Hilbert space. (Note that since $\nabla u$ and $\nabla v$ are stationary, the value of $\expec{\nabla u(x)\cdot\nabla v(x)}$ does not depend on $x\in\Z^d$, and thus we simply write $\expec{\nabla u\cdot\nabla v}$.
The following result is the discrete analogue to Proposition~\ref{P:corr-stoch}:
\begin{proposition}\label{P:corr:discrete}
  Assume (D1) and (D2). For $i=1,\ldots,d$ there exist unique random fields $\phi_i$, $q_i$ and $\sigma_i=\sigma_{ijk}$ such that
  \begin{enumerate}[(a)]
  \item $\phi_i$ is a random scalar field, $\sigma_i=\sigma_{ijk}$ is a random matrix field, and $\phi_i,\sigma_{ijk}\in\mathcal H$.
  \item $\prob$-a.s. we have
    \begin{eqnarray}\label{random-corr:dis1}
      \nabla^*(a(\nabla\phi_i+e_i))&=&0\qquad\text{in }\Z^d,\\
      \label{random-corr:dis2}
      q_i&=&a(\nabla\phi_i+e_i)-a_{\hom}e_i\qquad\text{in }\Z^d,\\
      \label{random-corr:dis3}
      \nabla^*\nabla\sigma_{ijk}&=&\nabla_kq_{ij}-\nabla_jq_{ik}\qquad\text{in }\Z^d,
    \end{eqnarray}
    where $a_{\hom}e_i:=\expec{a(\nabla\phi_i+e_i)}$, and $q_{ij}$ denotes the $j$th component of the vector $q_i$.
  \item $\sigma_i$ is skew symmetric, and
    \begin{equation*}
      \nabla^*\sigma_i=q_i\qquad\text{in }\Z^d,
    \end{equation*}
    where $(\nabla^*\sigma_i)_j=\sum_{k=1}^d\nabla^*_k\sigma_{ijk}$.
  \end{enumerate}
\end{proposition}
Since the proof of the proposition is similar to the continuum case, we omit it here and refer to \cite{GNO4reg,BMN17,AN}.
With help of Proposition~\ref{P:corr:discrete} we obtain the following discrete (rescaled) analogue to Theorem~\ref{Sto:T3}:
\begin{theorem}[Discrete two-scale expansion]\label{Sto:T3dis}
Assume (D1) and (D2). Let $\alpha>0$ and $f\in L^{2}(\Z^d)$. For $a\in\Omega$ let $u(a,\cdot),u_{0}:\Z^d\to\R$ denote the unique square summable solutions to
  \begin{align*}
    \alpha u(a,\cdot)+\nabla^*(a\nabla u(a,\cdot))&=f\qquad\text{in }\Z^d,\\
    \alpha u_{0}+\nabla^*(a_{\hom}\nabla u_{0})&=f\qquad\text{in }\Z^d.
  \end{align*}
  Let $(\phi,\sigma)=(\phi_1,\ldots,\phi_d,\sigma_1,\ldots,\sigma_d)$ denote the extended corrector of Proposition~\ref{P:corr:discrete}, and consider the two-scale expansion
  \begin{equation*}
    Z(a,\cdot)=u(a,\cdot)-\left(u_{0}+\sum_{i=1}^{d}\phi_i(a,\cdot)\nabla_{i}u_{0}\right).
  \end{equation*}
  Then for $\prob$-a.e.\ $a\in\Omega$ we have
  \begin{align*}
    &\sum_{x\in\Z^d}\alpha |Z(a,x)|^{2}+\lambda|\nabla Z(a,x)|^{2}\\
    &\leq C(d,\lambda)\Big(\alpha \sum_{x\in\Z^d}|\phi(a,x)|^{2}|\nabla u_{0}(x)|^{2}+\sum_{x\in\Z^d}\left(|\sigma(a,x)|^{2}+|a(x)|^{2}|\phi(a,x)|^{2}\right)|\nabla\nabla u_{0}(x)|^{2}\Big),
  \end{align*}
where $(\nabla\nabla u_0)_{ij}=\nabla_i\nabla_j u_0$.
\end{theorem}
The statement should be compared with a rescaled (i.e.\ $\frac{x}{\e}\leadsto x$) version of Theorem~\ref{Sto:T3}. The proof is (up to minor modification regarding the transition to the discrete setting) similar to the continuum case. We omit it here and refer to \cite[Proof of Proposition 3]{BMN17}.
\medskip

In the rest of this section we are interested in proving bounds for the correctors $(\phi,\sigma)$. 

\paragraph{Heuristics.} To get an idea of what we can expect regarding an estimate on $\expec{|\phi(x)|^2}$, we consider the simplified equation
\begin{equation*}
  \nabla^*\nabla\phi=\nabla^*(a\xi).
\end{equation*}
Since the divergence does not see constants, we may assume without loss of generality that $\expec{a(x)}=0$.
Formally a solution can be represented with help of the Green's function $G(x):=G(\mathbf{id};x,0)$ associated with $\nabla^*\nabla$:
\begin{equation*}
  \phi(x)=\sum_{y\in\Z^d}G(x-y)\nabla^*(a(y)\xi)=\sum_{y\in\Z^d}\nabla G(x-y)\cdot(\aa(y)\xi).
\end{equation*}
Thus
\begin{eqnarray*}
  \expec{\phi(x)^2}&=&\sum_{y}\sum_{y'}\nabla_i G(x-y)\nabla_j G(x-y')\expec{(\aa(y)\xi)_i(\aa(y')\xi)_j}\\
  &=&\sum_{y}\sum_{y'}\nabla_i G(x-y)\nabla_j G(x-y')\expec{(\aa(0)\xi)_i(\aa(y'-y)\xi)_j}\\
  &\stackrel{z=y'-y}{=}&\sum_{y}\sum_{z}\nabla_i G(x-y)\nabla_j G(x-y-z)\expec{(\aa(0)\xi)_i(\aa(z)\xi)_j}\\
  &\stackrel{y\leadsto x-y}{=}&\sum_{y}\sum_{z}\nabla_i G(y)\nabla_j G(y-z)\expec{(\aa(0)\xi)_i(\aa(z)\xi)_j}.
\end{eqnarray*}
Specify to $\xi=e_\alpha$, by diagonality have $(a(0)\xi)_i=\delta_{\alpha i}\aa_{\alpha}$. Since $\expec{\aa}=0$, we arrive at
\begin{eqnarray*}
  \expec{\phi(x)^2}&=&\sum_{y}\sum_{z}\nabla_\alpha G(y)\nabla_\alpha G(y-z)C(z)\\
  &\leq&\sum_{y}\sum_{z}(|y|+1)^{1-d}(|y-z|+1)^{1-d}|C(z)|.
\end{eqnarray*}
where $C(z):=\operatorname{COV}(\aa_{\alpha}(0),\aa_{\alpha}(z))$. Note that the behavior $|C(z)|\to0$ for $z\to\infty$ encodes a decay of correlations. Let us impose the strongest possible assumption, namely independence, i.e.\ $C(z)\sim \delta(z)$. We get
\begin{equation*}
  \sum_{y}\sum_{z}(|y|+1)^{1-d}(|y-z|+1)^{1-d}|C(z)|=  \sum_{y}(|y|+1)^{2(1-d)},
\end{equation*}
and see that the right-hand side is finite if and only if $d\geq 3$.

This suggests:
\begin{itemize}
\item We can only expect moment bounds on $\phi$ (uniformly in $x$) for $d\geq 3$.
\item We need assumptions on the decay of correlations of the random coefficients ($\leadsto$ quantification of ergodicity)
\item Regularity theory for elliptic equations is required, e.g.\ estimates on the gradient of the Green's function.
\end{itemize}

\subsection{Quantification of ergodicity via Spectral Gap}
In this section we discuss how ergodicity can be quantified by means of a spectral gap estimate. The presentation closely follows \cite{GNO1preprint}, which is an extended preprint to \cite{GNO1}.
\begin{definition}[vertical derivative and Spectral Gap (SG)]\label{def:SG}
  \begin{itemize}
  \item For $f\in L^1(\Omega)$ and $x\in\Z^d$ we define the vertical derivative as
    \begin{equation*}
      \partial_xf:=f-\expec{f\,\big|\,\mathcal F_x},
    \end{equation*}
    where $\expec{\cdot|\mathcal F_x(a)}$ denotes the conditional expectation where we condition on the $\sigma$-algebra $\mathcal F_x:=\sigma(\pi_z\,:\,z\neq x)$, $\pi_z\aa:=\aa(z)$. \\
  \item We say $\prob$ satisfies (SG) with constant $\rho>0$, if for any $f\in L^2(\Omega)$ we have
    \begin{equation*}
      \expec{(f-\expec{f})^2}\leq\frac{1}{\rho}\sum_{x\in\Z^d}\expec{(\partial_x f)^2}.
    \end{equation*}
  \end{itemize}
\end{definition}
We might interpret the vertical derivative as follows: $\expec{\cdot\big|\mathcal F_x}$ denotes the conditional expectation, where we condition on the event that we know the value of $\aa(z)$ for all sites $z\neq x$; thus, $\partial_xf$ ``measures'' how sensitive $f(\aa)$ reacts to changes of the value of $\aa$ at $x$. The estimate (SG) is also called ``Efron-Stein inequality'' and is an example of a concentration inequality. We refer to \cite{Ledoux} for a review on concentration inequalities. Note that we can bound $|\partial_xf(\aa)|$ from above by appealing to the classical partial derivative:
\begin{equation*}
  |\partial_x f(\aa)|\leq \sup\{f(\aa)-f(\tilde\aa)\,:\,\tilde\aa\in\Omega\text{ with }\aa=\tilde\aa\text{ on }\Z^d\setminus\{x\}\}\leq \int_{\lambda}^1|\frac{\partial f(\aa)}{\partial {\aa(x)}}|\,dx.
\end{equation*}
Let us anticipate that below in Section~\ref{S:sublin} we replace the vertical derivative $\partial_x f$ by a Lipschitz derivative, which is stronger than the vertical derivative and thus yields a weaker condition. Concentration inequalities such as (SG) yield a natural way to quantify ergodicity for random coefficients that rely on a product structure:
\begin{lemma}\label{L:iid}
  Suppose that $\prob$ is \textit{independent and identically distributed}, i.e.\
  \begin{equation*}
    \mathbb P=\otimes_{x\in\Z^d}\mathbb P_0(dx)\qquad\text{for some }\mathbb P_0\text{ probability measure on }\Omega_0.
  \end{equation*}
   Then $\expec{\cdot}$ satisfies (SG) with constant $\rho=1$.
\end{lemma}

\begin{proof}[Proof of Lemma~\ref{L:iid}]
 The argument is standard. We follow \cite{GNO1preprint} and start with preparatory remarks.
  \begin{itemize}
  \item Let $x_1,x_2,x_3,\ldots$ denote an enumeration of $\Z^d$,
  \item Since $\mathbb P$ is a product measure, we have
    \begin{eqnarray*}
      \expec{\zeta|\mathcal F_{x_n}}=\int_{\Omega_0}\zeta\,\mathbb P_0(dx_n).
    \end{eqnarray*}
  \item  We introduce the shorthand
    \begin{eqnarray*}
      \expec{\zeta}_n&:=&\int_{(\Omega_0)^n}\zeta\,\prod_{i=1}^n\mathbb P_0(dx_i),\\
      \zeta_n&:=&\expec{\zeta}_n,\\
      \zeta_0&:=&\zeta,
    \end{eqnarray*}
    i.e.\ $\zeta_n$ does not depend on the values of $a(x_1),\ldots,a(x_n)$.
  \end{itemize}
  Thanks to the product structure of $\prob$, we have 
  \begin{eqnarray*}
    \expec{|\partial_{x_n}\zeta|^2}&=&\expec{\expec{|\partial_{x_n}\zeta|^2}_{n-1}}=\expec{\expec{\big|\zeta-\int_{\Omega_0}\zeta\mathbb P_0(dx_n)\big|^2}_{n-1}}\\
    &\stackrel{\text{Jensen}}{\geq}&
    \expec{\big|\expec{\zeta-\int_{\Omega_0}\zeta\mathbb P_0(dx_n)}_{n-1}\big|^2}\\
    &=&
    \expec{|\zeta_{n-1}-\zeta_n\big|^2}.
  \end{eqnarray*}
  Now, the statement follows from the Martingale decomposition
  \begin{equation}
    \label{eq:27}
    \expec{(\zeta-\expec{\zeta})^2}=\sum_{n=1}^\infty\expec{\,(\zeta_{n-1}-\zeta_n)^2}.
  \end{equation}
  Here comes the argument for \eqref{eq:27}:  Since $\partial_x\expec{\zeta}=0$, it suffices to consider $\zeta\in L^2(\Omega)$ with $\expec{\zeta}=0$. By a density argument, it suffices to consider $\zeta\in L^2(\Omega)$ that only depend on a finite number of coefficients, i.e.\ $\zeta_N=\expec{\zeta}$ for some $N\in\N$.  Hence, by definition we have $\zeta_0=\zeta$ and $\zeta_N=\expec{\zeta}=0$ for $N$ large enough, and thus (by telescopic sum)
  \begin{equation}
    \label{eq:26}
    \zeta=\sum_{n=1}^N\zeta_{n-1}-\zeta_n.
  \end{equation}
  Taking the square and the expected value yields
  \begin{equation*}
    \expec{\zeta^2}=\sum_{n=1}^N\sum_{m=1}^N\expec{\,(\zeta_{n-1}-\zeta_n)(\zeta_{m-1}-\zeta_m)\,}.
  \end{equation*}
  Hence, \eqref{eq:27} follows, provided that the random variables
  $\{\,\zeta_{n-1}-\zeta_n\,\}_{n\in\N}$ are independent (i.e.\ pairwise orthogonal in
  $L^2(\Omega)$). For the argument let $m>n$. Since by
  construction $\zeta_{m-1}-\zeta_{m}$ does not depend on
  $\aa(y_1),\ldots,\aa(y_{m-1})$ we have
  \begin{equation}
    \label{eq:24}
    \zeta_{m-1}-\zeta_m=\expec{\,\zeta_{m-1}-\zeta_m\,}_{ m-1};
  \end{equation}
  and since $m-1\geq n$, we have
  \begin{equation}
    \label{eq:25}
    \expec{\zeta_{n-1}-\zeta_n}_{ m-1}=\expec{\expec{\zeta}_{n-1}}_{ m-1}-\expec{\expec{\zeta}_{n}}_{      m-1}=\expec{\zeta}_{m-1}-\expec{\zeta}_{ m-1}=0.
  \end{equation}
  Hence, using the general identity $\expec{\expec{u}_{m-1} v}=\expec{u\expec{v}_{m-1}}$, get
  \begin{eqnarray*}
    \expec{\,(\zeta_{m-1}-\zeta_m)(\zeta_{n-1}-\zeta_{n})\,}
    &\stackrel{\eqref{eq:24}}{=}&
    \expec{\,\expec{\zeta_{m-1}-\zeta_m}_{m-1}(\zeta_{n-1}-\zeta_{n})\,}\\
    &=&
    \expec{\,(\zeta_{m-1}-\zeta_m)\expec{\zeta_{n-1}-\zeta_{n}}_{        m-1}\,}\\
    &\stackrel{\eqref{eq:25}}{=}&0
  \end{eqnarray*}
  and the claim follows.
\end{proof}

We next illustrate that (SG) not only implies, but also quantifies ergodicity. For this reason let $p(t,x)$ denote the Green's function for the heat equation $\partial_t+\nabla^*\nabla$ (i.e.\ the unique function in $C([0,\infty),\ell^2(\Z^d))\cap C^1(\Omega,\ell^2(\Z^d))$ satisfying $\partial_t p+\nabla^*\nabla p=0$ on $(0,\infty)\times\Z^d$ and $p(0,x)=\delta(x)$). Note that $p(t,x)$ (which is also referred to as the heat kernel of the simple random walk on $\Z^d$) is non-negative, normalized $\sum_{x\in\Z^d}p(t,x)=1$, and in particular, it satisfies the on-diagonal heat kernel estimate
\begin{equation*}
  \sum_{x\in\Z^d}p^2(t,x)\leq C(d)(t+1)^{-\frac{d}{2}}.
\end{equation*}
With help of $p(t,x)$ we might define a semigroup $(P(t))_{t\geq 0}$ on $L^2(\Omega)$ by setting
\begin{equation*}
  P(t):L^2(\Omega)\to L^2(\Omega),\qquad P(t)\zeta:=\sum_{x\in\Z^d}p(t,x)\overline{\zeta}(a,x),
\end{equation*}
where $\overline\zeta(a,x):=\zeta(\tau_xa)$ denotes the stationary extension. Thanks to the interplay of $\overline{(\cdot)}$ and $\nabla$, stationarity of $\prob$ implies that the generator of $(P_t)_{t\geq0}$ is given by $-D^*D$, where $D$ denotes the horizontal derivative, see \eqref{def:D}. We thus may equivalently write $P_t\zeta=\exp(-tD^*D)\zeta$. Note that the exponential is unambiguously defined, since $-D^*D:L^2(\Omega)\to L^2(\Omega)$ is a bounded linear operator by the definition of $D^*D$ and the triangle inequality.
It turns out that ergodicity can be characterized with help of $P_t$.

\begin{lemma}[Characterization and quantification of ergodicity]\label{L:quanterg}
  Let $\prob$ be stationary. Consider the semigroup defined by
  \begin{equation*}
    P(t)\zeta:=\exp(-tD^*D)\zeta.
  \end{equation*}
  Then
  \begin{enumerate}[(a)]
  \item $\prob$ is ergodic, if and only if
    \begin{equation*}
      \forall\zeta\in L^2(\Omega)\,:\qquad\lim\limits_{t\to\infty}\expec{|P(t)\zeta-\expec{\zeta}|^2}=0.
    \end{equation*}
  \item If $\prob$ satisfies (SG) with constant $\rho>0$, then
    \begin{equation*}
      \expec{|P(t)\zeta-\expec{\zeta}|^2}^{\frac{1}{2}}\leq
      \frac{C(d)}{\sqrt \rho}(t+1)^{-\frac{d}{4}}\,\sum_{x\in\Z^d}\expec{|\partial_x\zeta|^2}^{\tfrac{1}{2}},
    \end{equation*}
  \end{enumerate}
\end{lemma}

\begin{proof}[Proof of Lemma~\ref{L:quanterg} (a)]
We follow the argument in \cite{GNO1preprint}, and consider the space of shift-invariant functions
\begin{equation*}
  I(\Omega):=\{\zeta\in L^2(\Omega)\,:\,D\zeta=0\},
\end{equation*}
and note that by definition, $\prob$ is ergodic, if and only if $I(\Omega)=\R$. (Indeed, this can be seen by considering first indicator functions of measurable sets, and then appealing to the fact that the linear span of indicator functions is dense in $L^2(\Omega)$).
\smallskip

\step 1
Claim:
\begin{equation*}
 I(\Omega)=\{\zeta\in L^2(\Omega)\,:\,D^*D\zeta=0\}=\text{kernel of $D^*D$}.
\end{equation*}
The inclusion $\subset$ us trivial. Let $\zeta\in L^2(\Omega)$ satisfy $D^*D\zeta=0$. Then
\begin{equation*}
  0=\expec{\zeta D^*D\zeta}=\expec{|D\zeta|^2},
\end{equation*}
and thus $\zeta\in I(\Omega)$.
\smallskip

\step 2
Claim:
\begin{equation*}
 I(\Omega)^\perp=\overline{\{D^*F\,:\,F\in L^2(\Omega)^d\}}\qquad \text{(in $L^2(\Omega)$)}.
\end{equation*}
Since $I(\Omega)$ is closed, it suffices to prove
\begin{equation*}
  (a)\quad X:=\{D^*F\,:\,F\in L^2(\Omega)^d\}\subset I(\Omega)^\perp\qquad\text{and}\qquad(b)\quad X^\perp\subset I(\Omega).
\end{equation*}
Argument for (a): 
\begin{equation*}
  \forall F\in L^2(\Omega)^d,\zeta\in I(\Omega)\,:\,0=\expec{F\cdot D\zeta}=\expec{(D^*F)\zeta}.
\end{equation*}
Argument for (b):
Let $\zeta\in X^\perp$. Then 
\begin{equation*}
\forall F\in L^2(\Omega)^d\,:\qquad
  0=\expec{\zeta D^*F}=\expec{D\zeta\cdot F}\qquad\Rightarrow\qquad\zeta\in I(\Omega).
\end{equation*}

\step 3 (A priori estimates).

Let $\zeta\in L^2(\Omega)$ and set
$u(t):=P(t)\zeta$. Claim:
\begin{align}
  \label{eq:L:char-2:1}
  &\forall t\geq 0\,:\,\expec{|u(t)|^2}\leq\expec{\zeta^2},\\
  \label{eq:L:char-2:2}
  &\lim\limits_{t\uparrow\infty}\expec{|Du(t)|^2}=0.
\end{align}
Recall that $\partial_t u+D^*Du=0$ and $u(0)=\zeta$. Testing with
$u(t)$ and $Du(t)$ yields
\begin{eqnarray*}
  \frac{1}{2}\frac{d}{dt}\expec{u(t)^2}&=&\expec{\frac{d}{dt}u(t)\,u(t)} =-\expec{|Du(t)|^2}\leq
  0,\\
  \frac{1}{2}\frac{d}{dt}\expec{|Du(t)|^2}&=&\expec{\frac{d}{dt}Du(t)\cdot
    Du(t)}=\expec{\frac{d}{dt}u(t)\,D^*Du(t)}\\
  \notag
  &=&-\expec{|D^*Du(t)|^2}\leq 0.
\end{eqnarray*}
Integration of the first identity yields \eqref{eq:L:char-2:1} and
$\int_0^\infty\expec{|Du(t)|^2}dt\leq \expec{\zeta^2}<\infty$, and thus \eqref{eq:L:char-2:2}, since $t\mapsto \expec{|Du(t)|^2}$ is monotone (non-increasing) by the second estimate.
\smallskip

\step 4 (Conclusion).

Let $\zeta\in L^2(\Omega)$ and write $\zeta=\zeta'+\zeta''$ with $\zeta'\in I(\Omega)^\perp$ and $\zeta''\in I(\Omega)$. We claim that
\begin{equation}\label{L2:eq1}
  P(t)\zeta\to \zeta''\qquad\text{in }L^2(\Omega)\text{ as }t\to\infty.
\end{equation}
With \eqref{L2:eq1} we can conclude the proof: If $\prob$ is ergodic, then $I(\Omega)=\R$ and $\zeta''=\expec{\zeta}$. On the other hand, if $P(t)\zeta\to\expec{\zeta}$, then $\zeta''=\expec{\zeta}$. Since this is true for any $\zeta$ and $\zeta''$ is the projection onto $I(\Omega)$, we get $I(\Omega)=\R$.

Argument fo \eqref{L2:eq1}: Since $I(\Omega)$ is the kernel of $D^*D$, we have $P(t)\zeta=P(t)\zeta'+\zeta''$, and it suffices to prove $P(t)\zeta'\to 0$ for all $\zeta'\in I(\Omega)^\perp$. By Step~2 for any $\nu>0$ we can find $F\in L^2(\Omega)^d$ with  $\expec{|\zeta'-D^*F|^2}\leq\nu$. \eqref{eq:L:char-2:3} yields
\begin{equation}
  \label{eq:L:char-2:3}
  \forall t\in\R_+\qquad \expec{|P(t)(\zeta'-D^*F)|^2}\leq
  \nu.
\end{equation}
We claim that
\begin{equation}
    \label{eq:L:char-2:4}
  \lim\limits_{t\uparrow\infty}\expec{|P(t)D^*F|^2}=0.
\end{equation}
Estimate
\eqref{eq:L:char-2:4} can be seen as follows. Since the shift operators
$\tau_{e_1},\ldots,\tau_{e_d}$ commute, we get
\begin{equation*}
  P(t)D^*F=\exp(-tD^*D)D^*F=\sum_{i=1}^dD^*_i\exp(-tD^*D)F_i.
\end{equation*}
Hence,
\begin{eqnarray*}
  \expec{|P(t)D^*F|^2}&\leq& d\,\sum_{i=1}^d\expec{|D^*_iP(t)F_i|^2}\\
  &\stackrel{\text{stationarity}}{=}&
  d\,\sum_{i=1}^d\expec{|D_iP(t)F_i|^2}.
\end{eqnarray*}
Now, \eqref{eq:L:char-2:2} implies \eqref{eq:L:char-2:4}. Since $\nu>0$ arbitrary, the conclusion follows.
\end{proof}

\begin{proof}[Proof of Lemma~\ref{L:quanterg} (b)]
We follow the argument in \cite{GNO1preprint}.

\step 1
W.l.o.g. assume that $\expec{\zeta}=0$. Set $u(t):=P(t)\zeta$, and recall that
\begin{equation*}
   u(t)=\sum_{z\in\Z^d}G(t,z)\ext\zeta(z).
\end{equation*}
We have
\begin{equation*}
  \expec{u(t)}=\sum_{z\in\Z^d}G(t,z)\expec{\ext\zeta(z)}=0.
\end{equation*}
Thus, we can apply (SG) and obtain 
\begin{align}
  \label{S2:P1:1}
  \expec{u^2(t)}\leq\frac{1}{\rho}\sum_{y\in\Z^d}\expec{|\partial_yu(t)|^2}.
\end{align}
We have
\begin{eqnarray*}
  \partial_y u(t)&=&\sum_{z\in\Z^d}G(t,z)\partial_y\ext u(t,z)\\
  &=&\sum_{z\in\Z^d}G(t,z)\ext{\partial_{y-z}u}(t,z).
\end{eqnarray*}
The combination of both yields
\begin{eqnarray*}
  &&\left(\sum_{y\in\Z^d}\expec{\left(\sum_{z\in\Z^d}G(t,z)\ext{\partial_{y-z}u}(t,z)\right)^2}\right)^{\frac{1}{2}}\\
  &=&
  \left(\sum_{y\in\Z^d}\expec{\left(\sum_{x\in\Z^d}G(t,y-x)\ext{\partial_x\zeta}(y-x)\right)^2}\right)^{\frac{1}{2}}\\
  &\stackrel{\triangle\text{-inequality}\atop\text{in
      $\left(\sum_{y\in\Z^d}\expec{(\cdot)^2}\right)^{\frac{1}{2}}$}}{\leq}&
  \sum_{x\in\Z^d}\left(\sum_{y\in\Z^d}\expec{\left(G(t,y-x)\ext{\partial_x\zeta}(y-x)\right)^2}\right)^{\frac{1}{2}}\\
  &\stackrel{G\text{ is deterministic,}\atop\text{stationarity}}{=}&
  \sum_{x\in\Z^d}\left(\sum_{y\in\Z^d}G^2(t,y-x)\expec{|\partial_x\zeta|^2}\right)^{\frac{1}{2}}\\
  &=&\sum_{x\in\Z^d}\expec{|\partial_x\zeta|^2}^{\frac{1}{2}}\left(\sum_{y\in\Z^d}G^2(t,y-x)\right)^{\frac{1}{2}}.
\end{eqnarray*}
We conclude by appealing to the on-diagonal heat kernel estimate
\begin{equation*}
  \sum_y G^2(t,y)=  G(2t,0)\leq C(d)(t+1)^{-\frac{d}{2}}.
\end{equation*}
\end{proof}

The estimate in part (b) of Lemma~\ref{L:quanterg} extends to the semigroup $\exp(-D^*(a(0)D))$. The extension is non-trivial, since on the one hand, the operator $\nabla^*(a(0)\nabla)$ and $\partial_x$ do not commute, and secondly, the regularity for $\nabla^*(a(0)\nabla)$ is more involved than that for the discrete Laplacian $\nabla^*\nabla$.  In \cite{GNO1} we obtained the following decay estimate:
\begin{theorem}[see \cite{GNO1}]\label{Tquant1}
  Let $\prob$ be stationary and satisfy (SG) with constant $\rho>0$. Consider the semigroup given by
  \begin{equation*}
    P(t):=\exp(-tD^*(\aa(0)D))
  \end{equation*}
  Then for all exponents $p$ with $p_0(d,\lambda)\leq p<\infty$, all $t\geq 0$ and $F\in L^{2p}(\Omega)^d$ we have
  \begin{equation*}
    \expec{|P(t)D^*F|^{2p}}^\frac{1}{2p}\leq C(d,\lambda,\rho,p)(t+1)^{-(\frac{d}{4}+\frac{1}{2})}\sum_{y\in\Z^d}\expec{|\partial_y F|^{2p}}^{\frac{1}{2p}}.
  \end{equation*}
\end{theorem}
The proof of this theorem is out of the scope of this lecture. We only give some remarks: The exponent $\frac{d}{4}+\frac{1}{2}$ is optimal and the improvement of $\frac{1}{2}$ compared to the exponent in Lemma~\ref{L:quanterg} is due to the fact that in Theorem~\ref{Tquant1} we consider initial values in divergence form. The connection to homogenization is as follows: Set $F(\aa):=-\aa(0)e_i$ and note that
  \begin{equation*}
    \sum_{y\in\Z^d}\expec{|\partial_y F|^{2p}}^{\frac{1}{2p}}=\expec{|\partial_0F|^{2p}}^{\frac{1}{2p}}\leq \sup_{\aa,\aa'\in\Omega}|\aa_{ii}(0)-\aa_{ii}'(0)|\leq 1-\lambda.
  \end{equation*}
  Hence, $\expec{|P(t)D^*F|^{2p}}^{\frac{1}{2p}}\lesssim (t+1)^{-(\frac{d}{4}+\frac{1}{2})}$. For $d>2$, $(t+1)^{-(\frac{d}{4}+\frac{1}{2})}$ is integrable on $\R_+$, and thus
  \begin{equation*}
    \phi_i(\aa):=\int_0^\infty P(t)D^*F\,dt\in L^{2p}(\Omega),
  \end{equation*}
  is well-defined and solves
  \begin{equation*}
    D^*\aa(0)D\phi_i=D^*F,\qquad \text{i.e.\ }D^*(\aa(0)(D\phi_i+e_i))=0.
  \end{equation*}
  Now it is easy to see that the stationary extension $\overline\phi_i(a,x):=\phi_i(\tau_xa)$ is a stationary solution to the corrector equation
  \begin{equation*}
    \nabla^*(a(\nabla\overline\phi_i+e_i))=0\qquad\text{in }\Z^d,\ \prob\text{-a.s.},
  \end{equation*}
  with $\expec{|\phi_i|^{2p}}^{\frac{1}{2p}}\leq C(d,\lambda,\rho)$. We can also consider the function defined for $d\geq 2$ and $T\geq 1$  by
  \begin{equation*}
    \phi_T=\int_0^\infty\exp(-\frac{t}{T})P(t)D^*F\,dt.
  \end{equation*}
  From Theorem~\eqref{Tquant1} we then deduce that
  \begin{equation*}
    \expec{|\phi_T|^{2p}}^\frac1{2p}\leq C(d,\lambda,\rho,p)
    \begin{cases}
      \log^\frac12 T&d=2, p=1,\\
      \log T&d=2, p>1\\
      1&d\geq 3.
    \end{cases}
  \end{equation*}
  By applying $D^*(a(0)D)$ to $\phi_T$, we find that $\frac{1}{T}\phi_T+D^*(a(0)(D\phi_T+\xi))=0$, and thus the stationary extension of $\phi_T$ is the solution to the modified corrector equation.
  \medskip

  In \cite{GNO1}, based on Theorem~\ref{Tquant1} we obtained various estimates on the corrector, its periodic approximation, and on the periodic representative volume element approximation for $a_{\hom}$ in the case of independent and identically distributed coefficients. In the following section we take a slightly different approach to obtain moment bounds which does not invoke the semigroup $P_t$.

\subsection{Quantification of sublinearity in dimension $d\geq 2$}
\label{S:sublin}
In this section we prove (under a strong quantitative ergodicity assumption) that (high) moments of $\nabla\phi$ and $\nabla\sigma$ are bounded, and we quantify the growth rate of $\expec{|\phi(x)|^2}$ and $\expec{|\sigma(x)|^2}$.  The argument that we present combines the strategy of \cite{BMN17} (which relies on a Logarithmic Sobolev inequality to quantify ergodicity) and ideas of \cite{GNO4est}, where optimal growth rates for the correctors are obtained in the continuum setting and for strongly correlated coefficients. We also refer to \cite{GO16, ArmInvent} where similar esimtate (that are stronger in terms of stochastic integrability) are obtained for coefficients satisfying a finite range of dependence condition (instead of the concentration inequality that we assume). Except for some input from elliptic regularity theory (that we detail below), the argument that we present is self-contained.
\def\lip{\operatorname{lip}}
We start by introducing our quantitative ergodicity assumption on $\prob$. Instead of the absolute value of the vertical derivative $\partial_xf$, see Definition~\ref{def:SG}, we appeal to the ``Lipschitz derivative''
\begin{equation*}
  |\partial_{\lip,x}f(a)|:=\sup\Big\{|f(a')-f(a'')|\,:\,a',\tilde a''\in\Omega,\,a=a'=a''\text{ in }\Z^d\setminus\{x\}\,\Big\}.
\end{equation*}
\begin{definition}[Logarithmic Sobolev inequality (LSI)]
  We say $\prob$ satisfies (LSI) with constant $\rho>0$, if for any random variable $f$ we have
    \begin{equation*}
      \expec{f^2\log\frac{f^2}{\expec{f^2}}}\leq\frac{1}{2\rho}\sum_{x\in\Z^d}\expec{|\partial_{\lip,x} f|^2}.
    \end{equation*}
\end{definition}
The (LSI) is stronger than (SG). Indeed, (LSI) implies (SG) (with the same constant) as can be seen by expanding $f=1+\e f'$ in powers of $\e$. In the context of stochastic homogenization (LSI) has been first used in \cite{MO}; see also \cite{BMN17},  \cite{GNO4reg}, and \cite{DG} for a recent review on (LSI) and further  concentration inequalities in the context of stochastic homogenization.  
\medskip

Our main result is the following:
\begin{theorem}\label{T4}
  Suppose $\prob$ satisfies (D1) and (LSI) with constant $\rho>0$. Let $(\phi_i,\sigma_i)$ denote the extended corrector of Proposition~\ref{P:corr:discrete}. Then for all $p\geq 1$ we have
  \begin{equation*}
    \expec{|\nabla\phi|^{2p}+|\nabla\sigma|^{2p}}^{\frac{1}{2p}}\leq C(p,\rho,d,\lambda)
  \end{equation*}
  and for all $x\in\Z^d$ we have
  \begin{equation*}
    \expec{|\phi(x)|^{2p}+|\sigma(x)|^{2p}}^{\frac{1}{2p}}\leq C(p,\rho,d,\lambda)\times
    \begin{cases}
      \log^\frac12(|x|+2)&d=2,\\
      1&d\geq 3.
    \end{cases}
  \end{equation*}
\end{theorem}
Note that the estimate is uniform $x$ for $d\geq 3$. In that case we can find stationary extended correctors, i.e.\ $(\phi,\sigma)$ satisfy $(\phi,\sigma)(a,x+y)=(\phi,\sigma)(\tau_xa,y)$ instead of the anchoring condition $(\phi,\sigma)(0)=0$. In dimension $d=2$ the correctors diverge logarithmically. The logarithm (and the exponent $\frac12$) is generically optimal as can be seen by studying the limit of vanishing ellipticity contrast for independent and identically distributed coefficients. 

\begin{remark}
  Consider the two-scale expansion in Theorem~\ref{Sto:T3dis}. If we combine it Theorem~\ref{T4}, we deduce that the
  remainder $Z$ of the two-scale expansion satisfies the estimate, for
  all $p\geq 1$,
  \begin{equation*}
    \expec{\Big(\sum_{\Z^d}\alpha|Z|^2+\lambda|\nabla Z|^2\big)^p}^{\frac{1}{2p}}
    \lesssim
    \left(\alpha \sum_{x\in\Z^d}|\nabla u_0(x)|^2\omega_d(x)+\sum_{x\in\Z^d}|\nabla\nabla u_0(x)|^2\omega_d(x)\right)^{\frac12},
  \end{equation*}
  where
  \begin{equation*}
    \omega_d(x):=
    \begin{cases}
      \log(|x|+2)&d=2,\\
      1&d\geq 3,
    \end{cases}
  \end{equation*}
  and $\lesssim$ means $\leq$ up to a constant that only depends on $d,\lambda,\rho$ and $p$.  For $d\geq 3$ standard $\ell^2$-regularity shows that the right-hand
  side is bounded by $\|f\|_{\ell^2}$. Likewise, for $d=2$, weighted $\ell^2$-regularity shows that the right-hand side is estimated by $\|f\sqrt{\omega_d}\|_{\ell^2}$. Overall we obtain the estimate
  \begin{equation*}
    \expec{\Big(\sum_{\Z^d}\alpha|Z|^2+\lambda|\nabla Z|^2\big)^p}^{\frac{1}{2p}}   \lesssim
    \left(\sum_{\Z^d}|f|^2\omega_d\right)^\frac12.
  \end{equation*}
  For a comparison with Theorem~\ref{Sto:T3} we need to pass to the scaled quantities $Z_\e:\e\Z^d\to\R$, $Z_\e(x):=Z(\frac{x}{\e})$, $\nabla_{i,\e} Z_\e(x):=\e^{-1}(\nabla_iZ)(\frac{x}{\e})$, and $f_\e(x):=\e^{-2}f(\tfrac{x}{\e})$. The previous estimate than turns into
  \begin{equation*}
    \expec{\Big(\sum_{\e\Z^d}\alpha|Z_\e|^2+\lambda|\nabla Z_\e|^2\big)^p}^{\frac{1}{2p}}   \lesssim
    \left(\sum_{\Z^d}|f_\e|^2\omega_d\right)^\frac12\times\
    \begin{cases}
      \e\log^{\frac12}(\tfrac{1}{\e}+2)&d=2,\\
      \e&d\geq 3.
    \end{cases}
  \end{equation*}
  Thus, for $d=2$ we obtain a different scaling in $\e$.
\end{remark}
\medskip
A continuum version of Theorem~\ref{T4} (with optimal stochastic integrability) has been recently obtained in \cite{GNO4est}. In the discrete case the result for $d\geq 3$ is a corollary of Theorem~\ref{Tquant1}, while for $d=2$ the estimate seems to be new.
\medskip

An important ingredient in the proof of Theorem~\ref{T4} is input from elliptic regularity theory, that we recall in the following paragraph.

\paragraph{Elliptic regularity theory.} Our proof of Theorem~\ref{T4} invokes three types of input from elliptic regularity theory:
\begin{enumerate}[(a)]
\item an off-diagonal estimate for the Green's function that relies on \textit{De Giorgi-Nash-Moser theory}, see Lemma~\ref{L:green:offdiagonal};
\item a \textit{weighted Meyer's estimate} established in \cite{BMN17}, see Lemma~\ref{L:meyers} below;
\item an \textit{annealed Green's function estimate} for high moments of $|\nabla_x\nabla_y G(x,y)|$ established in \cite{MO}, see Lemma~\ref{L:annealed}.
\end{enumerate}
The proof of these estimates is beyond the scope of this lecture. 

\begin{remark}
  Estimates (a) and (b) are deterministic, in the sense that they hold for all $a\in\Omega$. Estimate (c), which invokes the expectation, has a different nature and is a first example of a \textit{large scale regularity} result for elliptic operator with stationary and ergodic coefficients. We refer to the recent work \cite{GNO4reg} where a rather complete large scale regularity theory is developed. For a another approach to large scale regularity that is based on linear mixing conditions we refer to the works by Armstrong et al., see e.g.\ \cite{ArmInvent}, the lecture notes \cite{Armstrong-lecture} and the references therein.
\end{remark}
\begin{lemma}[Green's function estimates, e.g.\ see \cite{GW82,Del99}]\label{L:green:offdiagonal}
  For any $a\in\Omega$ the Green's function (which is non-negative) satisfies
  \begin{equation*}
    G(a;x,y)\leq C(d,\lambda)
    \begin{cases}
      \log(|x|+2)&d=2,\\
      (|x|+1)^{2-d}&d>2.
    \end{cases}
  \end{equation*}
\end{lemma}
We do not present the proof of the estimate (which is classical). It can either be obtained by adapting the continuum argument in \cite{GW82}, or by integrating the heat kernel estimates in \cite{Del99}. The second ingredient from elliptic regularity theory is the following:
\begin{lemma}[weighted Meyer's estimate, see Proposition~1 in \cite{BMN17}]\label{L:meyers}
  There exists $q_0>1$ and $\alpha_0>0$ (only depending on $d$ and $\lambda$) such that for any $a\in\Omega$ and any $v:\Z^d\to\R$ and $h:\Z^d\to\R^d$ related by 
  \begin{equation*}
    \nabla^*(a\nabla v)=\nabla^*\nabla h\qquad\text{in }\Z^d,
  \end{equation*}
  the following estimates hold:
  \begin{enumerate}[(a)]
  \item For all $(q,\alpha)\in[1,q_0]\times[0,\alpha_0]$ we have
    \begin{equation}\label{weighted:est1}
      \sum_{x\in\Z^d} |\nabla v(x)|^{2q}(|x|+1)^\alpha\leq C(d,q,\alpha)\sum_{x\in\Z^d}|\nabla h(x)|^{2q}(|x|+1)^\alpha.
    \end{equation}
  \item For $1<q\leq q_01$ and $L\geq 2$ consider the weight
    \begin{equation*}
      \omega_{q,L}(x):=
      \begin{cases}
        (|x|+1)^{2(q-1)}+L^{2(1-q)}(|x|+1)^{4(q-1)}&d=2,\\
        (|x|+1)^{2d(q-1)}&d\geq 3.
      \end{cases}
  \end{equation*}
  Then we have
  \begin{equation}\label{weighted:estimate}
    \sum_{x\in\Z^d} |\nabla v(x)|^{2q}\omega_{q,L}(x)\leq C(d,q)\sum_{x\in\Z^d}|\nabla h(x)|^{2q}\omega_{q,L}(x).
  \end{equation}
\end{enumerate}
\end{lemma}
For a proof see Step 1 -- Step 3 in the proof of Lemma~4 in \cite{BMN17}. The argument relies on a weighted Calderon-Zygmund estimate for $\nabla^*\nabla$, see Proposition~1 in \cite{BMN17}. In the continuum case the estimates are classical.
Note that the weight in \eqref{weighted:estimate} satisfies
\begin{equation}\label{sublin:weight}
  \left(\sum_{x\in\Z^d}\omega_{q,L}^{-\frac{1}{q-1}}(x)\right)=C(d,q)
  \begin{cases}
    \log L&d=2,\\
    1&d\geq 3.
  \end{cases}
\end{equation}
As a corollary we obtain a weighted estimate on the mixed second derivative of the Green's function,
\begin{corollary}[weighted Green's function estimate]\label{L:green}
  There exists $q_0>1$ and $\alpha_0>0$ (only depending on $d$ and $\lambda$) such that for all $(q,\alpha)\in[1,q_0]\times[0,\alpha_0]$ we have
  \begin{eqnarray}
    \label{4:eq:G1}
    \sup_{a\in\Omega}\sum_{x\in\Z^d}|\nabla\nabla G(a;x,0)|^{2q}(|x|+1)^\alpha\leq C(q,\alpha,d,\lambda).
  \end{eqnarray}
\end{corollary}
\begin{proof}
  Note that we have
  \begin{equation*}
    \nabla_x^*(a\nabla_x\nabla_{y,i}G(a;\cdot,y))=(\nabla_i^*\delta)(\cdot - y).
  \end{equation*}
  Hence, the estimate follows from \eqref{weighted:est1}.
\end{proof}
\begin{lemma}[annealed Green's function estimate, see \cite{MO}]\label{L:annealed}
  Suppose $\prob$ satisfies (D1) and (LSI). Then for all $p\geq 1$ we have
  \begin{equation*}
    \expec{|\nabla_x\nabla_y G(a;x,y)|^{2p}}^{\frac{1}{2p}}\leq C(d,\lambda,\rho)(|x-y|+1)^{-d}.
  \end{equation*}
\end{lemma}
For a proof see \cite{MO}.

\paragraph{Sensitivity estimate and proof of Theorem~\ref{T4}}

\begin{lemma}[Sensitivity estimate]\label{L:sensi}
  Suppose $\prob$ satisfies (D1) and (D2). Then there exists $\Omega'$ with $\prob(\Omega')=1$ such that for $i=1,\ldots,d$, all $a\in\Omega'$ and all $x\in\Z^d$ we have
  \begin{eqnarray*}
    |\partial_{\lip,x}\nabla\phi_i(a,y)|&\leq& C(d,\lambda)|\nabla\nabla G(a;y,x)||\nabla\phi_i(a,x)+e_i|.
  \end{eqnarray*}
\end{lemma}
\begin{proof}[Proof of Lemma~\ref{L:sensi}]
  We define $\Omega'$ as the set of all $a\in\Omega$ such that equations \eqref{random-corr:dis1} and \eqref{random-corr:dis3} admit for $i,j,k=1,\ldots,d$, sublinearly growing (and thus unique) solutions with $\phi_i(a,0)=0$ and $\sigma_{ijk}(a,0)=0$. By Proposition~\ref{P:corr:discrete} we have $\prob(\Omega')=1$. Furthermore,  note that \eqref{random-corr:dis3} can rewritten as
  \begin{equation*}
    \nabla^*\nabla\sigma_{ijk}=\nabla^*Q_{ijk},
  \end{equation*}
  where
  \begin{equation}\label{sensi:Q}
    Q_{ijk}(a,x):=(q_{i}(a,x+e_j)\cdot e_k)e_j-(q_{i}(a,x+e_k)\cdot e_j)e_k.
  \end{equation}
  (Indeed, this follows from the identity $\nabla_iu(x)=-(\nabla_i^*u)(x+e_i)$).
  
  \step 1
  Let $a\in\Omega'$ and $a'\in\Omega'$ with $a=a'$ in $\Z^d\setminus\{x\}$. Set $    \delta a=a-a'$ and
  \begin{gather*}
    \delta\phi_i:=\phi_i(a,\cdot)-\phi_i(a',\cdot),\qquad \delta \sigma_{ijk}:=\sigma_{ijk}(a,\cdot)-\sigma_{ijk}(a',\cdot),\\
    \delta q_i:=q_i(a,\cdot)-q_i(a',\cdot),\qquad \delta Q_{ijk}:=Q_{ijk}(a,\cdot)-Q_{ijk}(a',\cdot),\\
  \end{gather*}
  Then a direct calculation (using \eqref{random-corr:dis1} -- \eqref{random-corr:dis3}, and the fact that $\delta a(y)=0$ for all $y\neq x$) yields
  \begin{eqnarray}
    \label{sensi:eq4a}
    \nabla^*(a\nabla\delta\phi_i)&=&-\nabla^*(\delta a(\nabla\phi_i(a',\cdot)+e_i)),\\
    \label{sensi:eq4c}
    \nabla^*\nabla\delta\sigma_{ijk}&=&\nabla^*\delta Q_{ijk},\\
    \label{sensi:eq4d}
    \delta q_{i}(y)&=&\delta a(y)(\nabla\phi(a,x)+e_i)+a'(y)\nabla\delta\phi_i(y)\\
    \label{sensi:eq4b}
    \delta Q_{ijk}&=&\big(\delta a(y)(\nabla\phi(a,x+e_j)+e_i)\cdot e_k\big)e_j\\\nonumber
    &&\qquad -\big(\delta a(y)(\nabla\phi(a,x+e_k)+e_i)\cdot e_j\big)e_k\\\nonumber
    &&+\big(a'(y)\big(\nabla\delta\phi_i(y+e_j)\cdot e_k\big)e_j\\\nonumber
    &&\qquad -\big(a'(y)\nabla\delta\phi_i(y+e_k)\cdot e_j\big)e_k.
  \end{eqnarray}
  Since $\delta\phi_i$ and $\delta\sigma_{ijk}$ are sublinear (as differences of sublinear functions), we may test with the Green's function and get
  \begin{eqnarray}
    \label{sensi:eq2}
    \nabla\delta\phi_i(y)&=&-\nabla_y\nabla_x G(a;y,x)\cdot \delta a(x)(\nabla\phi_i(a',x)+e_i)
  \end{eqnarray}
  Applying \eqref{sensi:eq2} with $y=x$ and the roles of $a$ and $a'$ interchanged, yields
  \begin{equation*}
    (\nabla\phi(a',x)+e_i)-(\nabla\phi(a,x)+e_i)=-\nabla_x\nabla_x G(a';x,x)(\nabla\phi_i(a,x)+e_i),
  \end{equation*}
  and thus
  \begin{eqnarray}\label{sensi:eq7a}
    |\nabla\phi(a',x)+e_i|&\leq& (|\nabla_x\nabla_xG(a';x,x)|+1)|\nabla\phi_i(a,x)+e_i|\\\nonumber
    &\leq&(\tfrac{1}{\lambda}+1)|\nabla\phi_i(a,x)+e_i|.
  \end{eqnarray}
  We conclude that
  \begin{eqnarray}\label{sensi:eq5a}
    \label{sensi:eq5b}
    |\delta\nabla\phi_i(y)|&=&C(d,\lambda)|\nabla_y\nabla_x G(a;y,x)||\nabla\phi_i(a,x)+e_i|
  \end{eqnarray}

  \step 2

  We claim that $a\in\Omega'$, $a'\in\Omega$ with $a=a'$ on $\Z^d\setminus\{x\}$ implies that $a'\in\Omega'$. In view of the definition of $\Omega'$,  we need to show existence of sublinear solutions to \eqref{random-corr:dis1} and \eqref{random-corr:dis3} for $a'$. Indeed, this can be inferred as follows: Equations \eqref{sensi:eq4a} and \eqref{sensi:eq4c} admit unique sublinear solutions $\delta\phi_i$ and $\delta\sigma_{ijk}$ with $\delta\phi_i(0)=\delta\sigma_{ijk}(0)=0$, since the right-hand side of \eqref{sensi:eq4a} is the divergence of a compactly supported function, and the right-hand side of \eqref{sensi:eq4c} is the divergence of a square summable function. Now the sought for sublinear solutions are given by $\phi_i(a',\cdot):=\phi_i(a,\cdot)+\delta\phi_i$ and $\sigma_{ijk}(a',\cdot)=\sigma_{ijk}(a,\cdot)+\delta\sigma_{ijk}$. 
As a consequence of this stability of $\Omega'$ w.r.t. compactly supported variations of $a$, when estimating $|\partial_{\lip,x}f(a)|$ for $a\in\Omega'$, we only need to take the sup (in the definition of the Lipschitz derivative) over fields $a',a''\in\Omega'$ with $a=a'=a''$ in $\Z^d\setminus\{x\}$ into account. Thus, the claimed estimate follow from \eqref{sensi:eq5a}.
\end{proof}
We combine the sensitivity estimate with the weighted Green's function estimate, Corollary~\ref{L:green}, and the following consequence of (LSI),
\begin{lemma}\label{L:LSI2}
  Let $\prob$ satisfy (LSI) with constant $\rho>0$. Then for any $1\leq p<\infty$, any $\delta>0$ and all random variables $f$ we have the estimates
  \begin{eqnarray}\label{pSG}
    \expec{|f-\expec{f}|^{2p}}^\frac{1}{2p}&\leq& C(p,\rho)\expec{\left(\sum_{x\in\Z^d}|\partial_{\lip,x}f|^2\right)^p}^{\frac{1}{2p}},\\
    \label{pLSI}
    \expec{|f|^{2p}}^{\frac{1}{2p}}&\leq& C(\delta,p,\rho)\expec{|f|^2}^{\frac12}+\delta \expec{\left(\sum_{x\in\Z^d}|\partial_{\lip,x}f|^2\right)^p}^{\frac{1}{2p}}.
  \end{eqnarray}
\end{lemma}
Estimate \eqref{pSG} for $p=1$ is the usual Spectral Gap estimate, which is implied by (LSI). \eqref{pSG} for $p>1$ follows from the estimate for $p=1$ by the argument in \cite{GNO1}. For a proof of \eqref{pLSI} we refer to \cite{MO}.
We are now in position to establish moment bounds for $\nabla\phi_i$ and $\nabla\sigma_i$:
\begin{lemma}\label{L:moment}
  Suppose $\prob$ satisfies (D1) and (LSI). Then for all $1\leq p<\infty$ 
  \begin{equation*}
    \expec{|\nabla\phi_i+e_i|^{2p}+|\nabla\sigma_i|^{2p}}^{\frac{1}{2p}}\leq C(p,\rho,d,\lambda).
  \end{equation*}
\end{lemma}
\begin{proof}
  \step 1 Proof of the bound for $\nabla\phi_i$.

  Note that we have $\expec{|\nabla\phi_i+e_i|^2}^\frac12\leq C(d,\lambda)$ by construction. Hence, in view of Lemma~\ref{L:LSI2} we only need to prove that
  \begin{equation*}
    I:=\expec{\left(\sum_{x\in\Z^d}|\partial_{\lip,x}(\nabla\phi_i(0)+e_i)|^2\right)^p}^{\frac{1}{2p}}\leq C(p,\rho,d,\lambda)\expec{|\nabla\phi_i+e_i|^{2p}}^{\frac1{2p}},
  \end{equation*}
  since then, by choosing $\delta$ sufficiently small, the right-hand side of the estimate in Lemma~\ref{L:LSI2} can be absorbed into the left-hand side. An application of Lemma~\ref{L:sensi} yields
  \begin{eqnarray*}
    I\leq C(d,\lambda)\expec{\left(\sum_{x\in\Z^d}|\nabla\nabla G(a;0,x)|^2|\nabla\phi_i(a,x)+e_i|^2\right)^p}^{\frac{1}{2p}}
  \end{eqnarray*}
  We want to estimate the right-hand side by appealing to Corollary~\ref{L:green}. To that end fix an exponent $\alpha>0$ for which the corollary applies, and suppose that $p\gg 1$ is so large, such that $\alpha(p-1)>d$ and $q:=\frac{p}{p-1}$ falls into the range of Corollary~\ref{L:green}. Then, 
  \begin{eqnarray*}
    &&\left(\sum_{x\in\Z^d}|\nabla\nabla G(a;0,x)|^2|\nabla\phi_i(a,x)+e_i|^2\right)^p\\
    &\leq& \left(\sum_{x\in\Z^d}|\nabla\nabla G(a;0,x)|^{2q}(|x|+1)^\alpha\right)^{p-1}\left(\sum_{x\in\Z^d}|\nabla\phi_i(a,x)+e_i|^{2p}(|x|+1)^{-\alpha(p-1)}(x)\right)\\
    &\leq& C(d,\lambda,\alpha, p)\left(\sum_{x\in\Z^d}|\nabla\phi_i(a,x)+e_i|^{2p}(|x|+1)^{-\alpha(p-1)}(x)\right).
  \end{eqnarray*}
  We take the expectation, exploit stationarity, and arrive at
  \begin{equation*}
    I\leq \expec{|\nabla\phi_i+e_i|^{2p}}^{\frac{1}{2p}}\left( \sum_{x\in\Z^d}(|x|+1)^{-\alpha(p-1)}(x)\right)^{\frac{1}{2p}}.
  \end{equation*}
  Since $\alpha(p-1)>d$, the claimed bound follows.
  \medskip

  \step 2 Proof of the bound for $\nabla\sigma_i$.
  
   As in the proof of Lemma~\ref{L:sensi} we write \eqref{random-corr:dis3} in the form $\nabla^*\nabla\sigma_{ijk}=\nabla^*Q_{ijk}$ with $Q_{ijk}$ defined in \eqref{sensi:Q}. Step~1 implies that the stationary random field $Q_{ijk}$ has finite $2p$th moment, and thus the ergodic theorem yields, $\prob$-a.s.
   \begin{equation*}
     \limsup\limits_{L\to\infty}\left(L^{-d}\sum_{L\Box\cap\Z^d}|Q_{ijk}|^{2p}\right)^{\frac{1}{2p}}=\expec{|Q_{ijk}|^{2p}}^{\frac{1}{2p}}\leq C(d)\expec{|\nabla\phi_i+e_i|^{2p}}^{\frac1{2p}}.
   \end{equation*}
    We claim that $\nabla\sigma_{ijk}$ inherits this property, i.e.
    \begin{equation}\label{eq:moment:1}
      \limsup\limits_{L\to\infty}\left(L^{-d}\sum_{L\Box\cap\Z^d}|\nabla\sigma_{ijk}|^{2p}\right)^{\frac{1}{2p}}\leq C(d)\expec{|\nabla\phi_i+e_i|^{2p}}^{\frac1{2p}},
    \end{equation}
    which by the ergodic theorem then yields the sought for bound on $\expec{|\nabla\sigma|^{2p}}^{\frac{1}{2p}}\leq C(p,\rho,d,\lambda)$. Estimate \eqref{eq:moment:1} can be seen as follows: For $L\gg1$ let $\eta_L$ denote a cut-off function for $L\Box$ in $2L\Box$, and let $\sigma_L$ denote the unique solution with $\sigma_L(0)=0$ and $\nabla\sigma_L\in\ell^2$ to 
    \begin{equation*}
      \nabla^*\nabla\sigma_L=\nabla^*(Q_{ijk}\eta_L).
    \end{equation*}
    Then maximal $\ell^p$ regularity for $\nabla^*\nabla$ yields
    \begin{equation*}
      \left(L^{-d}\sum_{L\Box\cap\Z^d}|\nabla\sigma_L|^{2p}\right)^{\frac{1}{2p}}\leq C(d,p)
      \left(L^{-d}\sum_{\Z^d}|Q_{ijk}\eta_L|^{2p}\right)^{\frac{1}{2p}}\leq       \left(L^{-d}\sum_{2L\Box\cap \Z^d}|Q_{ijk}|^{2p}\right)^{\frac{1}{2p}}.
    \end{equation*}
    We conclude that $\nabla\sigma_L$ weakly converges locally in $\ell^{2p}$ to $\nabla\sigma$, and thus 
    \begin{eqnarray*}
      \left(L^{-d}\sum_{L\Box\cap\Z^d}|\nabla\sigma|^{2p}\right)^{\frac{1}{2p}}&\leq& \liminf\limits_{L\to\infty}\left(L^{-d}\sum_{L\Box\cap\Z^d}|\nabla\sigma_L|^{2p}\right)^{\frac{1}{2p}}\\
      &\leq&\limsup\limits_{L\to\infty}\left(L^{-d}\sum_{2L\Box\cap\Z^d}|Q_{ijk}|^{2p}\right)^{\frac{1}{2p}}\leq C(d)\expec{|\nabla\phi_i+e_i|^{2p}}^{\frac1{2p}}.
    \end{eqnarray*}
    Passing to the limit $L\to\infty$ yields \eqref{eq:moment:1}.

\end{proof}

Now we are in position to prove Theorem~\ref{T4}:

\begin{proof}[Proof of Theorem~\ref{T4}]
  The moment bounds for $\nabla\phi$ and $\nabla\sigma$ are already proven in Lemma~\ref{L:moment}. It remains to quantify the growth of the extended corrector. Note that it suffices to prove the estimate for large $p$. We follow the idea in \cite{GNO4est}. Yet, we replace the input from large scale regularity  theory by the regularity estimates discussed above. To ease notation, fix indices $i,j,k=1,\ldots,d$, and recall the definition of $Q_{ijk}$, see \eqref{sensi:Q}. We simply write $e,\phi,\sigma,q$, and $Q$ instead of $e_i,\phi_i,\sigma_{ijk},q_i$ and $Q_{ijk}$. Furthermore, for convenience we use the notation $\int f(x)\,dx$ and $\fint_{L\Box}f(x)\,dx$ for $\sum_{x\in\Z^d}f(x)$ and $\frac{1}{\#(L\Box\cap\Z^d)}\sum_{x\in(L\Box\cap\Z^d)}f(x)$, respectively. Below $\lesssim$ denotes $\leq$ up to a constant that can be chosen only depending on $d,\lambda,\rho,\lambda$ and $p$.

  \step 1 We claim that
  for any $L\geq 2$ and $x\in\Z^d$:
  \begin{equation*}
    \expec{|(\phi,\sigma)(x)-\fint_{L\Box}(\phi,\sigma)(y+x)\,dy|^{2p}}^\frac{1}{2p}\lesssim
    \begin{cases}
      \log^\frac{1}{2}L&d=2,\\
      1&d\geq 3.
    \end{cases}
  \end{equation*}
  Since $(\nabla\phi,\nabla\sigma)$ is stationary, it suffices to prove the estimate for $x=0$. Therefore consider
  \begin{equation*}
    F(a):=(\phi,\sigma)(a,0)-\fint_{L\Box}(\phi,\sigma)(a,y)\,dy,
  \end{equation*}
  which is a random variable with vanishing expectation. Hence, in view of \eqref{pSG} it suffices to show
  \begin{equation}\label{sublin:goal1}
    \expec{\Big(\sum_{x\in\Z^d}|\partial_{\lip,x}F|^2\Big)^p}^\frac{1}{2p}\lesssim \log^\frac12 L.
  \end{equation}

  \substep{1.1} Lipschitz estimate for $F$.

  We claim that for any $a\in\Omega'$ we have
  \begin{eqnarray}
    \label{eq:lipF}
    |\partial_{\lip,x}F(a)|&\lesssim&\bigg(\big(\sum_{x':|x'-x|\leq 1}|\nabla\phi(a,x')+e|\big)\big(|\nabla v(a,x)|+|\nabla h(x)|\big)\\\nonumber
&&+\int|\nabla_y\nabla_xG(a;y,x)||\nabla\phi(a,x)+e||\nabla h(y)|\,dy\bigg),
  \end{eqnarray}
  where $h:\Z^d\to\R$ denotes the unique sublinear solution to 
  \begin{equation*}
    \nabla^*\nabla h=\delta-\frac{1}{\#(L\Box\cap\Z^d)}\mathbf{1}_{L\Box\cap\Z^d}\qquad\text{subject to }h(0)=0,
  \end{equation*}
  where $\delta$ denotes the Dirac function centered at $0$, and $\mathbf{1}_{L\Box\cap\Z^d}$ the indicator function for $L\Box\cap\Z^d$, and $v(a,\cdot)$ denotes the unique (sublinear) solution to
  \begin{equation*}
    \nabla^*(a\nabla v)=-\nabla^*\nabla h,\qquad v(0)=0.
  \end{equation*}
  For the argument, first  note that $F$ admits the representation
  \begin{equation*}
    F=\sum_{y\in\Z^d}(\nabla\phi(a,y),\nabla\sigma(a,y))\cdot (\nabla h(y),\nabla h(y)).
  \end{equation*}
  Representing $h$ with the fundamental solution to $\nabla^*\nabla$ shows that
  \begin{equation}\label{sublin:eq1}
    |\nabla h(x)|\leq C(d)(\min\{|x|+1,L\})(|x|+1)^{-d}.
  \end{equation}
  Next we would like to estimate $\partial_{\lip,x}F$. In order to do so, recall the definition of $\Omega'$ from Lemma~\ref{L:sensi}, and let $a,a'\in\Omega'$ with $a=a'$ on $\Z^d\setminus\{x\}$. Let $\delta a$, $\delta\phi$, $\delta\sigma$ and $\delta Q$ be defined by \eqref{sensi:eq4a} -- \eqref{sensi:eq4b}. 
  Then
  \begin{eqnarray*}
    F(a)-F(a')&=&\int(\nabla\delta\phi,\nabla\delta\sigma)\cdot (\nabla h,\nabla h)\\
    &=&
    \int-\nabla\delta\phi\cdot(a\nabla v)+\nabla\delta\sigma\cdot\nabla h\\
    &=&
    \int-(a\nabla\delta\phi)\cdot\nabla v+\nabla\delta\sigma\cdot\nabla h\\
    &=&
    \delta a(x)(\nabla\phi(a',x)+e)\cdot\nabla v(x)+\int\nabla\delta Q(a,y)\cdot\nabla h(y)\,dy=:I+II,
  \end{eqnarray*}
  where the last step holds thanks to equations \eqref{sensi:eq4a} and \eqref{sensi:eq4c}. By \eqref{sensi:eq7a}, the modulus of the first term is estimated by
  \begin{equation*}
    |I|\lesssim|\nabla\phi(a,x)+e||\nabla v(a,x)|.
  \end{equation*}
  Regarding $II$, from\eqref{sensi:eq4b} and \eqref{sensi:eq5b}, we deduce that
  \begin{eqnarray*}
    |II|&\lesssim&\bigg(\,\big(|\nabla\phi(a,x+e_j)+e|+    |\nabla\phi(a,x+e_k)+e|\big)|\nabla h(x)|\\
    &&
    +\int|\nabla_y\nabla_xG(a;y,x)||\nabla\phi(a,x)+e||\nabla h(y)|\,dy\bigg).
  \end{eqnarray*}
  The combination of the previous estimates yields \eqref{eq:lipF}.

  \substep{1.2} Estimate of the first term in \eqref{eq:lipF}.

  In this step we estimate the first term on the right-hand side of \eqref{eq:lipF}. Set $H(a,x):=\sum_{x':|x'-x|\leq 1}|\nabla\phi(a,x')+e|$. We claim that
  \begin{equation}\label{eq:lipF1}
    \expec{\left(\sum_{x\in\Z^d}|H(a,x)|^2\big(|\nabla v(a,x)|+|\nabla h(x)|\big)^2\right)^p}^{\frac{1}{2p}}\lesssim
    \begin{cases}
      \log^\frac{1}{2}L&d=2,\\
      1&d\geq 3.
    \end{cases}
  \end{equation}  
  For the argument we may assume that $q=\frac{p}{p-1}$ is sufficiently small, such that Lemma~\ref{L:meyers} applies. Set $\omega:=\omega_{q,L}$, see \eqref{sublin:weight}, and note that
  \begin{eqnarray*}
    &&\left(\sum_{x\in\Z^d}|H(a,x)|^2(|\nabla v(a,x)|+|\nabla h(x)|)^2\right)^p\\
    &\leq&\left(\sum_{x\in\Z^d}|H(a,x)|^{2p}\omega^{-\frac{1}{q-1}}(x)\right)\left(\sum_{x\in\Z^d}(|\nabla v(a,x)|+|\nabla h(x)|)^{2q}\omega(x)\right)^{p-1}.
  \end{eqnarray*}
  The  weighted Meyer's estimate Lemma~\ref{L:meyers} yields
  \begin{equation*}
    \sum_{x\in\Z^d}(|\nabla v(a,x)|+|\nabla h(x)|)^{2q}\omega(x)\lesssim     \sum_{x\in\Z^d}|\nabla h(x)|^{2q}\omega(x)\lesssim
    \begin{cases}
      \log L&d=2,\\
      1&d\geq 3,
    \end{cases}
  \end{equation*}
  where the last estimates follows by a direct calculation using \eqref{sublin:eq1} and the definition of $\omega=\omega_{q,L}$. On the other hand, by Lemma~\ref{L:moment} the moments of $H$ are bounded, and since $H$ is stationary, we deduce that
  \begin{equation*}
    \expec{\sum_{x\in\Z^d}\sum_{x\in\Z^d}|H(a,x)|^{2p}\omega^{-\frac{1}{q-1}}(x)}=\expec{|H|^{2p}}\sum_{\Z^d}\omega^{-\frac1{q-1}}\lesssim
    \begin{cases}
      \log L&d=2,\\
      1&d\geq 3,
    \end{cases}
  \end{equation*}
  The combination of the previous two estimates yields \eqref{eq:lipF1}.

  \substep{1.3} Estimate for the second term in \eqref{eq:lipF} and conclusion of \eqref{sublin:goal1}.

  Set $H(a,y,x):=|\nabla\nabla G(a;y,x)||\nabla\phi(a,x)+e|$. Then two applications of the triangle inequality, and Cauchy-Schwarz inequality in probability, yield
  \begin{eqnarray*}
    III&:=&\expec{\Big(\sum_{x\in\Z^d}\big(\sum_{y\in\Z^d}|H(a;y,x)||\nabla h(y)|\big)^2\Big)^p}^\frac1p\\
    &\leq&\sum_{x\in\Z^d}\expec{\big(\sum_{y\in\Z^d}|H(a;y,x)||\nabla h(y)|\big)^{2p}}^{\frac1p}\\
    &=&\sum_{x\in\Z^d}\expec{\big(\sum_{y,y'\in\Z^d}|H(a;y,x)||H(a;y',x)||\nabla h(y)||\nabla h(y')|\big)^{p}}^\frac1p\\
    &\leq&\sum_{x,y,y'\in\Z^d}|\nabla h(y)||\nabla h(y')|\expec{|H(a;y,x)|^{2p}}^{\frac{1}{2p}}\expec{|H(a;y',x)|^{2p}}^{\frac1{2p}}\\
    &\leq&\sum_{x\in\Z^d}\big(\sum_{y\in\Z^d}|\nabla h(y)|\expec{|H(a;y,x)|^{2p}}^{\frac{1}{p}}\big)^2\\
    &=&\sum_{x\in\Z^d}\big(\sum_{y\in\Z^d}|\nabla h(y)|\expec{|H(a;y-x,0)|^{2p}}^{\frac{1}{p}}\big)^2,
  \end{eqnarray*}
  where the last identity holds thanks to the identity
  \begin{equation}\label{green:stat}
    H(a,y,x)=H(\tau_x,y,0).
  \end{equation}
  and stationarity of $\prob$. In view of the moment bounds on $\nabla\phi$, see Lemma~\ref{L:moment}, and the annealed Green's function estimate, see Lemma~\ref{L:annealed}, we have
  \begin{equation*}
    \expec{|H(a;y-x,0)|^{2p}}^\frac{1}{p}\lesssim (|x-y|+1)^{-2d},
  \end{equation*}
  and thus we arrive at
  \begin{equation*}
    III\lesssim\sum_{x\in\Z^d}\big(\sum_{y\in\Z^d}|\nabla h(y)|(|x-y|+1)^{-2d}\big)^2\lesssim \|\nabla h\|_{\ell^2}^2\lesssim
    \begin{cases}
      \log L&d=2,\\
      1&d\geq 3,
    \end{cases}
  \end{equation*}
  where the last two estimates hold due to Young's convolution estimate and a direct calculation that uses \eqref{sublin:eq1}. Combined with \eqref{eq:lipF1} and \eqref{eq:lipF}, we eventually get \eqref{sublin:goal1}.

  \step 2 We claim that for any $L\geq 2$ any $x\in\Z^d$ we have
  \begin{equation}\label{sublin:goal2}
    \expec{|\fint_{L\Box}(\nabla\phi,\nabla\sigma)(x+y)\,dy|^{2p}}^\frac{1}{2p}\lesssim
    \begin{cases}
      L^{-1}\log^{\frac{1}{2p}} L&d=2,\\
      L^{-\frac{d}{2}\frac{p-1}{p}}&d\geq 3.
    \end{cases}
  \end{equation}
  Since $(\nabla\phi,\nabla\sigma)$ is stationary, it suffices to prove the estimate for $x=0$. Therefore consider
  \begin{equation*}
    F'(a):=\int(\nabla\phi,\nabla\sigma)(a,y)\cdot m_L(y)\,dy,\qquad m_L:=\frac{1}{\#(L\Box\cap\Z^d)}\mathbf{1}_{L\Box\cap\Z^d}e_0,
  \end{equation*}
  where $e_0$ denotes an arbitrary unit vector in $\R^{d}\times\R^d$. It suffices to show that $\expec{|F'|^{2p}}^\frac{1}{2p}$ is bounded by the right-hand side of \eqref{sublin:goal2}. Since the expectation of $F'$ is zero, by \eqref{pSG} we only need to show that
  \begin{equation}\label{sublin:goal3}
    \expec{\Big(\sum_{x\in\Z^d}|\partial_{\lip,x}F'|^2\Big)^p}^\frac{1}{2p}\lesssim    \begin{cases}
      L^{-1}\log^{\frac{1}{2p}} L&d=2,\\
      L^{-\frac{d}{2}\frac{p-1}{p}}&d\geq 3.
    \end{cases}
  \end{equation}
  Following the argument in Substep 1.1 (with $\nabla h$ replaced by $m_L$) we obtain the estimate
  \begin{eqnarray}
    \label{eq:lipF'}
    |\partial_{\lip,x}F'(a)|&\lesssim&\bigg(\big(\sum_{x':|x'-x|\leq 1}|\nabla\phi(a,x')+e|\big)\big(|\nabla v(a,x)|+|m_L(x)|\big)\\\nonumber
                            &&+\int|\nabla_y\nabla_xG(a;y,x)||\nabla\phi(a,x)+e||m_L(y)|\,dy\bigg)\\
    &=:&I(x)+II(x),
  \end{eqnarray}
  where $v$ denotes the unique sublinear solution to 
  \begin{equation*}
    \nabla^*(a\nabla v)=-\nabla^*m_L,\qquad v(0)=0.
  \end{equation*}
  In order to get \eqref{sublin:goal3}, suppose that  $p\gg 1$ is so large, such that Lemma~\ref{L:meyers} applies with $q:=\frac{p}{p-1}$.
  Then, with $H(a,x):=\sum_{x':|x'-x|\leq 1}|\nabla\phi(a,x')+e|\big)$, we get
  \begin{eqnarray*}
    \expec{(\sum_{x\in\Z^d}|I(x)|^2)^p}^\frac{1}{2p}
    &\leq&
    \expec{(\sum_{x\in\Z^d}|H(a,x)|^2|\nabla v(x)|^2)^p}^\frac{1}{2p}+\expec{(\sum_{x\in\Z^d}|H(a,x)|^2m_L(x)^2)^p}^\frac{1}{2p}.
  \end{eqnarray*}
  By Jensen's inequality, Lemma~\ref{L:moment} and the definition of $m_L$, the second term is bounded by $L^{-\frac{d}{2}}$, while for the first term we appeal to H\"older's inequality and Lemma~\ref{L:meyers}. As in Substep 1.2 we get
  \begin{eqnarray*}
    \expec{(\sum_{x\in\Z^d}|H(a,x)|^2|\nabla v(x)|^2)^p}^\frac{1}{2p}&\lesssim&
    \left(\sum_{x\in\Z^d}m_L^{2q}\omega_{q,L}\right)^{\frac{1}{2q}}\times
    \begin{cases}
      \log^{\frac{1}{2p}} L&d=2,\\
      1&d\geq 3.
    \end{cases}\\
    &\lesssim&
               \begin{cases}
                 L^{-1}\log^{\frac{1}{2p}} L&d=2,\\
                 L^{-\frac{d}{2q}}&d\geq 3.
    \end{cases}
  \end{eqnarray*}
  Likewise, the estimate in Substep 1.3, with $\nabla h$ replaced by $m_L$ yields
  \begin{equation*}
    \expec{(\sum_{x\in\Z^d}|II(x)|^2)^p}^{\frac1{2p}}\lesssim \|m_L\|_{\ell^2}\leq L^{-\frac{d}{2}}.
  \end{equation*}
  The combination of the previous estimates yields \eqref{sublin:goal3}, and thus \eqref{sublin:goal2}.

  \step 3
 
  We claim that for any $L\geq 2$ and $x\in\Z^d$ we have for all $p\geq p_0$ (only depending on $d$ and $\lambda$),
  \begin{equation*}
    \expec{|\fint_{L\Box}(\phi,\sigma)(x+y)-(\phi,\sigma)(y)\,dy|^{2p}}^{\frac{1}{2p}}\lesssim\frac{|x|}{L}
        \begin{cases}
          \log^{\frac{1}{2p}}L&d=2,\\
          1&d\geq 3.
        \end{cases}
  \end{equation*}
For the argument note that there exists a path $\Gamma\subset\Z^d$ with $|\Gamma|:=\#\Gamma\lesssim |x|$
and $e:\Gamma\to\{\pm e_1,\ldots,\pm e_d\}$ s.t. for any
$u:\Z^d\to\R$ we have
  \begin{equation*}
    u(x_0)-u(x)=\sum_{p\in\Gamma}\nabla u(p)\cdot e(p).
  \end{equation*}
  Hence,
  \begin{equation*}
    F''(a):=\fint_{L\Box}(\phi,\sigma)(x+y)-(\phi,\sigma)(y)\,dy=\sum_{p\in\Gamma}\fint_{L\Box}(\nabla\phi,\nabla\sigma)(a,y+p)\cdot (e(p),e(p))\,dy,
  \end{equation*}
  and thus by the triangle inequality and stationarity of $(\nabla\phi,\nabla\sigma)$, and the estimate of Step~2,
  \begin{eqnarray*}
    \expec{|F''(a)|^{2p}}^{\frac{1}{2p}}&\leq& \sum_{p\in\Gamma}\expec{|\fint_{L\Box}(\nabla\phi,\nabla\sigma)(a,y+p))\cdot (e(p),e(p))|^{2p}}^\frac{1}{2p}\\
    &=& \sum_{p\in\Gamma}\expec{|\fint_{L\Box}(\nabla\phi,\nabla\sigma)(a,y)|^{2p}}^{\frac{1}{2p}}\\
    &\lesssim&\frac{|x|}{L}
        \begin{cases}
          \log^{\frac{1}{2p}}L&d=2,\\
          1&d\geq 3,
        \end{cases}
  \end{eqnarray*}
  where in the last step we assumed that $p$ is so large, such that $L^{-\frac{d}{2}\frac{p-1}{p}}\leq L^{-1}$ for $d\geq 3$.

  \step 4 Conclusion.

  Choose $L=|x|+2$. Then by the estimate in Step 1 and in Step 3,
  \begin{eqnarray*}
    &&\expec{|(\phi,\sigma)(x)|^{2p}}^{\frac{1}{2p}}=\expec{|(\phi,\sigma)(x)-(\phi,\sigma)(0)|^{2p}}^{\frac{1}{2p}}\\
    &\leq&
           \expec{|(\phi,\sigma)(x)-\fint_{L\Box}(\phi,\sigma)(y+x)\,dy|^{2p}}^\frac{1}{2p}\\
    &&+
    \expec{|\fint_{L\Box}(\phi,\sigma)(y+x)-(\phi,\sigma)(y)\,dy|^{2p}}^\frac{1}{2p}\\
    &&+
           \expec{|(\phi,\sigma)(0)-\fint_{L\Box}(\phi,\sigma)(y)\,dy|^{2p}}^\frac{1}{2p}\\
    &\lesssim&
    \begin{cases}
      \log^\frac12(|x|+2)&d=2,\\
      1&d\geq 3.
    \end{cases}
  \end{eqnarray*}
  \end{proof}


\appendix
\section{Solutions to Problem~\ref{Int:P1} -- \ref{Int:P6}}
\label{appendix-solution}
\begin{proof}[Proof of Problem~\ref{Int:P1}]
  For simplicity we set $a_\e:=a\left(\tfrac{\cdot}{\e}\right)$. By the fundamental theorem of calculus we have
  \begin{equation*}
    u_\e(x)-u_\e(0)=\int_0^x\partial_x u_\e\left(x'\right)\,\dd x'=\int_0^xa^{-1}_\e\left(x'\right)j_\e\left(x'\right)\,\dd x',
  \end{equation*}
  where $j_\e$ denotes the flux
  \begin{equation*}
    j_\e(x):=a\left(\tfrac{x}{\e}\right)\partial_x u_\e(x).
  \end{equation*}
  From \eqref{Int:Eq1} we learn that
  \begin{equation*}
    j_\e(x)=c_\e-\int_0^xf\left(x'\right)\,\dd x'
  \end{equation*}
  for a constant $c_\e\in\R$, which is uniquely determined by \eqref{Int:Eq2}: Indeed, we have
  \begin{align*}
    0=u_\e(L)-u_\e(0)&=\int_0^La^{-1}_\e\left(x'\right)j_\e\left(x'\right)\,\dd x'\\
                     &=\int_0^La^{-1}_\e\left(x'\right)\left(c_\e-\int_0^{x'}f\left(x''\right)\,\dd x''\right)\,\dd x',
  \end{align*}
  and thus,
  \begin{equation*}
    c_\e=\left(\int_0^La^{-1}_\e\left(x'\right)\,\dd x'\right)^{-1}\int_0^L\int_0^{x'}a^{-1}_\e\left(x'\right)f\left(x''\right)\,\dd x''\,\dd x'.
  \end{equation*}
  Since $u_\e(L)=0$, we get the representation, \eqref{Int:Eq5}, i.e.
  \begin{equation*}
    u_\e(x)=\int_0^xa^{-1}_\e\left(x'\right)\left(c_\e-\int_0^{x'}f\left(x''\right)\,\dd x''\right)\,\dd x'.
  \end{equation*}
  Since $f$ and $a_\e$ are smooth (by assumption), the right-hand side defines a smooth solution to \eqref{Int:Eq1} and \eqref{Int:Eq2}.
\end{proof}

\begin{proof}[Proof of Problem~\ref{Int:P3}]
  Application of Problem~\ref{Int:P2} yields
  \begin{align*}
    \int_0^La_\e^{-1}&=a_0^{-1}L+O(\e),\text{ where }a_0=\left(\int_0^1a^{-1}\right)^{-1}\\
    c_\e&= c_0+O(\e)\text{ where }c_0:=a_0\fint_0^La_0^{-1}\int_0^{x'}f\left(x''\right)\,\dd x''\dd x'=\fint_0^L\int_0^{x'}f\left(x''\right)\,\dd x''\dd x',\\
    u_\e(x)&\to u_0(x)+O(\e)\text{ where }u_0(x):=a_0^{-1}\int_0^x\left(c_0-\int_0^{x'}f(x'')\dd x''\right)\dd x'.
  \end{align*}
  Finally, it is easy to check that $u_0$ is smooth and solves \eqref{Int:Eq3} and \eqref{Int:Eq4}. 
\end{proof}

\begin{proof}[Proof of Problem~\ref{Int:P4}]
    Recall that $u$ admits the representation
    \begin{equation*}
      u(x)=\int_0^xa^{-1}\left(x'\right)\left(c-x'\right)\,\dd x'.
    \end{equation*}
    for some $c\in\R$. Hence, $u'(x)=\tfrac{c-x}{a(x)}$ and thus
    \begin{equation*}
      u\text{ is quadratic }\Leftrightarrow u'\text{ is affine }\Leftrightarrow a(\cdot)\text{ is a constant.}
    \end{equation*}
  \end{proof}

  \begin{proof}[Proof of Problem~\ref{Int:P5}]
    We first notice that $M_0:=\max_{\bar O}u_0=\frac{1}{8a_0}$. Indeed, this follows from
    \begin{equation*}
      u_0(x)=a_0^{-1}\int_0^x\left(\tfrac{1}{2}-x'\right)\,\dd x'.
    \end{equation*}
    We conclude by appealing to the quantitative homogenization result $\max_{\bar O}|u_\e-u_0|=O(\e)$: 
    \begin{equation*}
      M_\e\geq u_\e\left(\tfrac{1}{2}\right)=u_0\left(\tfrac{1}{2}\right)+O(\e)=M_0+O(\e),
    \end{equation*}
    and for some $x_\e$ we have
    \begin{equation*}
      M_\e=u_\e(x_\e)=u_0(x_\e)+O(\e)\leq M_0+O(\e).
    \end{equation*}
    Hence, $M_\e=M_0+O(\e)$.
  \end{proof}

  \begin{proof}[Proof of Problem~\ref{Int:P6}]
    We argue by contradiction and assume that (for a subsequence)
    \begin{equation*}
      \int_O |\partial_x u_\e-\partial_x u_0|^2\to 0,
    \end{equation*}
    which implies that $\partial_x u_\e(x)\to \partial_x u_0(x)$ for a.e. $x\in O$ for a subsequence. The representation formula and a direct computation shows that
    \begin{equation*}
      \partial_xu_\e(x)=\frac{(c_\e-x)}{a_\e(x)}\qquad \partial_xu_0(x)=\frac{(c_0-x)}{a_0}
    \end{equation*}
    Since $c_\e\to c_0$ (as shown in the proof of Problem~\ref{Int:P3}), we deduce that $\frac{1}{a_\e(x)}\to\frac{1}{a_0}$ for a.e. $x\in O$. Combined with the dominated convergence theorem, we conclude that $\tfrac{1}{a_\e}\to \tfrac{1}{a_0}$ in $L^2(O)$, and thus $\int_Oa_\e\to \int_O a_0$. However, by Problem~\ref{Int:P2} we have
    \begin{equation*}
      \int_Oa_\e \to \int_O\int_0^1a \neq \int_O a_0 \qquad\text{unless $a$ is a constant function.}
    \end{equation*}
    The second statement is a direct consequence of an integration by parts and Problem~\ref{Int:P3}.
  \end{proof}


\end{document}